\documentclass[11pt,english]{smfart}  
\usepackage{smfthm}
  
\usepackage{amssymb,url,xspace,amsthm}
\usepackage{mathrsfs} 
\usepackage{paralist}
\usepackage{datetime} 
\usepackage{colortbl}
\usepackage[dvipsnames]{xcolor}
\usepackage{mathtools}
\usepackage{euscript}
\usepackage[all]{xy}
\usepackage{graphicx} 

\usepackage{hyperref}
\usepackage[T1]{fontenc}
\usepackage{newtxtext}
\usepackage{newtxmath}
\usepackage{textcomp}
\usepackage[text={155mm,250mm},centering]{geometry} 
\usepackage{comment}  
\usepackage{cite}  
\usepackage{thmtools}
\usepackage{enumitem}
\usepackage{letltxmacro} 
\usepackage{nameref}
\usepackage{cleveref}
\usepackage{calc}
\usepackage{interval}
\usepackage{manfnt}
\usepackage{tikz-cd}  
\usepackage[hyperpageref]{backref}

\usepackage{faktor} 

 \setlist[itemize]{wide = 0pt, labelwidth = 2em, labelsep*=0em, itemindent = 0pt, leftmargin = \dimexpr\labelwidth + \labelsep\relax, noitemsep,topsep = 1ex,}
  \setlist[enumerate]{wide = 0pt, labelwidth = 2em, labelsep*=0em, itemindent = 0pt, leftmargin = \dimexpr\labelwidth + \labelsep\relax, noitemsep,topsep = 1ex}

\theoremstyle{plain}
\newtheorem{thmx}{Theorem} 
\renewcommand{\thethmx}{\Alph{thmx}} 
\newtheorem{thm}{Theorem}[section]  
\newtheorem{lem}[thm]{Lemma}
\newtheorem{claim}[thm]{Claim}

\newtheorem{proposition}[thm]{Proposition}

\theoremstyle{definition}
\newtheorem{dfn}[thm]{Definition}

\theoremstyle{remark}
\newtheorem{rem}[thm]{Remark} 
\newtheorem{example}[thm]{Example}

\numberwithin{equation}{section}  

\theoremstyle{plain}
\newlist{thmlist}{enumerate}{1}
\setlist[thmlist]{wide = 0pt, labelwidth = 2em, labelsep*=0em, itemindent = 0pt, leftmargin = \dimexpr\labelwidth + \labelsep\relax, noitemsep,topsep = 1ex, font=\normalfont, label=(\roman*), ref=\thethm.(\roman{thmlisti})}

\addtotheorempostheadhook[thm]{\crefalias{thmlisti}{thm}}

\addtotheorempostheadhook[assumpsion]{\crefalias{thmlisti}{assumption}}

\addtotheorempostheadhook[cor]{\crefalias{thmlisti}{cor}}

\addtotheorempostheadhook[proposition]{\crefalias{thmlisti}{proposition}}

\addtotheorempostheadhook[dfn]{\crefalias{thmlisti}{dfn}}

\addtotheorempostheadhook[lem]{\crefalias{thmlisti}{lem}}
\addtotheorempostheadhook[main]{\crefalias{thmlisti}{main}}

\addtotheorempostheadhook[rem]{\crefalias{thmlisti}{rem}}

\newlist{thmenum}{enumerate}{1} 
\setlist[thmenum]{wide = 0pt, labelwidth = 2em, labelsep*=0em, itemindent = 0pt, leftmargin = \dimexpr\labelwidth + \labelsep\relax, noitemsep,topsep = 1ex, font=\normalfont, label=(\roman*), ref=\thethmx.(\roman{thmenumi})}
\crefalias{thmenumi}{thmx} 

\newlist{corlist}{enumerate}{1} 
\setlist[corlist]{wide = 0pt, labelwidth = 2em, labelsep*=0em, itemindent = 0pt, leftmargin = \dimexpr\labelwidth + \labelsep\relax, noitemsep,topsep = 1ex, font=\normalfont, label=(\roman*), ref=\thecorx.(\roman{corlisti})}
\crefalias{corlisti}{corx} 
 
\newenvironment{(i)}
{{\it Proof of (i).}}
{\hfill $\Box$ \\}

 \newenvironment{(ii)}
{{\it Proof of (ii).}}
{\hfill $\Box$ \\}

\newenvironment{(iii)}
{{\it Proof of (iii).}}
{\hfill $\Box$ \\}

\crefname{lem}{Lemma}{Lemmas} 
\crefname{conjecture}{Conjecture}{Conjectures}
\crefname{thm}{Theorem}{Theorems}
\crefname{proposition}{Proposition}{Propositions}
\crefname{dfn}{Definition}{Definitions}
\crefname{rem}{Remark}{Remarks}
\crefname{cor}{Corollary}{Corollaries}
\crefname{corx}{Corollary}{Corollaries}
\crefname{problem}{Problem}{Problems}
\crefname{thmx}{Theorem}{Theorems}
\crefname{claim}{Claim}{Claims}
\crefname{assumption}{Assumption}{Assumptions}
\crefname{main}{Main Theorem}{Main Theorems}

\def\ep{\varepsilon}

\def\Res{{\rm Res}}

\def\codim{{\rm codim}}

\def\rank{{\rm rank}\,}

\newcommand{\cS}{\mathcal{S}}

\DeclareMathOperator{\GL}{GL}

\makeatletter
\newcommand*{\rom}[1]{\expandafter\@slowromancap\romannumeral #1@}
\makeatother

\newcommand{\sD}{\mathscr{D}}

\newcommand{\sO}{\mathscr{O}}

\newcommand{\sR}{\mathscr{R}}

\newcommand{\cC}{\mathcal C}

\newcommand{\cR}{\mathcal R}

\newcommand{\bC}{\mathbb{C}}
\newcommand{\bD}{\mathbb{D}}

\newcommand{\bF}{\mathbb{F}}

\newcommand{\bK}{\mathbb{K}}

\newcommand{\bN}{\mathbb{N}}

\newcommand{\bP}{\mathbb{P}}
\newcommand{\bQ}{\mathbb{Q}}
\newcommand{\bR}{\mathbb{R}}

\newcommand{\bU}{\mathbb{U}}

\newcommand{\bZ}{\mathbb{Z}}

\newcommand{\kp}{\mathfrak{p}}

\def\db{\overline{\partial}}

  \def\spec{\textrm{Spec}\,}
 \def\d{\partial}

\def\End{{\rm \small  End}}

\def\vol{\mathrm{vol}}

\newcommand{\diae}{{}^\diamond\! E}

\newcommand{\Sym}{{\rm Sym}}
\begin{document} 
 	\title[Pluriharmonic maps into Euclidean buildings and symmetric differentials]{Pluriharmonic maps into Euclidean buildings\\ and symmetric differentials} 
 \alttitle{Applications pluri-harmoniques à valeurs dans un immeuble euclidien et formes différentielles symmétriques}
 \dedicatory{In memory of Jean-Pierre Demailly}
	\author[D. Brotbek]{Damian Brotbek}  
\address{Institut \'Elie Cartan de Lorraine, Universit\'e de Lorraine, F-54000 Nancy,
	France.}
\email[Damian Brotbek]{damian.brotbek@univ-lorraine.fr}
\author[G. Daskalopoulos]{Georgios Daskalopoulos}
\address{Department of Mathematics, Brown Univeristy, Providence, RI}
\email[Georgios Daskalopoulos]{daskal@math.brown.edu}
\urladdr{https://www.math.brown.edu/gdaskalo/}

\author[Y. Deng]{Ya Deng}   
\address{CNRS, Institut \'Elie Cartan de Lorraine, Universit\'e de Lorraine, F-54000 Nancy,
	France.}
\email[Ya Deng]{ya.deng@math.cnrs.fr} 
\urladdr{https://ydeng.perso.math.cnrs.fr} 

\author[C. Mese]{Chikako Mese}
\address{Johns Hopkins University, Department of Mathematics, Baltimore, MD}
\email[Chikako Mese]{cmese@math.jhu.edu} 
\urladdr{https://sites.google.com/view/chikaswebpage/home}

\keywords{Logarithmic symmetric differentials, (pluri-)harmonic map, Bruhat-Tits  buildings,  rigid representation,  Simpson's integrality conjecture, variation of Hodge structure}
\subjclass{53C43, 	14F35}
\begin{abstract}  
Given a  complex smooth quasi-projective variety $X$, a semisimple algebraic group $G$ defined over some non-archimedean local field $K$ and  a Zariski dense   representation   $\varrho:\pi_1(X)\to G(K)$, we construct  a $\varrho$-equivariant (pluri-)harmonic map from the universal cover  of $X$ into the Bruhat-Tits building $\Delta(G)$ of $G$, with some suitable asymptotic behavior. This theorem  generalizes the previous work by Gromov-Schoen to the quasi-projective setting.    

As an application, we   prove that   $X$ has nonzero global logarithmic symmetric differentials if there exists a linear representation $\pi_1(X)\to {\rm GL}_N(\mathbb{K})$ with infinite image, where $ \mathbb{K}$ is any field.   This theorem generalizes the previous work by Brunebarbe, Klingler and Totaro to  the quasi-projective setting.  
\end{abstract}   
	
	\maketitle
\tableofcontents	
	\section{Introduction}\label{sec:introduction}
  \subsection{Main theorem} 
  Let $X$ be a complex smooth quasi-projective variety, and let $G$ be a semisimple algebraic group defined over a field $K$. In this paper, we mainly focus   on representations $\varrho:\pi_1(X) \to G(K)$, where $K$ can be the field of complex numbers, a number field, or a non-archimedean local field. We refer to such a representation $\varrho$ as \emph{Zariski dense} if the Zariski closure of its image is $G$.
  
  In the archimedean setting, i.e., when $K$ is the field of complex numbers, Donaldson, Corlette, and Labourie established the existence of $\varrho$-equivariant harmonic maps to symmetric spaces when $X$ is a compact K\"ahler manifold (cf. \cite{Don87,Cor88,Lab91}). Mochizuki extended this result to the quasi-projective case, proving the existence of $\varrho$-equivariant pluriharmonic maps in \cite{Moc07b}.
  
  In the non-archimedean setting, i.e., when $K$ is a non-archimedean local field, Gromov and Schoen proved  the existence of $\varrho$-equivariant pluriharmonic maps to the Bruhat-Tits building of $G$ when $X$ is a compact K\"ahler manifold  (cf. \cite{GS92}). However, extending their result to quasi-projective varieties has remained a significant open problem for the past three decades. A series of works by the second and fourth authors \cite{DM21,DMunique,DMks,DMrs} have made  progress in extending the Gromov-Schoen theory to the quasi-projective setting.
  
  The main goal  of this paper is to complete the generalization of Gromov-Schoen's theorem to the quasi-projective setting. 
 Our main theorem is as follows.
    \begin{thmx}[=\cref{thm:harmonicmaps,logestimate,thm:functorial}] \label{GS}
    		Let $X$ be a complex smooth quasi-projective variety, and let $G$ be a semisimple algebraic group defined over a non-archimedean local field $K$. Denote by $\widetilde{X}$ the universal cover of $X$. If $\varrho: \pi_1(X) \to G(K)$ is a Zariski-dense representation, then there exists a $\varrho$-equivariant,  pluriharmonic map $\tilde{u}: \widetilde{X} \to \Delta(G)$ to the Bruhat-Tits building $\Delta(G)$ of $G$, such that the following properties hold:
    		\begin{thmenum}
    			\item \label{main:existence}  the map $\tilde{u}$ is locally Lipschitz, and has logarithmic energy growth (cf. \cref{def:log energy}). 
    			\item \label{main:harmonic}
    		the map $\tilde{u}$ is  harmonic  with respect to any K\"ahler metric on $\widetilde{X}$. 
    			\item \label{main:energy}Let $\overline{X}$ be a smooth projective compactification of $X$, such that $\Sigma := \overline{X} \backslash X$ is a simple normal crossing divisor. For any \emph{smooth} point $x$ of $\Sigma$, if the local monodromy of $\varrho$ around the irreducible component of $\Sigma$ containing $x$ is quasi-unipotent, then there exists an open neighborhood $\Omega_x$ of $x$ in $\overline{X}$ such that the energy $E^{\tilde{u}}[\Omega_x \backslash\Sigma]$ of $\tilde{u}$ on $\Omega_x \backslash \Sigma$ is finite (cf. \eqref{eq:defenergy} and \eqref{eq:defenergy2} for the definition of energy). 
    			\item \label{main:functorial} Let $f: Y \to X$ be a morphism from a smooth quasi-projective variety $Y$. Denote by $\tilde{f}: \widetilde{Y} \to \widetilde{X}$ the lift of $f$ between the universal covers of $Y$ and $X$.  Then the $f^*\varrho$-equivariant map $\tilde{u }\circ \tilde{f}: \widetilde{Y} \to \Delta(G)$ is   pluriharmonic and has logarithmic energy growth. 
    		\end{thmenum} 
    	\end{thmx}
    	\subsection{An application}
Esnault asked   whether a smooth projective variety with an infinite fundamental group has non-trivial symmetric differentials. This was confirmed by Brunebarbe, Klingler, and Totaro \cite[Theorem 0.1]{BKT13} in the linear case, when $X$ is a compact Kähler manifold.  
    	\begin{thm}[\cite{BKT13}]\label{thm:BKT}
    		Let $X$ be a compact K\"ahler manifold. If there is a linear representation $\varrho:\pi_1(X)\to \GL_{N}(\bK)$ with $\bK$ being any field such that $\varrho(\pi_1(X))$ is infinite, then $H^0(X,{\rm Sym}^k\Omega_X)\neq 0$ for some positive integer $k$. 
    	\end{thm} 
Building on  ideas from previous works \cite{Kat97,Zuo96,Eys04,Kli13,BKT13} and using \cref{GS}, we extend \cref{thm:BKT} to the quasi-projective setting. 
\begin{thmx}\label{main}
	Let $X$ be a smooth quasi-projective variety, and let $\tau:\pi_1(X)\to\GL_N(\bK)$  be a linear representation where $\bK$ is  any field.  Let $\overline{X}$ be a  smooth projective compactification  of $X$ such that $\Sigma:=\overline{X}\backslash X$ is a simple normal crossing divisor. If the image of $\tau$ is an infinite group, then   there is a positive integer $k$ such that   $H^0(\overline{X},{\rm Sym}^k\Omega_{\overline{X}}(\log \Sigma))\neq 0$.   
\end{thmx}

Let us mention that \cref{GS} has further applications in other areas. For more recent developments, we refer readers to \cite{CDY22,DY23,DY24,DM24}.

\subsection{Notation and Convention}\label{sec:notation}
\begin{enumerate}[label=(\arabic*)]
	\item Unless otherwise specified,  algebraic varieties are assumed to be connected and   defined over the field of complex 
	numbers.
	\item   A \emph{log smooth pair} $(\overline{X},\Sigma)$ consists of a   smooth projective variety $\overline{X}$ and a simple normal crossing divisor $\Sigma$ on $\overline{X}$.  We  denote by $X:=\overline{X}\backslash \Sigma$, and    $\pi_X:\widetilde{X}\to X$ the universal cover map. 
	\item Let $\overline{X}$ be a smooth projective variety. A line bundle $L$ on $\overline{X}$ is \emph{sufficiently
		 ample} if there exists a projective embedding $\iota:\overline{X}\hookrightarrow \bP^N$  such that $L=\iota^*\sO_{\bP^N}(d)$ for some $d\geqslant 3$.
	\item A linear representation $\varrho:\pi_1(X)\to \GL_N(K)$ with $K$ some field is called \emph{reductive} if the Zariski closure of $\varrho(\pi_1(X))$ is a reductive algebraic group over $\overline{K}$.
	
	  If
	  $Y$ is a closed smooth subvariety of $X$, we denote by $\varrho_Y:\pi_1(Y)\to G(K)$    the composition of the natural homomorphism $\pi_1(Y)\to \pi_1(X)$ and $\varrho$.  
	\item Denote by $\bD$ the unit disk in $\bC$, and by $\bD^*$ the punctured unit disk. We write $\bD_r := \{ z \in \bC \mid |z| < r \}$, $\bD^*_r := \{ z \in \bC \mid 0 < |z| < r \}$, and $\bD_{r_1,r_2} := \{ z \in \bC \mid r_1 < |z| < r_2 \}$. 
\end{enumerate}

\subsection*{Acknowledgment} 
 The paper is long as   it incorporates  the section on harmonic map theory into Euclidean buildings in the (unpublished) work of the second and fourth authors \cite{DM21}.  We would like to thank Michel Brion, H\'el\`ene Esnault, Auguste H\'ebert,  Nicolas Monod, Guy Rousseau for useful discussions and comments.   
Damian Brotbek is supported by the grant \emph{Lorraine Université d'Excellence - Future Leader}.  Georgios Daskalopoulos is  supported in part by NSF DMS-2105226. Ya Deng  is  supported by the ANR grant  \emph{Karmapolis} (ANR-21-CE40-0010).    Chikako Mese is  supported in part by NSF DMS-2005406  and DMS-2304697. 

\section{Preliminaries}
\subsection{NPC spaces and Euclidean buildings}
For the definitions in this subsection, we refer the readers to \cite{bridson-haefliger,Rou09,KP23}. 
 \begin{dfn}[Geodesic  space] Let $(X,d_X)$ be a metric space.  A curve $\gamma:[0, \ell] \rightarrow X$ into $X$ is called a geodesic if the length $d_X(\gamma(a),\gamma(b))=b-a$ for any subinterval $[a, b] \subset[0,\ell]$.  A metric space $(X,d_X)$ is a \emph{geodesic space} if there exists a geodesic connecting every pair of points in $X$.
\end{dfn}
\begin{dfn}[NPC space]An NPC (non-positively curved) space $(X,d_X)$ is a complete geodesic space that satisfies the following condition: for any three points $P,Q,R\in X$ and a  geodesic $\gamma:[0, \ell] \rightarrow X$ with $\gamma(0)=Q$ and $\gamma(\ell)=R$, we have
$$
d^{2} (P, Q_{t} ) \leq(1-t) d^{2}(P, Q)+t d^{2}(P, R)-t(1-t) d^{2}(Q, R)
$$
for any $t\in [0,1]$, where $Q_{t}:=\gamma(t\ell)$.
\end{dfn}
A smooth Riemannian manifold with nonpositive sectional curvature  is  an NPC  space. Among these, the Bruhat-Tits building $\Delta(G)$ associated with a semisimple algebraic group  $G$ defined over a non-archimedean local field  $K$ is noteworthy an example  of NPC spaces.  We will not provide the lengthy definition of Bruhat-Tits buildings here, but interested readers can find precise definitions in references such as \cite{Rou09} and \cite{KP23}.  It is noteworthy that $G(K)$ acts isometrically on the building $\Delta(G)$, and transitively on its set of apartments.  Here, $G(K)$ denotes the group of $K$-points of $G$. The dimension of $\Delta(G)$ is equal to the $K$-rank of the algebraic group $G$, which is the dimension of a maximal $K$-split torus in $G$. 


 \subsection{Harmonic maps to NPC spaces}\label{sec:harmonic}
  Consider a map $f: \Omega\to Z$ from an  $n$-dimensional Riemannian manifold $(\Omega, g)$ to an NPC space $(Z, d_Z)$.  When the target space $Z$ is a smooth  Riemannian manifold of nonpositive sectional curvature,  the energy of a smooth map $f: \Omega \rightarrow Z$ is
 $$
 E^{f}=\int_{\Omega}|d f|^{2} {\rm dvol}_g
 $$
 where $(\Omega, g)$ is a Riemannian domain and ${\rm dvol}_g$ is the volume form of $\Omega$.   We say $f: \Omega\to Z$ is harmonic if it is locally energy minimizing; i.e. for any $x \in \Omega$, there exists $r>0$ such that the restriction $\left.u\right|_{B_{r}(x)}$ minimizes energy amongst all maps $v: B_{r}(x) \rightarrow  {Z}$ with the same boundary values as $\left.u\right|_{B_{r}(x)}$. Here $B_{r}(x)$ denotes the geodesic ball of radius $r$ centered at $x$.
 
 In this paper, we mainly consider the target $Z$ to be  NPC spaces, not necessarily smooth. Let us recall the definition of  harmonic maps in this context  (cf. \cite{KS} for more details).
 
 Let $(\Omega, g)$ be a bounded Lipschitz Riemannian domain. Let $\Omega_{\ep}$ be the set of points in $\Omega$ at a distance least $\ep$ from $\partial \Omega$.  Denote by $S_{\ep}(x):=\partial B_{\ep}(x)$. We say $f: \Omega \rightarrow Z$ is an $L^{2}$-map (or that $f \in L^{2}(\Omega, Z)$ ) if for some point $P\in \Omega$, we have  
 $$
 \int_{\Omega} d^{2}(f(x), P) d \mathrm{vol}_{g}<\infty.
 $$
 For $f \in L^{2}(\Omega, Z)$, define
 $$
 e^f_{\ep}: \Omega \rightarrow \mathbb{R}, \quad e^f_{\ep}(x)= \begin{cases}\int_{y \in S_{\ep}(x)} \frac{d^{2}(f(x), f(y))}{\ep^{2}} \frac{d \sigma_{x, \ep}}{\ep^{n-1}} & x \in \Omega_{\ep} \\ 0 & \text { otherwise }\end{cases}
 $$
 where $\sigma_{x, \ep}$ is the induced measure on $S_{\ep}(x)$. We define a family of functionals
 $$
 E_{\ep}^{f}: C_{c}(\Omega) \rightarrow \mathbb{R}, \quad E_{\ep}^{f}(\varphi)=\int_{\Omega} \varphi e^f_{\ep} d\vol_{g} .
 $$
 We say $f$ has finite energy, denoted by   $f \in W^{1,2}(\Omega, Z)$,  if
 $$
 E^{f}[\Omega]:=\sup _{\varphi \in C_{c}(\Omega), 0 \leq \varphi \leq 1} \limsup _{\ep \rightarrow 0} E_{\ep}^{f}(\varphi)<\infty .
 $$ 
In this case, it was proven in \cite[Theorem 1.10]{KS} that   there exists an absolutely continuous function $e^f(x)$ with respect to   Lebesgue measure,   which we call the \emph{energy density}, such that $e^f_{\ep}(x) d \mathrm{vol}_{g} $ converges weakly to $ e^f(x) {\rm dvol}_g$ as $\ep$ tends to $0$. In analogy to the case of smooth targets, we write $|\nabla f|^{2}(x)$ in place of $e^f(x)$. Hence $|\nabla f|^{2}(x)\in L^1_{\rm loc}(\Omega)$. In particular, the (Korevaar-Schoen) energy of $f$ in $\Omega$ is
\begin{align}\label{eq:defenergy}
	E^{f}[\Omega]=\int_{\Omega}|\nabla f|^{2} {\rm dvol}_g . 
\end{align}

\begin{dfn}[Harmonic maps]
	 We say a continuous map $f: \Omega \rightarrow Z$ from a Lipschitz domain $\Omega$ is harmonic if it is locally energy minimizing; more precisely, at each $p \in \Omega$, there exists  
 an open  neighborhood $\Omega_p$ of $p$ such that all  comparison maps which agree with $u$ outside of this neighborhood have no less energy. 
\end{dfn} 
 
For  $V \in \Gamma \Omega$ where $\Gamma \Omega$ is the set of
Lipschitz vector fields  on $\Omega$,  in \cite[\S 2.3]{KS}, the \emph{directional energy} $|f_*(V)|^2$  is similarly
defined.  The real valued $L_{\rm loc}^1$ function $|f_*(V)|^2$ generalizes
the norm squared on the directional derivative of $f$.  The
generalization of the pull-back metric is the continuous, symmetric, bilinear, non-negative and tensorial operator
\[
\pi_f(V,W)=\Gamma \Omega \times \Gamma \Omega \rightarrow
L^1(\Omega, \bR)
\]
where
\[
\pi_f(V,W)=\frac{1}{2}|f_*(V+W)|^2-\frac{1}{2}|f_*(V-W)|^2.
\]
  We refer to  \cite[\S 2.3]{KS} for more
details.

Let  $(x_1, \dots, x_n)$ be local coordinates of $(\Omega,g)$, and $g=(g_{ij})$, $g^{-1}=(g^{ij})$ be the local metric expressions.  Then energy density function of
$f$ can be written (cf.~\cite[(2.3vi)]{KS})
\[
|\nabla f|^2 =  g^{ij} \pi_f(\frac{\partial}{\partial x_i}, \frac{\partial}{\partial x_j})
\]
 Next assume  $(\Omega, g)$ is a Hermitian domain and let $(z_1=x_1+ix_2, \dots, z_n=x_{2n-1}+ix_{2n})$ be local complex coordinates.   If we extend $\pi_f$ linearly  over $\bC$, then we have
 \[
 \frac{1}{4}|\nabla f|^2= g^{i\bar j} \pi_f ( \frac{\partial f}{\partial z_i},
 \frac{\partial f}{\partial \bar z_j}).
 \]

\begin{dfn}[Locally Lipschitz]
	A continuous map $f:\Omega\to Z$ is called \emph{ locally Lipschitz} if for any $p\in \Omega$, there exists   an open  neighborhood $\Omega_p$ of $p$   and a constant $C>0$ such that  $d(f(x), f(y))\leq Cd(x,y)$ for any $x,y\in \Omega_p$. 
\end{dfn}
\begin{rem}\label{rem:Litschitz}
It follows from the  definition of $|\nabla f|^2$ that if $f$ is locally Lipschitz, then for any $p\in \Omega$,  there exists    an open  neighborhood $\Omega_p$ of $p$   and a constant $C>0$  such that  over $\Omega_p$ one has
	 $
	 |\nabla f|^2\leq C.
	 $   
\end{rem}

 \subsection{Admissible coordinates}\label{sec:metric}
 The following definition of \emph{admissible coordinates} introduced in \cite{Moc06} will be used throughout the paper. 
 	\begin{dfn}(Admissible coordinates)\label{def:admissible}	Let $\overline{X}$ be a complex manifold and let $\Sigma$ be a simple normal crossing divisor in $\overline{X}$. Let $x$ be a point of $\Sigma$, and assume that $\{\Sigma_{j}\}_{ j=1,\ldots,\ell}$ 
 	are components of $\Sigma$ containing $x$. An \emph{admissible coordinate neighborhood} of $x$
 	is the tuple $(U ;z_1,\ldots,z_n;\varphi)$ (or simply  $(U ;z_1,\ldots,z_n)$ if no confusion arises) where
 	\begin{enumerate}[label=(\alph*)]
 		\item $U $ is an open subset of $\overline{X}$ containing $x$.
 		\item There is a holomorphic isomorphism   $\varphi:U \to \mathbb D^n$ such that  $\varphi(\Sigma_j)=(z_j=0)$ for any
 		$j=1,\ldots,\ell$.
 	\end{enumerate}  
 \end{dfn} 
We define a \emph{Poincar\'e-type} metric $\omega_P$ on $(\mathbb D^*)^\ell\times \mathbb D^{n-\ell}$  by
 \begin{align}\label{eq: Poin}
 \omega_P=\sum_{j=1}^{\ell}\frac{\sqrt{-1}dz_j\wedge d\overline{z}_j}{|z_j|^2(\log |z_j|^2)^2}+\sum_{k=\ell+1}^{n} \sqrt{-1}dz_k\wedge d\overline{z}_k. 
 \end{align}
We note that, using the notation from the definition, one can construct a global complete metric \( g \) on \( X \) of Poincaré-type at every point of \( \Sigma \), provided that \( \overline{X} \) is a compact Kähler manifold. 

We briefly recall the construction. Fix any Kähler metric \( \overline{\omega} \) on \( \overline{X} \). Write \( \Sigma = \sum_{j=1}^{k} \Sigma_j \) as a sum of irreducible components. For each \( j = 1, \dots, k \), choose a smooth Hermitian metric \( |\cdot|_j \) on \( \sO_{\overline{X}}(\Sigma_j) \) and take a section \( \sigma_j \in H^0(\overline{X}, \sO_{\overline{X}}(\Sigma_j)) \) such that \( \Sigma_j = (\sigma_j = 0) \) and \( |\sigma_j|_j < 1 \) over \( \overline{X} \). Then, it suffices to set, for some \( C \in \mathbb{R}_{>0} \) large enough,
\begin{equation} \label{griffithsmetric}
	g := C \overline{\omega} + \sum_{j=1}^{k} \frac{d|\sigma_j|_j \wedge d^c |\sigma_j|_j}{|\sigma_j|_j^2 (\log |\sigma_j|_j^2)^2}.
\end{equation}
This metric is said to be of \emph{Poincaré-type} around \( \Sigma \), meaning that for any \( x \in \Sigma \) and for any admissible coordinates centered at \( x \), there exist constants \( C_1, C_2 > 0 \) such that 
\[
C_1 \omega_P \leq g \leq C_2 \omega_P.
\]

	\section{Existence of Harmonic maps to Bruhat-Tits buildings}
   In this section, we prove  the existence assertion of equivariant pluriharmonic map in  \cref{GS}, together with a weaker version of  \cref{main:existence}, and  \cref{main:harmonic}. Several technical steps are deferred to the appendix. 

\begin{thm}[Existence of (pluri-)harmonic maps] \label{thm:harmonicmaps}	
Let $(\overline{X},\Sigma)$ be a log smooth pair, $G$ be a semisimple   algebraic group defined over a non-archimedean local field $K$, and $\Delta(G)$ be the Bruhat-Tits building of $G$.   Let $L$ be a sufficiently ample line bundle   on $\overline{X}$. Let $\varrho: \pi_1(X)\to G(K)$ be a Zariski dense representation. Then there exists a $\varrho$-equivariant pluriharmonic map $\tilde{u}:\widetilde{X}\to \Delta(G)$, that is  locally Lipschitz, and has logarithmic energy growth with respect to $(\overline{X},L)$ (cf. \cref{logenenergy}). Moreover, $\tilde{u}$  is harmonic with respect to \emph{any} K\"ahler metric  of $\widetilde{X}$.  
\end{thm}

\subsection{Preliminary lemmas}
Throughout the rest of this section,    let $G$ be a semisimple algebraic group defined over a non-archimedean local field $K$, and $\Delta(G)$ be the Bruhat-Tits building of $G$.  We denote by $d(\bullet,\bullet)$ the distance function of $\Delta(G)$.  We fix a Zariski dense representation $\varrho:\pi_1(X) \rightarrow G(K)$ as in Theorem~\ref{thm:harmonicmaps}.
Below, we summarize some  results regarding the action of $\varrho$.

\begin{lem} \label{lem:ptinfty}
If $\varrho:\pi_1(X) \rightarrow G(K)$ is  Zariski dense, then the following holds: \begin{enumerate}[label=\rm (\roman*)]
	\item \label{A1}The action of $H:=\varrho(\pi_1(X))$ on $\Delta(G)$ is without fixed points at infinity.
	\item  \label{A2}$\Delta(G)$ contains a non-empty closed
	minimal convex $H$-invariant subset $\mathcal{C}$.
\end{enumerate}
Here,  $\mathcal{C}$ is minimal means that there does not exist a non-empty closed convex strict subset of $\mathcal{C}$ invariant under $H$.  
\end{lem}
We refer the readers to \cite[Chapter II.8]{bridson-haefliger} for the definition of boundary at infinity of CAT(0) spaces. Roughly speaking, it is the set of equivalent classes of geodesic rays. 
\begin{proof}
If $H$ fixes a point at infinity, then $H$ is contained  $P(K)$ where $P$ is a proper parabolic subgroup of  $G$. This contradicts the fact that $H$ is Zariski dense and proves  \Cref{A1}.
\Cref{A2} follows from  \cite[Theorem 4.3, (A.ii)]{caprace-monod}.  We can argue as follows: suppose $\Delta(G)$ has no minimal closed convex $H$-invariant set.
Then it contains a decreasing sequence $X_n$ of closed convex $H$-invariant
sets whose intersection is empty.
Choose now a   base point $x$ in $\Delta(G)$ and consider the projection $x_n$ of $x$ to
$X_n$. Namely, $x_n$ is the unique point in $X_n$ such that $d(x,x_n)=\inf _{y\in X_n}d(x,y)$.  Such map exists by \cite[Proposition 2.4.(1)]{bridson-haefliger}.  This sequence is unbounded, otherwise the intersection was not
empty. Since the space is locally compact, it converges to some point at infinity. 
This point at infinity is fixed by any $h$ in $H$ because the distance
$d(h.x_n, x_n)$
is bounded by  $d(h.x, x)$   by \cref{lem:proj} below.  This proves \Cref{A2}. 
\end{proof}

\begin{lem} \label{lem:proj}
 There exists a unique  \emph{closest point projection} map  $\Pi:\Delta(G)\to \cC$, i.e., for any $x\in \Delta(G)$, there exists a unique   $\Pi(x)\in \cC$ such that $d(x,\Pi(x))=\inf _{y\in \cC}d(x,y)$.    Such projection map $\Pi: \Delta(G) \rightarrow \mathcal{C}$ is   distance decreasing, and $H$-equivariant; i.e.~$\Pi(g x)=g\Pi(x)$ for any $g\in H$ and any $x\in \Delta(G)$. 
\end{lem}

\begin{proof} 
The existence assertion for such projection map $\Pi$ follows from  \cite[Proposition 2.4.(1)]{bridson-haefliger}.  For $g \in H$ and any $y\in \cC$, we have
\[
d(g\Pi(x),gx)=d(\Pi(x), x)\leq d(g^{-1}y,x)  = d(y, gx).
\]
This  implies $\Pi(gx)=g \Pi(x)$.  By \cite[Proposition 2.4.(4)]{bridson-haefliger}, $\Pi$ is distance decreasing.   This proves \Cref{A2}.
\end{proof}

\begin{rem}\label{example:subtree}
	The proof of \cref{thm:harmonicmaps} proceeds by induction on the dimension of the domain $X$. To carry out this induction, we must first establish the uniqueness of the pluriharmonic map at each dimension. However, it is currently unknown whether an equivariant pluriharmonic map into $\Delta(G)$ is unique. To address this issue, we construct an equivariant pluriharmonic map into a closed minimal convex set $\mathcal{C}$ of Lemma~\ref{lem:ptinfty} and show that it is the unique equivariant pluriharmonic map into $\mathcal{C}$. This step is necessary due to the existence of examples of algebraic subgroups $H$ of a semisimple algebraic group $G$ with a proper, non-empty, closed minimal convex $H$-invariant subset of $\Delta(G)$ (cf. \Cref{exmaple:sub} below).
\end{rem}

\begin{example}\label{exmaple:sub}
	Let  $K$ be   a non-archimedean local field and let  $L$ be a finite extension of $K$.  Assume that   $G$ is an algebraic group defined over $K$ and split over $L$. Then $G(K)$ is Zariski dense and unbounded in $G(L)$, and the Bruhat-Tits building $\Delta(G,K)$ is a proper, closed, unbounded $G(K)$-invariant subset embedded in $\Delta(G,L)$. As an example, if $G={\rm SL}_2$, $K=\mathbb{Q}_2$, and $L=\mathbb{Q}_2(\sqrt2)$, then $\Delta(G,L)$ is a tree and $\Delta(G,K)$ is a closed subtree. This illustrates the importance of considering the existence of proper, non-empty, closed minimal convex $H$-invariant subsets in $\Delta(G)$.
\end{example}

%

As a closed convex subset of an NPC space, $\mathcal{C}$  is itself is a NPC space.  Since $\cC$ is $\varrho(\pi_1(X))$-invariant, we can define
\begin{equation} \label{hatrho}
\hat{\varrho}: \pi_1(X) \rightarrow  \mathsf{Isom}(\mathcal{C}) 
\end{equation}
by setting $\hat \varrho(\gamma)$ to be the restriction of $\varrho(\gamma)$ to $\mathcal{C}$. Here $\mathsf{Isom}(\mathcal{C}) $ denotes the isometry group of $\cC$. To lighten the notation, we abusively write $\varrho$ for $\hat{\varrho}$.

\begin{lem}   \label{lem:semisimple}
$\varrho(\pi_1(X)) \subset \mathsf{Isom}(\mathcal{C})$ consists of only semisimple elements, i.e., for any $g\in \varrho(\pi_1(X))$, there exists $P_0 \in \cC$ such that $\inf_{P \in \cC} d(P, gP)=d(P_0, gP_0)$. 
\end{lem}

\begin{proof}
	Since $G$ is semisimple,  $\Delta(G)$ is a Euclidean building without a Euclidean factor. 
Let $\hat g \in \varrho(\pi_1(X))$ such that $\hat g= g|_{\mathcal{C}}$ for some $g \in G(K)$.
By \cite[Theorem 4.1]{parreau} and the assumption that $\Delta(G)$ does not have a Euclidean factor, $g$ is either elliptic or hyperbolic.  That is, there exists $P_0 \in \Delta(G)$ such that $\min_{P \in \Delta(G)} d(P, gP)=d(P_0, gP_0)$.  By Lemma~\ref{lem:proj}, $\Pi$ is distance decreasing and $\varrho(\pi_1(X))$-invariant. It yields
\begin{align}\label{def:trans}
\inf_{P \in \Delta(G)} d(P, gP) & =    d(P_0, gP_0) \, \geqslant \, d(\Pi(P_0), \Pi(gP_0)) \, = \, d(\Pi(P_0), {g}\Pi(P_0)) \\
& \geqslant  \inf_{P \in \mathcal{C}} d(P, {g}P) \, \geqslant \, \inf_{P \in \Delta(G)} d(P, gP). \nonumber
\end{align}
In particular, 
\[
d(\Pi(P_0), \hat g\Pi(P_0)) = d(\Pi(P_0),g\Pi(P_0))= \inf_{P \in \mathcal{C}} d(P,gP)=\inf_{P \in \mathcal{C}} d(P,\hat gP). 
\]
Hence $\hat g$ is a semisimple isometry of $\mathcal{C}$.
\end{proof}
\begin{dfn}[Translation length]\label{def:translation}
For any $\gamma\in \pi_1(X)$,  the \emph{translation length} of $\varrho(\gamma)$ is 
\begin{align}\label{eq:translation}
	L_\gamma:=\inf_{P\in \Delta(G)}d(P,\varrho(\gamma)P)\stackrel{\eqref{def:trans}}{=}\inf_{P\in \cC}d(P,\varrho(\gamma)P).
\end{align}
\end{dfn}


\subsection{Equivariant maps and sections} \label{sec:ems}
Endow $X$ with a K\"ahler metric $g$.  Let $\mathcal{C}$ be as in Lemma~\ref{lem:ptinfty} and $\varrho :\pi_1(X) \rightarrow \mathsf{Isom}(\mathcal C)$ be as in (\ref{hatrho}).
The set of all $\varrho $-equivariant maps  into $\mathcal{C}$ are in one-to-one correspondence with the set of all sections of the fiber bundle
$
\Pi: \widetilde{X} \times_{\varrho } \mathcal{C}  \rightarrow X
$.  More precisely, for a $\varrho $-equivariant map $\tilde{f}: \widetilde{X} \rightarrow \mathcal{C}$,  we define a  section of $\Pi$ by  setting
$
  f(\pi_X(\tilde p)) = [(\tilde p, \tilde f(\tilde p))],
$ 
where $\tilde{p}$ is any point in $\widetilde{X}$.
Since the energy density function  $|\nabla \tilde f|^2$ on $\widetilde{X}$ is a $\pi_1(X)$-invariant function, it descends to a function on $X$, denoted by 
$
|\nabla f|^2$. We also define  the energy of $f$ in any open subset $U$ of $X$ by setting
\begin{align}\label{eq:defenergy2}
	 E^f[U] = \int_U |\nabla f|^2 d\mbox{vol}_g. 
\end{align}

\subsection{Pullback bundles}\label{sec:pullback}
Let $f:Y\to X$ be a morphism between smooth quasi-projective varieties. Let $\widehat{Y}$ be a connected component of $\widetilde{X}\times_XY$.  
Then we have the following commuting diagram:
\begin{equation*}
\begin{tikzcd}\widetilde Y \arrow[d, 
  "\pi_{\widehat Y}"]\arrow[dd, bend right=30, "\pi_Y"']
    &
    \\
\widehat Y \arrow[r, "\hat{f}"] \arrow[d, 
  "\hat \pi_Y"]
    & \widetilde X \arrow[d, "\pi_X" ] \\
  Y  \arrow[r, "f" ]
& X  \end{tikzcd}
\end{equation*}
It induces a fiber bundle $\hat{\Pi}_Y:\widehat{Y}\times_{f^*\varrho } \mathcal C\to Y$, 
%
such that one has the following commuting diagram:
\begin{equation*}
\begin{tikzcd}
 \widehat Y \times_{f^*\varrho } \mathcal C \arrow[r, "F"] \arrow[d, 
  "\hat \Pi_Y"]
    & \widetilde X \times_{\varrho } \mathcal C \arrow[d, "\Pi_X" ] \\
  Y  \arrow[r, "f" ]
& X. \end{tikzcd}
\end{equation*}
Note that, given any section  $u: X \rightarrow  \widetilde X \times_\varrho \mathcal C$ of $\Pi_X$,   the composition
\[
u \circ f:Y \rightarrow \widetilde X \times_\varrho \mathcal C
\]
 defines a section  
 of the fiber bundle $\widehat Y \times_{f^*\varrho} \mathcal C \simeq f^*(\widetilde X \times_\varrho \mathcal C)\to Y$, which in turn defines a  $f^*\varrho$-equivariant map $\hat u_f: \hat Y \rightarrow \mathcal C$.  Define $\widetilde{u_f}:=\hat{u}_f\circ \pi_{\widehat{Y}}$, which is an  $f^*\varrho$-equivariant map $\widetilde Y \rightarrow \mathcal C$. It  defines a section
 \begin{equation*} 
 u_f: Y \rightarrow \widetilde Y \times_{f^*\varrho} \mathcal C.
 \end{equation*}
In this paper, we will mainly  focus on the special case where  $Y$ is a closed smooth subvariety of $X$ and  
 $
 \iota: Y  \rightarrow X
 $ 
 is the inclusion map. In this cases,  we will use the notation
\begin{align}\label{restrictionsection}
	 u_Y: Y \rightarrow \widetilde Y \times_{\varrho_Y} \mathcal C.
\end{align}
in place of $u_\iota$, where $\varrho_Y:\pi_1(Y)\to {\rm Isom}(\cC)$ denotes  the composition of $\iota_*:\pi_1(Y)\to \pi_1(X)$ and $\varrho$.   
On the other hand, for any section $
u: Y \rightarrow  \widetilde Y \times_{f^* \varrho} \mathcal C$ 
of the fiber bundle $\widetilde Y \times_{f^* \varrho} \mathcal C \rightarrow Y$, the composition of $u$ with the natural map $\widetilde Y \times_{f^* \varrho} \mathcal C\to  \widetilde{X}\times_\varrho\cC$ 
 is a map $Y \rightarrow \widetilde X \times_\rho \mathcal C$.
 For notational simplicity, we will abusively denote this map as
\begin{equation} \label{descendingmap}
u:Y \rightarrow  \widetilde X \times_{\varrho } \mathcal C.
 \end{equation}  

\subsection{Regularity results of Gromov-Schoen}
Let $X$ be a hermitian manifold and let  $\tilde{u}:\widetilde{X}\to \Delta(G)$ be a $\varrho$-equivariant harmonic map.  Following  \Cref{sec:pullback},  let  $u:X  \rightarrow \widetilde{X} \times_{\varrho } \Delta(G)$ be the section corresponding to $\tilde{u}$.  We recall some results in \cite{GS92}. 
\begin{thm}[\cite{GS92}, Theorem 2.4]\label{thm:Lcon}
A harmonic map $\tilde{u}:\widetilde X \rightarrow \Delta(G)$ is locally Lipschitz continuous.\qed 
\end{thm}
\begin{dfn}[Regular points and singular points] \label{def:sing}
A point $x \in \widetilde{X}$ is said to be a {\it regular point} of $\tilde{u}$  if there exists a neighborhood ${\mathcal N}$ of $x$ and an apartment $A \subset \Delta(G)$ such that $\tilde{u}(\mathcal N) \subset A$.  
A {\it singular point} of $\tilde{u}$ is a point in $\widetilde{X}$ that is not a regular point.  Since $\tilde{u}$ is $\varrho$-equivariant and $G(K)$ acts transitively on the apartments of $\Delta(G)$,   it follows that if $x \in \widetilde{X}$ is a regular point (resp.~singular point) of $\tilde{u}$, then every point of $\pi_X^{-1}(\pi_X(x))$ is a  regular point (resp.~singular point) of $\tilde{u}$.
  We denote by $\mathcal R(\tilde{u})$ (resp.~$\mathcal S(\tilde{u})$) the set of all regular points (resp.~singular points) of $\tilde{u}$ and let $\mathcal R(u)=\pi_X(\mathcal R(\tilde{u}))$ (resp.~$\mathcal S(u)=\pi_X(\mathcal S(\tilde{u}))$). 
\end{dfn}

\begin{lem}[ \cite{GS92}, Theorem 6.4] \label{gs}
The set ${\mathcal S}(u)$  is  a closed subset of $X$ of  Hausdorff codimension at least two.
  For any compact subdomain $\Omega_1$ of $X$, there is a sequence 
of Lipschitz functions $\{\psi_i\}$ with $\psi_i \equiv 0$ in a neighborhood of ${\mathcal S}(\tilde{u}) \cap \bar \Omega_1$, $0 \leq \psi_i \leq 1 $ and $\psi_i(x) \rightarrow 1$ for all $x \in \Omega_1 \backslash {\mathcal S}(u)$ such that
\[
\lim_{i \rightarrow \infty} \int_{\Omega_1} |\nabla u|^2 |\nabla \psi_i| \omega^n =0
\]
and 
\[
\lim_{i \rightarrow \infty}\int_{\Omega_1} |\nabla \nabla u| \, | \nabla \psi_i| \omega^n =0.   
\] \qed
\end{lem} 

\subsection{A Bertini-type theorem}

In this subsection, we will prove  a Bertini-type theorem that plays a crucial role in proving the pluriharmonicity of $\tilde{u}$ in \cref{thm:harmonicmaps}.
\begin{proposition}\label{prop:Bertini}
	Let \((\overline{X},\Sigma)\) be a log smooth pair with \(n:=\dim X\geqslant 2\).  Let \(L\) be a very ample line bundle on \(\overline{X}\) and fix an integer \(k\geqslant 3\). Set \(T= |L^k|^{\times (n-1)}\). Consider the universal complete intersection \[\overline{\mathscr{R}}=\left\{(x,H_1,\cdots H_{n-1})\in \overline{X}\times T\mid  x\in H_1\cap \cdots \cap H_{n-1}\right\}\subset \overline{X}\times T,\]
	and let \(\mathscr{R}:=\overline{\mathscr{R}}\cap( X\times T)\) be the restriction of the universal family to \(X\times T\). Denote by \(\overline{\pi}:\overline{\mathscr{R}}\to  T\) and \(\pi :\mathscr{R}\to T\), the canonical projections induced by the  second projection   $\overline X \times T\to T$. Let \(T^\circ\subset T\) be the Zariski open subset such that, 
for every  \((H_1,\dots, H_{n-1})\in T^{\circ}\), the hypersurfaces \(H_1,\dots, H_{n-1}\)  are smooth, and the divisor \(H_1+\dots+ H_{n-1}+\Sigma\)  is simple normal crossing. 
Let us denote by \(\pi^{\circ}:\mathscr{R}^\circ=\pi^{-1}(T^\circ)\to T^\circ\) be the restricted family. Then:  
	\begin{thmlist} 
		\item The open subset \(T^{\circ}\) is non-empty.
		\item\label{item:tangent}   For any point $x\in X$ and $v\in T_{x}X$, there exists some  \((H_1,\dots, H_{n-1})\in T^{\circ}\) such that $x\in H_1\cap\ldots\cap H_{n-1}$ and $H_1\cap\ldots\cap H_{n-1}$ is tangent to $v$. 
		\item The family \(\pi^{\circ}:\mathscr{R}^{\circ}\to T^{\circ}\) is locally topologically trivial.
	\end{thmlist}
\end{proposition}
The proof of \cref{prop:Bertini} relies on the following   Bertini-type result.
\begin{lem}\label{lem:Bertini general}
	Let \(N\geq 3\) be a positive integer.  Let \(Y \subset \mathbb{P}^N\) be a smooth projective subvariety of dimension \(m\geq 1\). Fix an integer \(d \geqslant 3\). Let \(x\in \mathbb{P}^N\) and $v\in T_{\bP^N,x}$. Let $P_{x,v}\subset |\sO_{\bP^N}(d)|$ be the general  hypersurfaces in $\bP^N$ of degree \(d\) which pass  through \(x\) and are tangent to $v$. If $\dim Y\geqslant 2$, or $x\not\in Y$,  then $P_{x,v}$ is non-empty and 
	\begin{thmlist}
		\item  a general element of  $P_{x,v}$ is smooth;
		\item  a general element of  $P_{x,v}$  intersects with  \(Y\) transversely;
		\item the base locus of $P_{x,v}$ is $\{x\}$.  
	\end{thmlist}
\end{lem}
\begin{proof}
	Consider the incidence variety
	\[I=\left\{(y,H)\in Y\times  P_{x,v} \mid   y\in H \ \mbox{ and }\ T_yY\subset T_yH\right\}.\]
Then $I$ parametrizes the set of points \((y,H)\) such that \(H\) intersects \(Y\) non-transversally at \(y\). We first prove that  \(p_2(I)\neq  P_{x,v}\) where \(p_2:(y,H)\mapsto H\) is the second projection. We shall do this by a classical dimension count.

	Fix \(y\in Y\) and denote by \[I_y=p_1^{-1}(\{y\})\cong \{H\in  P_{x,v}\mid  (y,H)\in I\}\subset P_{x,v} \] where \(p_1:(y,H)\mapsto y\) is the first projection.   
	Consider the 1-jet map
\begin{align}\label{eq:J1}
	{\rm J}^1_x:H^0(\mathbb{P}^N,\mathscr{O}_{\mathbb{P}^N}(d))\to   \mathscr{O}_{\mathbb{P}^N}(d)\otimes \sO_{\bP^N,x}/\mathfrak{m}^2_{\mathbb{P}^N,x}
\end{align}
	which is surjective as $d\geq 3$. Note that $(x,v)\in T_{X,x}$ defines a linear map  $$L_v:\sO_{\bP^N,x}/\mathfrak{m}^2_{\mathbb{P}^N,x}\to \bC^2 $$
	given by $L_v(f)=(f(x), df(v))$. Let $V_{x,v}:=\ker (L_v\circ 	{\rm J}^1_x)$. For any $H\in |V_{x,v}|$, we have $x\in H$ and $H$ is tangent to $v$.  Hence $|V_{x,v}|=P_{x,v}$.    Note that $\dim V_{x,v}=\dim H^0(\mathbb{P}^N,\mathscr{O}_{\mathbb{P}^N}(d))-2$.  
	Consider the  map 
	\[	{\rm J}^1_{Y,y}:H^0(\mathbb{P}^N,\mathscr{O}_{\mathbb{P}^N}(d))\to    \mathscr{O}_{\mathbb{P}^N}(d)|_{Y}\otimes \sO_{Y,y}/\mathfrak{m}^2_{Y,y}\cong \mathbb{C}^{m+1}\]
	which is surjective as $d\geq 3$. 
	Then $|\ker {\rm J}^1_{Y,y}\cap V_{x,v}|=I_y$.  
	\begin{claim}\label{claim:cod}
		We have \[\codim_{P_{x,v}}I_y\begin{cases}
			=m+1 \quad \mbox{ if }x\neq y\\
		\geq 	m-1 \quad \mbox{ if } x=y.
		\end{cases}\]
	\end{claim}
	\begin{proof}
We may suppose that \(x=[1:0:\cdots:0]\).  An element \(H\in |\mathscr{O}_{\mathbb{P}^N}(d)|\) is given by a homogenous polynomial of degree \(d\), \[F=\sum_{\substack{i_0,\dots, i_N\\ i_0+\cdots +i_N=d}}a_{i_0,\dots, i_N}X_0^{i_0}\cdots X_N^{i_N}.\]
		 	Consider the inhomogeneous coordinate $(z_1,\ldots,z_N):=(\frac{X_1}{X_0},\ldots,\frac{X_N}{X_0})$. Then   $F$ be can expressed as \begin{align}\label{eq:inhomo}
		 		 f_0:=\sum_{\substack{i_0,\dots, i_N\\ i_0+\cdots +i_N=d}}a_{i_0,\dots, i_N}z_1^{i_1}\cdots z_N^{i_n}.
		 	\end{align} 
		We write $v:=\sum_{i=1}^{N}b_i\frac{\d}{\d z_i}|_{x}$.  The condition \(H\in |V_{x,v}|\) is equivalent to 
		\begin{align} \label{eq:V_x}
			a_{d,0,\dots, 0}=0, \quad
		b_1a_{d-1,1,0\ldots,0}+b_2a_{d-1,0,1,0\ldots,0}+\ldots+b_Na_{d-1,0,\ldots,1}=0.
		\end{align}
		
		\medspace
		
\noindent		{\bf Case 1:   \(y\neq x\). }	   We may suppose that   \(y=[0:1:0:\cdots:0]\). 
		On the open set \((X_1\neq 0)\subset \mathbb{P}^N\) we choose the coordinate $(z_0,z_2,\ldots,z_N):=(\frac{X_0}{X_1},\frac{X_2}{X_1}\ldots,\frac{X_N}{X_1})$. One deshomogenizes \(F\) to the polynomial 
\begin{align*} 
	 f=a_{1,d-1,0,\dots, 0}z_0+a_{0,d,0,\dots, 0}+ a_{0,d-1,1,0,\dots,0}z_2+\cdots+ a_{0,d-1,0,\dots,0,1}z_N+o(z).  
\end{align*}
		Therefore, the map $${\rm J}^1_y: H^0(\mathbb{P}^N,\mathscr{O}_{\mathbb{P}^N}(d))\to    \mathscr{O}_{\mathbb{P}^N}(d) \otimes \sO_{\bP^N,y}/\mathfrak{m}^2_{\bP^N,y}$$
		 is just given by \[{\rm J}^1_y(f)=a_{1,d-1,0,\dots, 0}z_0+a_{0,d,0,\dots, 0}+ a_{0,d-1,1,0,\dots,0}z_2+\cdots+ a_{0,d-1,0,\dots,0,1}z_N.\]
	 Since $d\geq 3$, it follows from 
\eqref{eq:V_x}  that $${\rm J}^1_y|_{V_{x,v}}:V_{x,v}\to \mathscr{O}_{\mathbb{P}^N}(d)\otimes \sO_{\bP^N,y}/\mathfrak{m}^2_{\mathbb{P}^N,y}$$
		is surjective.  Therefore, 
		\[	{\rm J}^1_{Y,y}|_{V_{x,v}}:V_{x,v}\to   \mathscr{O}_{\mathbb{P}^N}(d)|_{Y}\otimes \sO_{Y,y}/\mathfrak{m}^2_{Y,y}\cong \mathbb{C}^{m+1}\]
		is also surjective. This implies that 
 \[\codim_{P_{x,v}}I_y=\rank (	{\rm J}^1_{Y,y}|_{V_{x,v}})=m+1.\]
		
		\medspace
		
\noindent		{\bf  Case 2: $y=x$}.   
In the inhomogeneous coordinates  $(z_1,\ldots,z_N)$   introduced earlier,  the map \({\rm J}^1_{x}\) defined in \eqref{eq:J1} can be expressed as
		\begin{align}\label{eq:evaluation}
		{\rm J}^1_{x}(f_0)=a_{d,0,\ldots,0}+a_{d-1,1,0,\dots, 0}z_1 +\cdots+ a_{d-1,0,\dots,0,1}z_N, 
		\end{align}
		where $f_0$ is defined in \eqref{eq:inhomo}.
Then the rank of 	$${\rm J}^1_x|_{V_{x,v}}:V_{x,v}\to \mathscr{O}_{\mathbb{P}^N}(d)\otimes \sO_{\bP^N,x}/\mathfrak{m}^2_{\mathbb{P}^N,x}$$
is  $N-1$. It follows that  $\rank {\rm J}^1_{Y,x}|_{V_{x,v}}\geq m-1$.  Therefore,   \[\codim_{P_{x,v}}I_y=\rank (	{\rm J}^1_{Y,y}|_{V_{x,v}})\geq m-1.\]
	\end{proof}
By \Cref{claim:cod},   for any \(y\in Y\backslash\{x\}\), one has 
	\[\dim I_y\leqslant \dim P_{x,v}- m-1\quad \text{and} \quad \dim I_x\leqslant \dim P_{x,v}- m+1.\]
This implies that,  when $x\not\in Y$,  one has 
	\[\dim I= m+\dim P_{x,v}- m-1 <\dim P_{x,v}.\]
	When $x \in Y$ and $\dim Y\geqslant 2$,  one has  
	\[\dim I= \max\{m+\dim P_{x,v}- m-1,\dim P_{x,v}- m+1\}<\dim P_{x,v}.\]
In conclusion, $p_2(I)\subsetneqq P_{x,v}$. Note that for any $H\in P_{x,v}\backslash p_2(I)$, $H$ contains $x$, $H$ is tangent to $v$, and   it intersects with $Y$ transversely. 
	
	Now we want to show that a general element in $P_{x,v}$ is smooth. We first note that the base locus of $P_{x,v}$ is $\{x\}$.      By the Bertini Theorem, a general element of $P_{x,v}$  is smooth away from $x$. We just need to show that a general element of $P_{x,v}$  is smooth at $x$.   If $F\in H^0(\bP^N, \sO_{\bP^N}(d))$ is not smooth at $x$, then ${\rm J}_x^1(f_0)=0$, where $f_0(z_1,\ldots,z_N)$ is the inhomogeneous polynomial of $F$ defined in \eqref{eq:inhomo}. Therefore,  if we denote by $V\subset P_{x,v}$ the set of hypersurfaces which are singular at $x$, we have \[\codim_{P_{x,v}}V=N-1.\]
	This proves that a general element of $H$ in $P_{x,v}$ is smooth.  
\end{proof}

We can now turn to the proof of the proposition.
\begin{proof}[Proof of \cref{prop:Bertini}]
We embed \(\overline{X}\) into some \(\mathbb{P}^N\) using the very ample line bundle \(L\).   	The fact that \(T^{\circ}\) is non-empty is a direct consequence of \cref{lem:Bertini general}.

	
	Let us prove \cref{item:tangent}.  Write $\Sigma:=\sum_{i=1}^{m}\Sigma_i$. For any $I=\{i_1,\ldots,i_k\}\subset \{1,\ldots,m\}$, we denote by $\Sigma_I:=\Sigma_{i_1}\cap\ldots\cap \Sigma_{i_k}$.  Fix any $d\geqslant 3$, and consider $P_{x,v}\subset |\sO_{\bP^N}(d)|$  as above. Note that $x\not\in \Sigma_I$ for any $I\subset \{1,\ldots,m\}$.  According to \cref{lem:Bertini general},   a general hypersurface $H_1$ in $ P_{x,v}$ is smooth,   which intersects $\overline{X}$  transversely, and is also transverse to each $\Sigma_I$  with $\dim \Sigma_I\geqslant 1$. Therefore, $H_1\cap \Sigma$ is a simple normal crossing divisor of the smooth projective variety $H_1$, and $v\in T_{H_1,x}$. We now apply   \cref{lem:Bertini general} for the log smooth pair $(H_1\cap \overline{X},H_1\cap \Sigma)$ inductively to find smooth hypersurfaces $H_2,\ldots,H_{n-1}\in |\sO(d)|$ satisfying the conditions in \cref{item:tangent}.

	Let us now come to the last part of the statement. Let us consider \(\overline{\mathscr{R}}^{\circ}=\overline{\pi}^{-1}(T^\circ)\) and   denote by \(\overline{\pi}^{\circ}:\overline{\mathscr{R}}^{\circ}\to T^{\circ}\) the induced morphism. This is a smooth proper family of curves, and therefore each fiber has the same genus which we shall denote by \(g\). Moreover, since every fiber intersects with \(\Sigma\) transversally, this intersection  consists of exactly \(M:=(dL)^{(n-1)}\cdot \Sigma\) distinct points. In particular, the map \(\overline{\pi}^\circ|_{\Sigma\cap \overline{\mathscr{R}}^{\circ}}:\Sigma\cap \overline{\mathscr{R}}^{\circ}\to T^{\circ}\) is étale. From there one deduces that for any small enough (euclidean) open subset \(U\subset T^{\circ}\), there exists a homeomorphism \(\varphi:\overline{\pi}^{-1}(U)\to U\times C\), such that \(\varphi(\Sigma\cap \overline{\pi}^{-1}(U))=\{q_1,\dots, q_M\}\) where \(C\) is a fixed curve of genus \(g\)  with \(M\) distinct marked points. This implies in particular that \(\mathscr{R}|_U\cong U\times (C\backslash\{q_1,\dots, q_M)\) is topologically trivial. 
\end{proof}

\subsection{Logarithmic energy growth (I)} \label{subsec:logenergy}
Let $(\overline{X},\Sigma)$ be a log smooth pair. Let $L$ be a \emph{sufficiently ample} line bundle on $\overline{X}$. For a harmonic map on $X$, we introduce the notion of logarithmic energy growth with respect to $(\overline{X}, L)$. 

We first recall a Lefschetz hyperplane theorem for smooth quasi-projective varieties in \cite[Theorem 1.9]{Eyr04}.
\begin{thm}\label{thm:Lefschetz} 
Let $(\overline{X},\Sigma)$ be a log smooth pair. If $L$ is a very ample line bundle on $\overline{X}$, then for \emph{any} smooth hypersurface $H\in |L|$ such that $H+ \Sigma$ is  simple normal crossing  (the choice of such a hypersurface is generic by the Bertini theorem),  the natural homomorphism $\pi_1(H\backslash \Sigma)\to \pi_1(\overline{X}\backslash \Sigma)$ is surjective.      \qed
\end{thm}
For any element $s \in H^0(\overline{X},L)$, we  set 
$\overline{Y}_{\!\! s} := s^{-1}(0)$, $Y_s: =\overline{Y}_{\!\! s} \backslash \Sigma$,  and denote by $\iota_{Y_s}:Y_s \rightarrow X$  the inclusion map.
Let \begin{align}\label{eq:U}
	 {\mathbb U}= \{ s \in  H^0(\overline{X},L)\mid  \overline{Y}_{\!\! s} \  \mbox{ is smooth} \ \mbox{and} \ \overline{Y}_{\!\! s} + \Sigma \  \mbox{ is a  normal crossing divisor} \}. 
\end{align}  
For $q \in X$, consider  the subspace
\begin{align}\label{eq:Uq}
V(q) =\{s \in H^0(\overline{X}, L)\mid s(q)=0\} 
\mbox{ and } 
{\mathbb{U}}(q) = \bU \cap V(q).
\end{align}
According to \cref{lem:Bertini general}, the sets $\mathbb U$ and $\mathbb U(q)$  are Zariski dense open subsets of $H^0(\overline{X},L)$ and  $V(q)$ respectively.

According to \Cref{thm:Lefschetz}, it follows that $\varrho(\pi_1(Y_s)) = \varrho(\pi_1(X))$. This equality implies that if $\varrho(\pi_1(X))$ does not fix  a point at infinity of $\cC$, then $\varrho_{Y_s}$ also does not fix a point at infinity of $\cC$.  

In \cite{DMrs},  the second and fourth authors introduced the definition of \emph{logarithmic energy growth} for harmonic maps from quasi-projective curves to ${\rm CAT}(0)$-spaces. We can now extend this definition to any smooth quasi-projective variety.

Let	 $T:=|L|^{\times (n-1)}$.  Consider the universal complete intersection \[\overline{\mathscr{R}}=\left\{(x,H_1,\cdots H_{n-1})\in \overline{X}\times T\mid  x\in H_1\cap \cdots \cap H_{n-1}\right\}\subset \overline{X}\times T.\]
Let $T^\circ$ be the Zariski open subset of $T$ defined in \cref{prop:Bertini}.  We set $\sR^\circ:=(X\times T^\circ)\cap \overline{\mathscr{R}}$ and let 
 us denote by \(\pi^{\circ}:\mathscr{R}^\circ \to T^\circ\) the projection map. Then    by applying \Cref{thm:Lefschetz} inductively,  for each fiber $\cR$ of $\pi^\circ$,    the homomorphism  $\pi_1(\cR)\to \pi_1(X)$ is surjective.
\begin{dfn}[Logarithmic energy growth (I)] \label{logenenergy} 
	Let $\varrho:\pi_1(X)\to G(K)$ be a Zariski dense representation where $G$ is a semi-simple algebraic group defined over a non-archimedean local field $K$. Assume that $\tilde{u}:\widetilde{X}\to \Delta(G)$ is a   $\varrho$-equivariant harmonic map. 
If $\dim_{\bC} X=1$, we say    $\tilde{u}$ has \emph{logarithmic energy growth}
if there is a positive constant $C$ such that  for any $0<r<1$, one has 
 \begin{equation} \label{Cbddef}
\frac{L_\gamma^2}{2\pi} \log \frac{1}{r} \leq E^{{u}}[\mathbb D_{r,1}] \leq \frac{L_\gamma^2}{2\pi} \log \frac{1}{r}+C,
\end{equation}
where $\mathbb D$ is a conformal disk in $\overline{X}$ centered at $p\in \Sigma$.   The constant
 $L_\gamma$ is   the \emph{translation length} of $\varrho(\gamma)$ defined in \cref{def:translation}, where   $\gamma \in \pi_1(X)$ is the element corresponding to the loop $\gamma$ around $p$.

 If $\dim_{\bC}X\geqslant 2$,   a $\varrho$-equivariant harmonic map $\tilde{u}:\widetilde{X} \rightarrow \Delta(G)$ has  \emph{logarithmic energy growth   with respect to $(\overline{X},L)$},  if for any fiber  $\cR$ of $\pi^\circ:\sR^\circ\to T^\circ$,    the   section $u_{\cR}:\cR\to \cR\times_{\varrho_\cR}\Delta(G)$    has  logarithmic energy growth.  Here $\varrho_\cR:\pi_1(\cR)\to \Delta(G)$ and $u_\cR$ are defined in \eqref{restrictionsection}.
\end{dfn}
\begin{rem}
	Note that when $\dim X\geqslant 2$, the definition of logarithmic energy growth in \cref{logenenergy} depends \emph{a priori} on the choice of a projective compactification $\overline{X}$ of $X$ and a sufficiently ample line bundle $L$ on $\overline{X}$. In  \cref{thm:unicity}, we will prove that for the harmonic map constructed in \Cref{thm:harmonicmaps}, it has logarithmic energy growth with respect to any projective compactification $\overline{X}$ and any sufficiently ample line bundle $L$. Consequently, we can give  a more intrinsic definition of logarithmic energy growth in \cref{def:log energy} that surpasses \cref{logenenergy}. 
\end{rem}

\begin{example} 
To clarify Definition~\ref{logenenergy}, we   give an example of a harmonic map that \emph{does not} have logarithmic energy growth in the sense of Definition~\ref{logenenergy}.  For a non-archimedean local field $K$, the building of $\GL_1(K)$ is a real line $\mathbb R$.   The action of $\GL_1(K)$ on $\mathbb R$ is translation by $\nu(k)$ where $\nu: K^* \rightarrow \mathbb R$ is the valuation of $K$. 
Let  $X=\mathbb C^*$ and $\varrho:\pi_1(X) \rightarrow \GL_1(K)$ be the trivial representation, i.e.~$\varrho(\gamma)$ is the identity map for any $\gamma \in \pi_1(\mathbb C^*)$.  Consider the universal cover
\begin{align*}
	\pi:\bC&\to \bC^*\\
	w&\mapsto \exp(w).
\end{align*}
Define a map
\begin{align*}
 \tilde{u}:  \bC &\rightarrow \mathbb R\\
	  w &\mapsto \frac{1}{2}\int_{0}^w (\exp^*(d \log z+d\log \bar z)) = {\rm Re}(w).
\end{align*} 
Then $\tilde{u}$     is a $\varrho$-equivariant pluriharmonic function. It descends to a function $u:\bC^*\to\bR$ defined by $u(w):=  \log |w|$.
  
Endow $\bD^*$ with the standard Euclidean metric $\sqrt{-1}\frac{dz\wedge d\bar{z}}{2}$. However, note that the energy is independent of the choice of metric on the Riemann surface.  We can easily compute the energy of $u$ in the annulus $\bD_{r,1}:=\{r <|z| < 1\} \subset \mathbb C^*$:  
\[
E^u[\bD_{r,1}] =  \int_{\log {r}}^{\log 1} dt \cdot  \int_0^{2\pi} d\theta  =2\pi \log \frac{1}{r}.
\]
Although the energy of $u$ grows logarithmically as $r \rightarrow 0$, the $\varrho$-equivariant harmonic function $\tilde{u}$ {\it does not} have logarithmic energy growth in the sense of Definition~\ref{logenenergy}.  Indeed, the definition of logarithmic energy growth depends on the translation length $L_\gamma$ of $\varrho(\gamma)$ where $\gamma \in \pi_1(\mathbb C^*)$ corresponds to the loop around the puncture.  Since $\varrho$ is the trivial representation,   the translation length is $L_\gamma=0$ and the $\varrho$-equivariant harmonic function of  logarithmic energy growth is  identically equal to a constant.
 \end{example}

\subsection{Existence of harmonic maps from Riemann surfaces}

We state the existence and uniqueness of  equivariant harmonic maps from Riemann surfaces of logarithmic energy growth.
\begin{lem}\label{lem:rs}
Let $Y=\overline{Y} \backslash \{p_1, \dots, p_n\}$ where $\overline{Y}$ is a compact Riemann surface and  let $G$ be a semisimple linear algebraic group defined over    a non-archimedean local field $K$.  Assume that  $\varrho_Y:\pi_1(Y) \rightarrow G(K)$  is a Zariski dense representation.   Let $\mathcal{C} \subset \Delta(G)$ be a non-empty closed minimal $\varrho_Y(\pi_1(Y))$-invariant convex subset as in Lemma~\ref{lem:ptinfty}.    
Then there exists a  \emph{unique} $\varrho_Y$-equivariant harmonic map $\tilde{u}: \widetilde Y \rightarrow \mathcal{C}$  with logarithmic energy growth.  
\end{lem}

\begin{rem} \label{rmk:unique}
The existence statement in Lemma~\ref{lem:rs} directly follows from \cite[Theorem 1.1]{DMrs}. On the other hand, the uniqueness theorem of \cite[Theorem 1.2]{DMrs} is proven under the additional assumption that $\varrho:\pi_1(X)\to \Delta(G)$ does not fix the point at infinity. Thus, the main focus of the proof of Lemma~\ref{lem:rs} is to adapt the proof of \cite[Theorem 1.2]{DMrs} to the case where $\mathcal C$ is not necessarily the entire $\Delta(G)$.
\end{rem}

\begin{proof}[Proof of \cref{lem:rs}]
To prove  existence, we use the fact that $\mathcal{C}$ is an NPC space and apply \cite[Theorem 1.1]{DMrs}  for which the assumptions are:  
\begin{itemize}
\item[(A)] the action of $\varrho_Y(\pi_1(Y))$ on $\cC$ does not fix a point at infinity, and
\item[(B)] $\varrho_Y(\lambda^j)$ is semisimple for each $j\in \{1,\dots, n\}$, where $\lambda^j \in \pi_1(Y)$ is the element associated to the loop around the puncture $p_j$.
\end{itemize}
Lemma~\ref{lem:ptinfty} (i) implies assumption (A) and Lemma~\ref{lem:semisimple} implies assumption (B). 

To prove the uniqueness, we  use the minimality of $\mathcal{C}$ and a slight variation of the proof of \cite[Theorem 1]{DMrs} where the target space is a building.    We shall assume on the contrary that $\tilde{u}_0, \tilde{u}_1:\widetilde Y \rightarrow \mathcal{C}$ are distinct $\varrho_Y$-equivariant harmonic maps with logarithmic energy growth. The following three steps lead to a contradiction to the assumption that $\varrho_Y$ does not fix a point at infinity.

\noindent {\bf Step 1}.   We first define an increased  sequence of  subsets of $\cC$  
\begin{equation} \label{Cn}
C_0\subset\cdots\subset C_k\subset \cdots
\end{equation}
inductively as follows:
First, let $C_0=\tilde{u}_0(\widetilde Y)$, and then let $C_k$ be the union of the images of all geodesic segments connecting points of $C_{k-1}$.  
The $\varrho_Y(\pi_1(Y))$-invariance of $C_0$ implies the $\varrho_Y(\pi_1(Y))$-invariance of $C_k$.  The set  $\bigcup_{k=0}^{\infty}C_k$ is the convex hull of the image of $\tilde{u}_0$, and   the minimality of $\mathcal{C}$ implies 
\[
 \mathcal{C} =\overline{\bigcup_{k=0}^{\infty} C_k.} 
\]

\noindent {\bf  Step 2}.  To each $Q \in \mathcal{C}$, we assign a geodesic segment $\bar \sigma^Q$ in $\mathcal{C}$ as follows:
First, for  $Q=\tilde{u}_0(q) \in C_0$,  let 
\begin{equation} \label{geoseg}
\bar \sigma^Q:[0,1] \rightarrow \mathcal C, \ \ \bar \sigma^Q(t)=(1-t)\tilde{u}_0(q)+t \tilde{u}_1(q).  
\end{equation} 
In the above, the weighted sum $(1-t)P+tQ$ is used to  denote the points on the geodesic segment connecting    $P$ and $Q$.  Note that $\bar \sigma^Q$ is well-defined    by \cite[(3.1), (3.3)]{DMunique}.
Since $\mathcal{C}$ is a convex subset of $\Delta(G)$, $\tilde{u}_0$ and $\tilde{u}_1$ are harmonic as maps into $\Delta(G)$, we can thus apply
\cite[(3.16)]{DMunique} to conclude that  $\{\bar \sigma^Q\}_{Q \in C_0}$ is a  family of pairwise parallel of geodesic segments of  uniform length.  (We can assume they are all unit length by normalizing the target space.)  Since $\tilde{u}_0$ and $\tilde{u}_1$ are both $\varrho_Y$-equivariant, the assignment $Q \mapsto \bar \sigma^Q$ is  $\varrho_Y(\pi_1(Y))$-equivariant; i.e.~$
\varrho_Y(\gamma) \bar \sigma^Q = \bar \sigma^{\varrho_Y(\gamma)Q}$ for any $Q \in C_0$ and  $\gamma \in  \pi_1(Y)
$.

For $n \in \mathbb{N}$,  we inductively  define a $\varrho_Y(\pi_1(Y))$-equivariant map from $C_n$ to a family of pairwise parallel geodesic segments  as follows:  For any pair of points $Q_0, Q_1 \in C_{n-1}$, apply the  \emph{Sandwich Lemma} of [BH, II.2.12 Exercise] with vertices $Q_0, Q_1, P_0:=\bar \sigma^{Q_0}(1), P_1:=\bar \sigma^{Q_1}(1)$ to define a one-parameter family of parallel geodesic segments $\bar \sigma^{Q_t}:[0,1] \rightarrow \mathcal{C}$ with initial point $Q_t=(1-t)Q_0+tQ_1$ and terminal point $P_t=(1-t)P_0+tP_1$.  
The inductive hypothesis implies that  the map $Q \mapsto \bar \sigma^Q$ defined on $C_n$ is also  $\varrho_Y(\pi_1(Y))$-equivariant.    Finally, consider $Q \in \mathcal C$ such that  $Q_i \to Q$  where $Q_i \in \cup_{k=1}^{\infty} C_k$.  In this case, let   $ \sigma^{Q_i}$ be the corresponding $\rho_Y(\pi_1(Y))$-invariant geodesic segments and let $\sigma^Q$ be the limit of $\sigma^{Q_i}.$ 
The above construction defines a $\varrho_Y(\pi_1(Y))$-equivariant map  
\[
Q \mapsto \bar \sigma^Q
\]
from $\mathcal{C}$ 
to  a family of pairwise parallel geodesic  segments contained in $\mathcal{C}$.

\noindent {\bf  Step 3.} We extend these geodesic segments into a geodesic ray as follows:  
For $Q \in \mathcal{C}$, we inductively construct a sequence $\{Q_i\}$ of points in $\mathcal{C}$ by first setting $Q_0=Q$ and then defining $Q_i= \bar \sigma^{Q_{i-1}}(\frac{3}{4})$.
Next, let 
\[
L^Q=\bigcup_{i=0}^\infty I^{Q_i}
\]
where $I^{Q_i}=\bar \sigma^{Q_i}([0,1])$. Therefore, $L^Q$ is  a  union of pairwise parallel geodesic segments.  Thus,  $\{L^Q\}_{Q \in \cC}$ is a family of  pairwise parallel geodesic rays. Moreover,  the $\varrho_Y(\pi_1(Y))$-equivariance of the map  $Q \mapsto \bar \sigma^Q$ implies 
$\varrho(\gamma)\bar \sigma^{Q_{i-1}}(\frac{3}{4})=\bar \sigma^{\varrho(\gamma)Q_{i-1}}(\frac{3}{4})$.  Thus, if $\{Q_i\}$ is the sequence constructed starting with $Q_0=Q$, then $\{\varrho_Y(\gamma)Q_i\}$ is the sequence constructed starting  with $\varrho_Y(\gamma)Q_0= \varrho_Y(\gamma)Q$.  We thus conclude 
\[
\varrho(\gamma)L^Q=\bigcup_{i=0}^{\infty} \varrho(\gamma) I^{Q_i} = \bigcup_{i=0}^{\infty} I^{\varrho(\gamma)Q_i}=L^{\varrho(\gamma)Q}.
\]  
We are done by letting the geodesic ray $\sigma^Q:[0,\infty) \rightarrow \mathcal{C}$ be the extension of the geodesic segment $\bar \sigma^Q:[0,1] \rightarrow \mathcal{C}$ parameterizing $L^Q$.
Consequently,  we have constructed a $\varrho_Y(\pi_1(Y))$-equivariant map  
\[
Q \mapsto \bar \sigma^Q
\]
from $\mathcal{C}$ 
to  a family of pairwise parallel geodesic rays in $\cC$.

The above construction shows that  $\varrho_Y(\pi_1(Y))$ fixes the equivalence class  $[L^Q]$ of geodesic rays. This implies that the action of $\varrho_Y(\pi_1(Y))$ on $\cC$ fixes a  point at infinity.  It contradicts with Assumption (A), and we prove  the uniqueness assertion. 
\end{proof}

\subsection{Pluriharmonicity}  
\begin{dfn}[Pluriharmonic maps] \label{def:pluriharmonic}
Let $X$ be  a complex manifold.  A locally Lipschitz map $u:X \rightarrow \Delta(G)$ is  \emph{pluriharmonic} if $u \circ \psi:\bD\to \Delta(G)$ is  harmonic for any holomorphic map $\psi: \bD \rightarrow X$. 
\end{dfn} 
We will prove that in order to establish the pluriharmonicity of a harmonic map $u$ to the Euclidean building, it is sufficient to verify it over the regular set of $u$.

\begin{lem} \label{phequiv}
Let $u: U = \mathbb{D}^n \rightarrow \Delta(G)$ be a harmonic map with respect to the standard Euclidean metric on $U = \mathbb{D}^n$.  If  $\partial \bar{\partial} u = 0$ on the regular set $\mathcal{R}(u)$, then $u$ is pluriharmonic.  
\end{lem}
\begin{rem}
	 Note that if $x \in \mathcal{R}(u)$, we can select a neighborhood $\Omega_x$ of $x$ and an apartment $A$ such that $u(\Omega_x) \subset A$. Our assumption implies that, upon identifying $A \simeq \mathbb{R}^N$, the map $u: \Omega_x\rightarrow \mathbb{R}^N$ is smooth and satisfies $\partial \bar{\partial} u = 0$. 
\end{rem}

  \begin{proof}
Since pluriharmonicity is a local property, we are free to shrink $U$ and localize around any given point. We first establish the following claim: if $\bD \hookrightarrow U$ is an embedded holomorphic disk, then the restriction of $u$ to $\bD$ is holomorphic.

After possibly shrinking $U$, we can choose an admissible coordinate system $(U; z_1, z_2, \dots, z_n)$ such that $\bD = (z_2 = \dots = z_n = 0)$.
  Denote $z_*=(z_2, \dots, z_n)$ 
and let
  \[
  \bD^{z_*}:=\bD \times \{z_*\} \simeq \bD.
   \] 
Recall that the singular set $\mathcal S(u)$  of $u$ has Hausdorff codimension at least two by \cref{gs}. It follows from  \cite{Shi68} that, for almost every $z_* \in \bD^{n-1}$,  the Hausdorff dimension
\begin{equation} \label{hd0}
\dim_{\mathcal H}(S^{z_*})=0,
\end{equation}
 where $\ S^{z_*}:=\mathcal S(u) \cap \bD^{z_*}$.
Let $u_{z_*}=u|_{\bD^{z_*}}$ and $R^{z_*} = \mathcal R(u) \cap \bD^{z_*}$, where $\cR(u)$ denotes the set of regular points of $u$.

  Let $z_*$ be such that (\ref{hd0}) holds.  Let $\Omega \subset R^{z_*}$ be any Lipschitz domain  such that $u_{z_*}(\Omega) \subset A$ where $A \simeq \bR^N$ is an apartment of $\Delta(G)$.  Let $\Pi:\Delta(G) \rightarrow A$ be the closest point projection map into $A$. The differential equality $\partial \bar \partial u=0$ is the first variation formula for $u_{z_*}: \bD^{z_*}   \rightarrow A\simeq \bR^N$ and thus  $E^{u_{z_*}}[\Omega] \leq E^v[\Omega]$ for any comparison map $v:\Omega \rightarrow A$.  For a comparison map $v:\Omega \rightarrow \Delta(G)$ not  mapping into $A$, we have
$E^{u_{z_*}}[\Omega] \leq E^{\Pi \circ v}[\Omega] \leq E^{v}[\Omega]$ since the projection map $\Pi$ is distance decreasing.  This implies that $u_{z_*}$ is a harmonic map when restricted to the regular set $R^{z_*}$.  

We now show that $u_{z_*}$ is harmonic as a map from $\bD^{z_*}$.
Let $v:\bD^{z_*} \rightarrow \Delta(G)$ be a harmonic map with the same boundary values as $u_{z_*}$.
Since both $u_{z_*}$ and $v$ are smooth harmonic maps in $\bD^{z_*} \backslash S^{z_*}$, the function $d^2(u_{z_*},v)$ is subharmonic in $\bD^{z_*} \backslash S^{z_*}$ (cf.~\cite[Remark 2.4.3]{KS}).  By (\ref{hd0}), for any $j \in \bN$, there exists a open cover $\{B_{r_k}(p_k)\}_{k=1}^N$ of $S^{z_*}$ such that $\sum_{k=1}^N r_k < \frac{1}{j}$.  For each $k=1, \dots, N$, let $\varphi_k$ be a smooth function on $\bD^{z_*}$ satisfying the following properties: $0 \leq \varphi_k \leq 1$,  $\varphi_k$ is identically equal to 0 in $B_{r_k}(z_k)$,  $\varphi_k$ is identically equal to 1 outside $B_{2r_k}(z_k)$  and $|\nabla \varphi_k| \leq \frac{2}{r_k}$.  Let $\phi_j=\Pi_{k=1}^N \varphi_k$.  
For any smooth function $\eta \geqslant 0$  with compact support in  $\bD^{z_*}$,  we have
\begin{eqnarray} \label{gg}
0 & \leq  & \int_{\bD^{z_*}}(\eta \phi_j) \triangle d^2(u_{z_*},v) \frac{idz_1 \wedge d\bar z_1}{2}
\nonumber \\
& = & -\int_{\bD^{z_*}} \phi_j \nabla \eta \cdot  \nabla d^2(u_{z_*},v) \frac{idz_1 \wedge d\bar z_1}{2} 
\   -\int_{\bD^{z_*}} \eta \nabla \phi_j \cdot  \nabla d^2(u_{z_*},v) \frac{idz_1 \wedge d\bar z_1}{2}.
\end{eqnarray}
Because the Lipschitz constants  of  $u_{z_*}$ and $v$ are bounded  in the support of $\eta$,  
$$|\nabla \phi_j| \leq \sum_{k=1}^N |\nabla \varphi_k| \leq \sum_{k=1}^{N}\frac{2}{r_k}$$ and the support of $\varphi_k$ is contained in a disk of area $\pi (2r_k)^2$ , there exists a constant $C>0$ that can be chosen independently of $j$ such that 
\begin{eqnarray*}
\left| \int_{\bD^{z_*}} \eta \nabla \phi_j \cdot  \nabla d^2(u_{z_*},v) \frac{idz_1 \wedge d\bar z_1}{2} \right|  & \leq & 
\sum_{k=1}^N \int_{\bD^{z_*}}|\nabla  \varphi_k|  | \nabla d^2(u_{z_*},v) | \frac{idz_1 \wedge d\bar z_1}{2}  \\
& \leq & 
\sum_{k=1}^N\sup_{z \in {\rm supp}(\eta)} |\nabla d^2(u_{z_*}(z),v(z))| \cdot  \frac{2}{r_k}  \cdot \pi (2r_k^2) < \frac{C\pi }{j}.
\end{eqnarray*}
Thus, letting $j \rightarrow \infty$ in (\ref{gg}), we obtain
\[
0 \leq -\int_{\bD^{z_*}}  \nabla \eta \cdot  \nabla d^2(u_{z_*},v) \frac{idz_1 \wedge d\bar z_1}{2}.
\]
In other words, $d^2(u_{z_*},v)$ is (weakly) subharmonic in $\bD^{z_*}$.  Since $d^2(u_{z_*},v)=0$ on $\partial \bD^{z_*}$, the maximum principle implies $d^2(u_{z_*},v)=0$ in $\bD^{z_*}$.  Thus, $u_{z_*}=v$ and hence $u_{z_*}$ is harmonic for a.e.~$z_* \in \bD^{n-1}$.  Since the uniform limit of  harmonic maps is harmonic, $u_{z_*}$ is harmonic for all $z_* \in \bD$.  This completes  the proof of the assertion.

Now let $\psi: \bD \rightarrow U$ be a holomorphic map and $C$ be the set of critical points of $\psi$.  There is a neighborhood $V$ of any $z \in \bD \backslash C$ such that $\psi|_V$ is an embedding. The composition $u \circ \psi|_V$ is harmonic  by the above assertion.  Thus,    $u \circ \psi$ is  harmonic in $\bD \backslash C$.  Letting $v: \bD \rightarrow \Delta(G)$ be a harmonic map with the same boundary values as $u$, we can use the same argument above to prove $d^2(u,v)=0$. Hence $u$ is harmonic, and the lemma is proved. 
\end{proof}


\subsection{Existence of pluriharmonic map from quasi-projective surfaces} 
\begin{thm} \label{lem:ks}
Let $(\overline{X},\Sigma)$ be a log smooth pair with $\dim X=2$.  Let  $L$ be a sufficiently ample line bundle on $\overline{X}$. Let $G$ be a semi-simple algebraic group over a non-archimedean local field $K$.  Assume that $\varrho: \pi_1(X) \to G(K)$ is a Zariski-dense representation, and that $\mathcal{C} \subset \Delta(G)$ is a non-empty minimal   convex $\varrho(\pi_1(X))$-invariant closed subset (cf.~\Cref{lem:ptinfty}). 
	
Fix   a Kähler metric $g$ on $X$ of Poincaré type as described in \Cref{sec:metric}. 	Then there exists a $\varrho$-equivariant harmonic map $\tilde{u}: \widetilde{X} \to \mathcal{C}$, where $\varrho$ is considered as a representation $\pi_1(X) \to \mathrm{Isom}(\mathcal{C})$ as defined in \eqref{hatrho}, such that the following holds:
	\begin{enumerate}[label=(\arabic*)]
		\item \label{item:prop1} The map $\tilde{u}$ is pluriharmonic.
		\item \label{item:prop2}  The map $\tilde{u}$ has logarithmic energy growth with respect to $(\overline{X}, L)$.
		\item   Properties in \Cref{item:prop1,item:prop2} uniquely characterize this map $\tilde{u}$. 
	\end{enumerate}
\end{thm} 
\begin{proof}
If $\varrho(\pi_1(X))$ is bounded, then $\varrho(\pi_1(X))$ fixes a point $P \in \Delta(G)$, allowing us to define $\tilde{u}(x) = P$ for any $x \in \widetilde{X}$. Therefore, we assume that $\varrho(\pi_1(X))$ is unbounded. In this case, $\cC$ must also be unbounded. Otherwise, by the Bruhat-Tits fixed point theorem, $\cC$ would have a barycenter that is fixed by $\varrho(\pi_1(X))$, contradicting our assumption that $\varrho(\pi_1(X))$ is unbounded. 

   The existence of a $\varrho$-equivariant harmonic map 
\begin{equation} \label{equivharmonic}
\tilde{u}: \widetilde X \rightarrow \mathcal C \subset \Delta(G)
\end{equation}
 follows from \cite{DMks}.  Indeed, 
the closed unbounded convex subset $\mathcal C \subset \Delta(G)$ is an NPC space.  Then \cite[Theorem 1]{DMks} asserts that there exists  a $\varrho$-equivariant harmonic map $\tilde{u}: \widetilde{X} \rightarrow \mathcal C$.  
Let $u$ be its  corresponding section (cf.~\Cref{sec:ems}).  

\medspace

\noindent \begin{(i)}
The harmonic map $\tilde{u}$ is in fact a pluriharmonic map.  We defer the details of this proof to Theorem~\ref{thm:pu} in \Cref{sec:pluriharmonic}. 
\end{(i)}
\\
  \begin{(ii)}
As $\dim_{\bC} X=2$, it suffices to check that for  
any $s \in \bU$  with $\bU$ defined in (\ref{eq:U}),  $u_s:=u|_{Y_s}$ has logarithmic energy growth, where $\overline{Y}_{\!\! s}:=s^{-1}(0)$ and $Y_s:=\overline{Y}_{\!\! s}\cap X$. 
Let $p \in \Sigma \cap \overline{Y}_{\!\! s}$.   Since $\overline{Y}_{\!\! s} + \Sigma$ is a normal crossing divisor,    $p$ is a smooth point of $\Sigma$. 
By  \cite[Theorem 6.6]{DMks}, there exists  an admissible coordinate neighborhood $(U;z_1,z_2)$ centered at   $p$, and a positive constant $C$ such that   $U\cap \Sigma=(z_1=0)$ and 
\begin{equation} \label{fromjoel}
 \int \limits_{\bD^* \times \bD} 
\left( \left| \frac{\partial u}{\partial z_1}(z_1,z_2) \right|^2 
- \frac{L^2_\gamma}{2\pi}  \frac{1}{|z_1|^2} \right) \frac{idz_1 \wedge d\bar z_1}{2} \wedge \frac{idz_2 \wedge d\bar z_2}{2} \leq C,
\end{equation}
where  $L_\gamma$ is the translation length of $\varrho(\gamma)$ with $\gamma \in \pi_1(X)$ corresponding to the loop $\theta \mapsto (r e^{i\theta},0)$.   
\begin{claim}
	 There is a positive constant $C_0$ such that 
	 \begin{equation} \label{esjoel1}
	 	\left| \frac{\partial u}{\partial z_2} (z_1,z_2)\right|\leq C_0, \ \    \forall (z_1,z_2) \in  \bD^*_\frac{1}{2} \times  \bD_\frac{1}{2}.
	 \end{equation}
\end{claim}
\begin{proof}
	  By  \cref{def:translation}, we have
	 \begin{align}\nonumber
	 	\frac{L^2_\gamma}{2\pi} \frac{1}{r} &\leq    \frac{1}{2\pi r}  \left(\int_0^{2\pi} 
	 	\left| \frac{\partial u}{\partial \theta}(re^{i\theta},z_2)  \right| \frac{d\theta}{r} \right)^2 \leq \frac{r}{2\pi }  \left(\int_0^{2\pi} 
	 	\left| \frac{\partial u}{\partial z_1}(re^{i\theta},z_2) \right|  d\theta\right) ^2 
	 	\\  \label{lb}
	 &	\leq  \int_0^{2\pi} 
	 	\left| \frac{\partial u}{\partial z_1}(re^{i\theta},z_2) \right|^2 rd\theta
	 \end{align}
	 for any $z_2\in \bD$ and $r\in (0,1)$. Here the last inequality follows from the Cauchy-Schwarz inequality. 
	 Thus,  (\ref{fromjoel}) and (\ref{lb}) imply that
	 \begin{equation} \label{z1z1}
	 	0 \leq \int \limits_{{\mathbb D}_{r_1,r_2}} 
	 	\left| \frac{\partial u}{\partial z_1}(z_1,z_2) \right|^2 \frac{idz_1 \wedge d\bar z_1}{2}
	 	- \frac{L^2_\gamma}{2\pi}  \log \frac{r_2}{r_1}  \leq C(z_2), \ \ \mbox{for a.e.}~z_2 \in {\bD}_{\frac{1}{2}}
	 \end{equation}
	 where  $C(z_2)$ is a non-negative integrable function defined on ${\mathbb D}_\frac{1}{2}$.   
	 
	 We will next show that,  we can replace $C(z_2)$ in (\ref{z1z1}) by a positive constant $C_0$ that depends only on $\varrho(\gamma)$ and the Lipschitz constant of  $u|_{\partial \mathbb D \times \mathbb D}$. Indeed,  \cite[Theorem 3.1]{DMrs} and   (\ref{z1z1}) imply that for each $z_2$, the map $z_1 \mapsto u_{z_2}:=u(z_1,z_2)$ is the unique Dirichlet solution for the boundary value $u_{z_2}\big|_{\partial \bD_\frac{1}{2}}$ and that  the constant $C_0$ depends only on the translation length $L_\gamma$ and the Lipschitz constant of $u_{z_2}\big|_{\partial \bD_\frac{1}{2}}$.   Here we are using the fact that the isometries of $\Delta(G)$ are always semisimple when $G$ is semisimple by \cref{lem:semisimple}.   Since $u$ is locally Lipschitz,  the Lipschitz constant of $u_{z_2}\big|_{\partial \bD_\frac{1}{2}}$ has a uniform bound for all $z_2 \in \bD_\frac{1}{2}$.  Hence, the choice of $C_0$ can be made independently of $z_2$.  The lower semicontinuity of energy then implies that (\ref{z1z1}) with $C_0$ instead of $C(z_2)$ holds for all $z_2 \in \bD_\frac{1}{2}$ (not just a.e.~$z_2$); i.e.
	 \begin{equation} \label{esjoel2}
	 	0 \leq \int \limits_{{\mathbb D}_{r_1,r_2}} 
	 	\left| \frac{\partial u}{\partial z_1}(z_1,z_2)\right|^2  \frac{i dz_1 \wedge d\bar z_1}{2}
	 	- \frac{L^2_\gamma}{2\pi}  \log \frac{r_2}{r_1} \leq C_0, \ \ \forall z_2 \in \bD_\frac{1}{2}.
	 \end{equation}
	   Since for each  $z_2\in \bD$,  $u_{z_2}$ is a harmonic section of logarithmic energy growth, the proof of \cite[Lemma 4.1]{DMks} implies  (\ref{esjoel1}).  For the sake of completeness, we summarize this argument here.  Let $z_2, z_2' \in \bD_\frac{1}{2}$ and\[
	 \delta_{z_2,z_2'}(z_1)=d(\tilde{u}(z_1,z_2), \tilde{u}(z_1,z_2'))
	 .\]
	 Since $u_{z_2}$ is harmonic for each $z_2$,  $\delta^2_{z_2,z_2'}$ is a continuous subharmonic function defined in  $\bD^*$ (cf.~\cite[Remark 2.4.3]{KS}).  Since $ {u}_{z_2}$ and $ {u}_{z_2'}$  have logarithmic energy growth, by  \cite[Remark 3.12]{DMrs}, one has
	 $$
	 \lim_{|z_1| \rightarrow 0} \delta^2_{z_2,z_2'}(z_1) +\ep \log |z_1| =-\infty.
	 $$  
	 Thus, $\delta^2_{z_2, z_2'}$
	 extends to subharmonic function on $\bD_\frac{1}{2}$ (cf.~\cite[Lemma~3.2]{DMrs}).  We can apply the  maximum principle to conclude that  
	 \begin{equation} \label{maxprin}
	 	\delta^2_{z_2,z_2'}(z_1) \leq 
	 	\sup_{\zeta \in \partial \bD} \delta^2_{z_2,z_2'}(\zeta)\leq  \Lambda^2 |z_2 -z_2' |^2, \ \  \forall z_1 \in \bD^*_\frac{1}{2}
	 \end{equation}
	 where  the constant $\Lambda$  can be chosen independently  of $z_2, z_2' \in \bD_\frac{1}{2}$ since $\tilde{u}$ is locally Lipschitz continuous.  This implies   \eqref{esjoel1}.
\end{proof}
 Consider a local trivialization of $L|_U \simeq U \times \bC$, and let $s_U \in \sO(U)$ denote the image of the section $s$ under this trivialization.
  Define
\[
\Phi: \bD \times \bD \rightarrow  \Phi(\bD \times \bD), \ \ \Phi(z_1,z_2) = (w_1,w_2), \ \ \left\{ \begin{array}{l}
 w_1  =  z_1
 \\
 w_2  =  s_U(z_1,z_2)
\end{array}.
\right.
\]  
The fact that $\overline{Y}_{\!\! s}\cap U=s_U^{-1}(0)$ intersects with  $(z_1=0)$ transversely implies that  $\frac{\partial s_U}{\partial z_2}(z_1,z_2) \neq 0$ for $(z_1,z_2)$ sufficiently close to $(0,0)$.  Thus, after  shrinking $U$, we can assume that  $\Phi$ defines a holomorphic change of coordinates  in $U$.
Define a  holomorphic function $\eta(w_1,w_2)$ by
\[
\Phi^{-1}(w_1,w_2)=(z_1,z_2),  \ \ \left\{ \begin{array}{l}
 z_1 = w_1
 \\
 z_2 = \eta(w_1,w_2)
\end{array}.
\right.
\]
Note that 
$
 w_1 \mapsto (w_1,\eta(w_1,w_2)) $
 defines $w_1$ as holomorphic coordinate of the Riemann surface $s^{-1}(w_2)$.  
 
Denote
\[
u_{w_2}(w_1) := u (w_1,\eta(w_1,w_2)).
\]
Whenever $u(w_1,\eta(w_1,w_2))$ is a regular point (cf.~Definition~\ref{def:sing} and Lemma~\ref{gs}), we apply the chain rule to obtain 
\begin{eqnarray}
\frac{d u_{w_2}}{d w_1} (w_1) & = &   \frac{\partial u}{\partial z_1} (w_1,\eta(w_1,w_2)) + \frac{\partial u}{\partial z_2} (w_1,\eta(w_1,w_2)) \frac{\partial \eta}{\partial w_1} (w_1, w_2).
\nonumber 
\end{eqnarray}
Since  $\left| \frac{\partial \eta}{\partial w_1} (w_1, w_2)\right|$ is bounded,  the estimate (\ref{esjoel1}) implies that there exists a constant $C>0$ such that 
\begin{eqnarray}
	\label{tcr} 
\left| \frac{d u_{w_2}}{d w_1} (w_1) \right|^2 & \leq & \left|\frac{\partial u}{\partial z_1} (w_1,\eta(w_1,w_2)) \right|^2+ C\left|\frac{\partial u}{\partial z_2} (w_1,\eta(w_1,w_2)) \right|+C.
\end{eqnarray}
Since the regular set  $\mathcal R(u)$ of $u$ is an open set  is of full measure, $\Phi(\mathcal R(u))$ is also an open set of full measure.  Furthermore, since $u$ locally Lipschitz continuous, the right hand side of (\ref{tcr}) is a bounded function.   Thus, we can subtract  $\frac{L_\gamma^2}{2\pi} \frac{1}{|w_1|^2}$ from both sides of (\ref{tcr})    and integrate over  $\bD^*_\ep \times  \bD_\ep$ for some small  $\ep>0$  such that $\Phi^{-1}(\bD_\ep \times \bD_\ep) \subset \bD_\frac{1}{2} \times \bD_{\frac{1}{2}}$ to obtain 
\begin{eqnarray} 
\lefteqn{\int \limits_{\bD^*_\ep \times \bD_\ep} 
\left( \left| \frac{d u_{w_2}}{d w_1}(w_1)\right|^2 - \frac{L_\gamma^2}{2\pi} \frac{1}{|w_1|^2} \right)  
\frac{idw_1 \wedge d\bar w_1}{2} \wedge \frac{idw_2 \wedge d\bar w_2}{2}
}
\nonumber
\\
\nonumber
& \leq & 
\int \limits_{\bD^*_\ep \times  \bD_\ep} 
\left(  \left| \frac{\partial u}{\partial z_1} (w_1,\eta(w_1,w_2))\right|^2  - \frac{L_\gamma^2}{2\pi} \frac{1}{|w_1|^2}  \right) 
\frac{idw_1 \wedge d\bar w_1}{2} \wedge \frac{idw_2 \wedge d\bar w_2}{2}
\nonumber 
\\
& & \ + C\int \limits_{\bD^*_\ep \times  \bD_\ep} \left|\frac{\partial u}{\partial z_2} (w_1,\eta(w_1,w_2)) \right| \frac{idw_1 \wedge d\bar w_1}{2} \wedge \frac{idw_2 \wedge d\bar w_2}{2} +C
\nonumber 
\\
& = & 
\int \limits_{\Phi^{-1}(\bD^*_\ep \times  \bD_\ep)} 
\left(  \left| \frac{\partial u}{\partial z_1} (z_1,z_2)\right|^2  - \frac{L_\gamma^2}{2\pi} \frac{1}{|z_1|^2}  \right)
 \left| \frac{\partial s_U}{\partial z_2} \right|^2 
\frac{idz_1 \wedge d\bar z_1}{2} \wedge \frac{idz_2 \wedge d\bar z_2}{2}
\nonumber 
\\
& & \ + C\int \limits_{\Phi^{-1}(\bD^*_\ep \times  \bD_\ep)} \left|\frac{\partial u}{\partial z_2}(z_1,z_2) \right|  \left| \frac{\partial s_U}{\partial z_2} \right|^2  \frac{idz_1 \wedge d\bar z_1}{2} \wedge \frac{idz_2 \wedge d\bar z_2}{2} +C.
\nonumber
 \label{udub}
\end{eqnarray} 
Since  $\left| \frac{\partial s_U}{\partial z_2} \right|^2$ is bounded,   (\ref{fromjoel}) implies that  the first integral on the right hand side of the above inequality is finite.  By \eqref{esjoel1}, the second integral on the right hand side is also finite. 
Thus,  we conclude that for a.e.~$w_2 \in \bD_\ep$, 
\[
\int \limits_{\bD_\ep^*}
\left( \left| \frac{d u_{w_2}}{d w_1}(w_1)\right|^2 - \frac{L_\gamma^2}{2\pi} \frac{1}{|w_1|^2} \right) \frac{idw_1 \wedge d\bar w_1}{2}
\leq C(w_2).
\]
We can now proceed as before (cf.~from (\ref{z1z1}) to (\ref{esjoel2})) to show that $C(w_2)$ can be replaced by a constant $C$ independent of $w$; i.e. there exists a  positive constant $C$   such that for every $w_2 \in \bD_\ep$ and $0<r_1<r_2<\ep$,  we have
\begin{equation} \label{w1}
0 \leq \int \limits_{{\mathbb D}_{r_1,r_2}} 
\left| \frac{\partial u_{w_2}}{\partial w_1} \right|^2 \frac{i dw_1 \wedge d\bar w_1}{2}
- \frac{L^2_\gamma}{2\pi}  \log \frac{r_2}{r_1}  \leq C.
\end{equation}
Note that the lower bound of 0 follows from (\ref{lb}).
 Applying (\ref{w1}) with $w_2=0$,  we conclude that $u_0=u_s$ has logarithmic energy growth in the sense of \cref{logenenergy}.  \end{(ii)}  
 \\
\noindent \begin{(iii)}
To  prove the uniqueness assertion, 
let $v:\widetilde{X} \rightarrow \mathcal{C}$ be another  $\varrho$-equivariant pluriharmonic map into $\mathcal{C}$ of logarithmic energy growth with respect to $(\overline{X},L)$.  For any $q \in X$, there exists  a  section $s \in \mathbb  U(q)$ with $\bU(q)$ defined in \eqref{eq:Uq}.  We define $\varrho_{Y_s}:=\varrho|_{\pi_1(Y_s)}$.  By the definition of $\bU(q)$ and \Cref{thm:Lefschetz}, $\varrho_{Y_s}(\pi_1(Y_s))=\varrho(\pi_1(X))$ and thus $\varrho_{Y_s}$ does not fix a point at infinity of $\cC$.  
Consider the sections of the fiber bundle $\widetilde{X} \times_\varrho \mathcal{C} \rightarrow X$ defined by  the pluriharmonic maps $u$ and $v$, and denote their restrictions to $Y_s$ by  $u_{Y_s}: Y_s \rightarrow \widetilde Y_s \times_{\varrho_{Y_s}} \mathcal{C}$ and $v_{Y_s}: Y_s \rightarrow \widetilde Y_s \times_{\varrho_{Y_s}} \mathcal{C}$.   Since $Y_s$ is a Riemann surface,   the pluriharmonicity of $u$ and $v$ implies that $u_{Y_s}$ and $v_{Y_s}$ are  harmonic sections, and have logarithmic energy growth by \cref{logenenergy}. 
By the uniqueness assertion of   Lemma~\ref{lem:rs}, we  conclude $u_{Y_s}=v_{Y_s}$.   Since  $q$ is an arbitrary point in $X$, we conclude $u=v$. 
\end{(iii)}    
The proof of the theorem is accomplished. 
\end{proof}

\subsection{Proof of Theorem~\ref{thm:harmonicmaps} 
}
\label{sec:proof}
	%

\begin{proof}[Proof of \Cref{thm:harmonicmaps}] 
 The proof is organized into five steps. In the first step,   we construct  a  map  $\tilde{u}:\widetilde{X}\to \Delta(G)$   through an inductive process. Moving onto the second step, we establish that  such $\tilde{u}$ is   locally harmonic with respect to the Euclidean metric.  In the third step we  prove the pluriharmonicity  of $\tilde{u}$.  Subsequently, in the fourth step, we establish that $\tilde u$ is harmonic with respect to any K\"ahler metric on $X$. Finally, in the last step, we show the uniqueness of $\tilde u$.

\medspace

\noindent {\bf Step 1:} {\it We prove the existence of $u$}. Consider the following assertion:
\begin{itemize}
  \item[($\ast$)]   Let $\cC$ be a   non-empty minimal closed convex $\varrho(\pi_1(X))$-invariant subset  of $  \Delta(G)$ introduced in Lemma~\ref{lem:ptinfty}. Let $L$ be a sufficiently  ample line bundle   on $\overline X$.  
Then  there exists a $\varrho$-equivariant pluriharmonic map $\tilde{u}:\widetilde{X} \rightarrow \mathcal{C} \subset \Delta(G)$  of logarithmic energy growth with respect to $(\overline X,L)$.  Moreover, such  map $u$ is the  \emph{unique} $\varrho$-equivariant pluriharmonic map  into $\mathcal{C}$ of logarithmic energy  growth with respect to $(\overline{X},L)$. 
\end{itemize}

\vspace*{0.1in}

\noindent {\it Initial Step}. The statement $(\ast)$ is true for $\dim_{\mathbb C} X=2$ by Theorem~\ref{lem:ks}.

\medspace

\noindent {\it Inductive Step}.  We assume ($\ast$) whenever $\dim_{\mathbb C}X=2,\dots, n-1$.  
Now let $\dim_{\mathbb C}X=n\geqslant 3$.   
For each section   $s \in \mathbb U$ with $\bU$ as in \eqref{eq:U},   $\varrho_{Y_s}(\pi_1(Y_s))=\varrho(\pi_1(X))$ by \cref{thm:Lefschetz}. Thus, the inductive hypothesis implies that there exists a $\varrho_{Y_s}$-equivariant pluriharmonic map of logarithmic energy growth
\[
\tilde{u}_s: \widetilde Y_s \rightarrow  \mathcal{C}. 
\]
Denote the associated section by
$u_s: Y_s \rightarrow \widetilde Y_s \times_{\varrho_{Y_s}} \mathcal{C}
$
which can be viewed as a map 
$$
u_s: Y_s  \to \widetilde{X} \times_\varrho \mathcal{C}
$$
by (\ref{descendingmap}).

\begin{claim} \label{welldefined}   
For $q \in X$ and $s_1, s_2 \in \bU(q)$ with $\bU(q)$ defined in \eqref{eq:Uq}, we have
$
u_{s_1}(q) = u_{s_2}(q).
$
\end{claim}
\begin{proof}
For  $i=1,2$ and $q \in X$, we define $\bU(s_i,q)$ as follows:\\  
\[ 
\bU(s_i, q)= \{s \in \bU(q)\mid {\overline{Y}_{\!\! s}}\ \mbox{transversal to} \ {\overline{Y}}_{\!\! s_i} \ \mbox{and} \ \Sigma \cup \overline{Y}_{\!\! s}\cup {\overline{Y}}_{\!\! s_i} \ \mbox{is normal crossing}\}.
\] 
By \cref{lem:Bertini general},   $\bU(s_i, q)$ is a non-empty Zariski open  subset of $\bU(q)$. This implies $\bU(s_1, q) \cap \bU(s_2, q) \neq \varnothing$. 

Fix $s\in \bU(s_1, q) \cap \bU(s_2, q)$.
Let $\iota:Y_{\! s_i} \cap Y_s \rightarrow X$ be the inclusion map.  By \cref{thm:Lefschetz}, we know that $\pi_1(Y_{\! s_i} \cap Y_s)\to \pi_1(X)$,  $\pi_1(Y_{s})\to \pi_1(X)$ and $\pi_1(Y_{\! s_i})\to \pi_1(X)$ are all surjective.    By the inductive hypothesis, there exist pluriharmonic sections 
\[
u_s:  Y_s \rightarrow \widetilde Y_s \times_{\varrho_{Y_s}}  \mathcal{C} \ \mbox{ and } \ 
  u_{s_i}:  Y_{\! s_i} \rightarrow \widetilde Y_{s_i} \times_{\varrho_{Y_{s_i}}} \mathcal{C}.
\] 
which are of logarithmic energy growth with respect to $(\overline{Y}_{\!\! s},L|_{\overline{Y}_{\!\! s}})$ and  $(\overline{Y}_{\!\! s_i},L|_{\overline{Y}_{\!\! s_i}})$ respectively. 
By the uniqueness assertion of the inductive hypothesis, the restriction maps
\[
u_s|_{Y_{s_i} \cap Y_s}: Y_{s_i} \cap Y_s \rightarrow {\widetilde{Y_{s_i} \cap Y_s}} \times_{\varrho_{Y_{s_i} \cap Y_s}}  \mathcal{C}
\]
and 
\[
u_{s_i}|_{Y_{s_i} \cap Y_s}: Y_{s_i} \cap Y_s \rightarrow {\widetilde{Y_{s_i} \cap Y_s}} \times_{\varrho_{Y_{s_i} \cap Y_s}}  \mathcal{C}.
\]
defined in \eqref{restrictionsection} 
are in fact the same section.
Since $q \in Y_{s_i} \cap Y_s$, we conclude $u_{s_i}(q) =u_s(q)$. 
\end{proof}
Therefore, by \Cref{welldefined}, we can define
\[
u:X \rightarrow \widetilde{X} \times_\varrho \mathcal{C}, \ \ u(q):=u_s(q) \text{ for } s\in \bU(q).
\] 
To complete the inductive step, we are left to  show that  $u$  is  a pluriharmonic section of  logarithmic  energy growth with respect to $(\overline{X},L)$, and moreover is unique amongst such pluriharmonic sections of $\widetilde{X} \times_\varrho \mathcal C \rightarrow X$. 

\medspace

\noindent {\bf Step 2:} {\it  We prove that $u$ is locally harmonic with respect to  the Euclidean metric}. 
Let $T:=|L|^{\times (n-1)}$ and let $T^\circ$ be the Zariski open subset of $T$ defined in \cref{prop:Bertini}.  We first apply \cref{prop:Bertini} to prove the following:
	\begin{claim} \label{coornbhd}
		For every $x_0 \in X$, there exists a coordinate system  $(U; z_1, \dots, z_n)$ centered at $x_0$ such that for every $i=1,\ldots,n$ and every fixed  $w:=(z_1, \dots, z_{i-1}, z_{i+1}, \dots, z_n) \in \bD^{n-1}$, the disk 
		\[
		\bD_w:=\{ (z_1, \dots, z_{i-1}, z, z_{i+1}, \dots, z_n):  |z|<1\}
		\]
		is contained in some complete intersection $H_1 \cap \dots \cap H_{i-1} \cap H_{i+1} \cap \dots \cap H_n$, where $(H_1,\ldots,H_{i-1},H_{i+1},\ldots,H_n)\in T^\circ$.
	\end{claim} 
	\begin{proof}
		To prove Claim~\ref{coornbhd},  we  fix   $s_0 \in H^0(\overline X,L)$ such that $x_0 \notin (s_0 =0)$.    By \cref{item:tangent}, we can find   $s_1,\ldots,s_n\in H^0(\overline{X}, L)$ such that 
		\begin{enumerate}[label=(\alph*)]
			\item the hypersurfaces ${\overline{Y}}_{\! s_1},\ldots, {\overline{Y}}_{\!\!  s_n}$ are smooth and intersect transversely, where ${\overline{Y}}_{\!\! s_i}:=s_i^{-1}(0)$. 
			\item \label{item:normal} $\sum_{i=1}^{n}{\overline{Y}}_{\!\! s_i}+\Sigma$ is normal crossing.
			\item   $x_0\in   {\overline{Y}}_{\!\! s_1}\cap \ldots\cap {\overline{Y}}_{\!\!  s_n}$. 
		\end{enumerate}  
		Define
		$u_i:=\frac{s_i}{s_0}$ which is a global rational function of $\overline X$ and regular on some neighborhood $U$ of $x_0$.  After shrinking $U$ properly, the map
		\begin{align*}
			\varphi: U&\to \bC^n\\
			x&\mapsto (u_1(x), u_2(x),\ldots, u_n(x))
		\end{align*}
		is biholomorphic to its image $\varphi(U)=\mathbb D_\ep^n$.
		In particular, this defines an admissible coordinate system $(U; z_1, \dots, z_n; \varphi)$ centered at $x_0$. Fix any $i \in \{1, \dots, n\}$. For any $\zeta = (\zeta_1, \dots, \zeta_{i-1}, \zeta_{i+1}, \dots, \zeta_n) \in \mathbb{D}_\epsilon^{n-1}$, the disk
		\[
		\bD_\zeta := \left\{ (\zeta_1, \dots, \zeta_{i-1}, z, \zeta_{i+1}, \dots, \zeta_n) \in \mathbb{D}^n_\epsilon \mid |z| < \epsilon \right\}
		\]
		is contained in
		\[
		U \cap (z_1 - \zeta_1 = 0) \cap \dots \cap (z_{i-1} - \zeta_{i-1} = 0) \cap (z_{i+1} - \zeta_{i+1} = 0) \cap \dots \cap (z_n - \zeta_n = 0).
		\]
		After possibly shrinking $\varepsilon$, for any $\zeta \in U$, the divisor $E_j(\zeta) := U \cap (z_j - \zeta_j = 0)$ in $U$ coincides with $(s_j - \zeta_j s_0 = 0) \cap U$ for each $j \in \{1, \dots, i-1, i+1, \dots, n\}$, and we have
		$$
	\left(	\overline{Y}_{\!\! s_1 - \zeta_1 s_0},\cdots,\overline{Y}_{\!\! s_{i-1} - \zeta_{i-1} s_0},\overline{Y}_{\!\! s_{i+1} - \zeta_{i+1} s_0},\cdots,\overline{Y}_{\!\! s_{n} - \zeta_{n} s_0}\right)\in T^\circ.		$$
  By \Cref{item:normal}, we can shrink $\ep > 0$ further such that the divisor $\sum_{j \neq i} E_j(\zeta) + \Sigma \cap U$ remains a normal crossing divisor for any $\zeta \in \mathbb{D}_\ep^{n-1}$. \Cref{coornbhd} follows after composing $\varphi$ with the rescaling:
		\[
		\begin{aligned}
			\mathbb{D}_\ep^n &\to \mathbb{D}^n \\
			(z_1, \dots, z_n) &\mapsto \left( \frac{z_1}{\ep}, \dots, \frac{z_n}{\ep} \right).
		\end{aligned}
		\]
		Thus, for any $w = (z_1, \dots, z_{i-1}, z_{i+1}, \dots, z_n) \in \bD^{n-1}$, the disk $\bD_w$ is contained in the curve
		\[
		\mathcal{R}_w = Y_{s_1 - z_1 s_0} \cap \dots \cap Y_{s_{i-1} - z_{i-1} s_0} \cap Y_{s_{i+1} - z_{i+1} s_0} \cap \dots \cap Y_{s_n - z_n s_0}.
		\]
 The claim is thus proved. 
	\end{proof} 
	According to \cref{prop:Bertini,coornbhd}, for any $w = (z_1, \dots, z_{i-1}, z_{i+1}, \dots, z_n) \in \bD^{n-1}$,  we can define a holomorphic map
	\begin{align*}
		\nu:\bD^{n-1} &\to T^\circ\\
		w&\mapsto 	\left(	\overline{Y}_{\!\! s_1 - \zeta_1 s_0},\cdots,\overline{Y}_{\!\! s_{i-1} - \zeta_{i-1} s_0},\overline{Y}_{\!\! s_{i+1} - \zeta_{i+1} s_0},\cdots,\overline{Y}_{\!\! s_{n} - \zeta_{n} s_0}\right).
	\end{align*} 
Let $\pi: \mathscr{R} \to T$ be the universal family of complete intersection curves in $X$ as defined in \Cref{prop:Bertini}. Consider the base change $\mathscr{R}' := \mathscr{R} \times_T \bD^{n-1} \to \bD^{n-1}$ of $\mathscr{R}$ over $\bD^{n-1}$ via $\nu$. By \Cref{prop:Bertini}, the family $\mathscr{R}' \to \bD^{n-1}$ is topologically trivial, with $\mathcal{R}_w$ denoting the fiber over each $w \in \bD^{n-1}$.

	We now proceed with the proof that $u$ is harmonic with respect to the Euclidean metric on $\bD^n$.  
	The first step is to show that, after shrinking $U$ if necessary,  $u$ is Lipschitz continuous in $U$.
	Fix $w:=(z_1, \dots, z_{i-1}, z_{i+1}, \dots, z_n) \in \bD^{n-1}$.
	The restriction of $u$ to $\mathcal R_w$, denoted as $u_w$, is the unique harmonic section 
	\[
	u_w:  \mathcal R_w \rightarrow \tilde{\mathcal R}_w \times_{\varrho_{\mathcal R_w}}\cC
	\]
	where $\varrho_{\mathcal R_w}:= \varrho \circ (\iota_{\mathcal R_w})_*$  with $\iota_{\mathcal R_w}:\mathcal R_w \hookrightarrow X$  the inclusion map.
	We endow $\mathcal R_w$ with a conformal hyperbolic metric $h_w$.   In particular, $h_0 :=h_{(0,\dots,0)}$ is  the conformal hyperbolic metric on $\mathcal R_0:=\mathcal R_{(0,\dots,0)}$. 
	
	To estimate the local Lipschitz constant of $u_w$, we recall its construction in \cite{DMrs}.  The first step is to construct a locally Lipschitz $\varrho_{\mathcal R_w}$-equivariant map $k:\widetilde{\mathcal R}_w \rightarrow \cC$ using \cite[Proposition 2.6.1]{KS}.  Let $\gamma_1, \dots, \gamma_p$ be the generators of $\pi_1(\mathcal R_w)$ and let 
	\[
	\delta(P)=\max_{i=1, \dots, p} d(\varrho_{\mathcal R_w}P,P).
	\]
	Fix $P' \in \Delta(G)$ and let $\delta'=\delta(P')$.  The Lipschitz constant $L(x)$ of $k$ at $x$ is bounded by 
	\[
	L(x) \leq C \delta'
	\]
	where $C$ depends on the metric $h_w$.  As remarked in the last paragraph of the  proof of \cite[Proposition 2.6.1]{KS}, $C$ can be chosen independently of $h_w$ since $h_w$ has sectional curvature bounded from below.   
	
	In \cite{DMrs}, we construct a prototype map, i.e., a $\varrho_{\mathcal R_w}$-equivariant map $v:\tilde{\mathcal R}_w \rightarrow \Delta(G)$ that is equal to $k$ away from disks containing the punctures and equal to the Dirichlet solution on the punctured disks with boundary value given by $k$.  In this way, we construct a locally Lipschitz map $v$  with controlled energy towards the puncture.  The energy of $u$  away from the punctures is bounded by the energy of $v$ away from the punctures.  Therefore, the  local Lipschitz constant of $u_w$ depends on the local Lipschitz constant of  $v$ which in turn depends on the local Lipschitz constant of $k$.  In summary, the local Lipschitz constant of $u_w$ depends only on $\delta'$.

According to \cref{prop:Bertini}, $\mathscr{R}'\to \bD^{n-1}$ is a topologically trivial family such that  $\cR_w$  is the fiber over $w$.    Hence there exists a diffeomorphism $\phi_w: \mathcal R_w \rightarrow \mathcal R_0$ and 
	\[
	(\iota_{\mathcal R_w})_*= (\iota_{\mathcal R_0} )_* \circ (\phi_w)_*.
	\]  
	Thus, the Lipschitz constant of $u_w$ for $w\in \bD^{n-1}$ can be bounded uniformly. Thus, by shrinking $U$ if necessary, we may assume that the Lipschitz constant of $u$ along the disk $\bD_w$ for any $i \in \{1, \dots, n\}$ and $w \in \bD^{n-1}$ is uniformly bounded by a  constant $C$.
	Therefore, if  $z=(z_1,\dots, z_n), w=(w_1, \dots, w_n) \in U$,
	then
	\begin{align*}
		d(u(z), u(w)) & \leq   d(u(z_1,z_2, z_3, \dots, z_n), u(w_1, z_2, z_3, \dots, z_n))
		\\
		&   \ + d(u(w_1,z_2, z_3\dots, z_n), u(w_1, w_2, z_3 \dots, z_n))
	  \ \ + \dots   +
		\\
		&   d(u(w_1,w_2, w_3, \dots, w_{n-1}, z_n), u(w_1, w_2,  w_3, \dots, w_{n-1}, w_n))
		\\
		& \leq   C|z_1-w_1| + C|z_2-w_2| + \dots + C|z_n-w_n|.
	\end{align*}
	By Sedrakyan's inequality, $\frac{1}{n}\left( \sum_{i=1}^n |z_i-w_i| \right)^2 \leq \sum_{i=1}^n |z_i-w_i|^2$, and thus
	\[
	d^2 (u(z), u(w))  \leq C^2 n\left(   |z_1-w_1|^2 + .... + |z_n-w_n|^2 \right) = C^2 n |z-w|^2, \  \ \ \forall z,w \in U.
	\] 
	In other words,   $u$ is Lipschitz continuous  in $U$.  
	
		We now prove that $u$ is harmonic in  $U=\bD^n$ with respect to the Euclidean metric on $\bD^n$.
	For the proof, we will denumerate the $n$-number of disks  that make up $U$ and write
	\[
	U=\bD^n= \bD_1 \times \dots \times \bD_n.
	\]
	Here the notation is abusive and we emphasize that $\bD_i$ is not the disk in $\bC$ of radius $i$ as introduced in \Cref{sec:notation}.  
	Furthermore,
	we denote $\widehat{\bD_i}$ to be the product of $(n-1)$   disks obtained by removing the $i$-th disk from $ \bD_1 \times \dots \times \bD_n$; i.e.
	\[
	\widehat{\bD_i}:=\bD_1 \times \dots \times \bD_{i-1} \times \bD_{i+1} \times \dots \times \bD_n.
	\]
	Let ${\rm dvol}_0$ (resp.~$\widehat{{\rm dvol}_0}$) be the Euclidean volume  form of $\bD^n$ (resp.~$ \widehat{\bD_i}$).  
	We use the coordinate
	$$(z_1, \dots, z_n) \in \bD_1 \times \dots \times \bD_n \mbox{ and }
	z_i = x_i + \sqrt{-1} y_i \in \bD_i
	$$
	for $U$.  

	For any 	
	$w:=(z_2, \dots, z_n) \in \widehat{\bD_1}$,
	the restriction of $u$ to $\bD_1 \simeq \bD_1 \times \{w\}$, denoted as $u_w$,  is a harmonic map.  The energy density function  $|\nabla u_{w}|^2$ of $u_w$ is an $L^1$-function defined on $\bD_1 \simeq \bD_1 \times \{w\}$. 
	
	Following \cite[\S 1.9]{KS}, we have the identity
	\begin{eqnarray} 
		|\nabla u_w|^2& = & |u_*(\frac{\partial}{\partial x_1})|^2(\cdot,w)+ |u_*(\frac{\partial}{\partial y_1})|^2(\cdot,w)\label{ks1.9}
	\end{eqnarray}  
	as $L^1$ functions on $\bD_1 \simeq \bD_1 \times \{w\}$ for a.e.~$w \in \widehat{\bD_1}$.
	For the sake of completeness, we prove (\ref{ks1.9}) here:    For a fixed $(y_1,w)$,  let $I_{(y_1,w)}=\{x_1 \in \bR \mid  (x_1+\sqrt{-1}y_1,w) \in \bD^n\}$.  Following the notation of \cite[Theorem 1.9.6]{KS}, we denote   the energy density function of the 1-variable map $u|_{I_{(y_1,w)}}$ by  $|u_*(\frac{\partial}{\partial x_i})|^2$ and call it the $\frac{\partial}{\partial x_1}$-directional energy density function of $u$.  
	By
	\cite[Lemmas 1.9.1 \& 1.9.4]{KS}, 
	\[
	\lim_{\ep \rightarrow 0} \frac{d^2(u(x_1,y_1,w), u(x_1+\ep,y_1,w)}{\ep^2} =|u_*(\frac{\partial}{\partial x_i})|^2(z_1,w), 
	\ \ \mbox{ for a.e.}~x_1 \in I_{(y_1,w)}.
	\]
	Similarly, for a fixed $y_1$, let $I_{y_1}=\{x_1 \in \bR\mid x_1+\sqrt{-1}y_1 \in \bD_1\}$.  Following notation of \cite[Theorem 1.9.6]{KS}, we denote   the energy density function of the 1-variable map $u_w|_{I_{y_1}}$ by
	$|(u_w)_*(\frac{\partial}{\partial x_1})|^2$.  	By
	\cite[Lemmas 1.9.1 \& 1.9.4]{KS}, 
	\[
	\lim_{\ep \rightarrow 0} \frac{d^2(u_w(x_1,y_1), u_w(x_1+\ep,y_1)}{\ep^2} =|(u_w)_*(\frac{\partial}{\partial x_i})|^2(z_1), 
	\ \ \mbox{ for a.e.}~x \in I_{y_1}, \ \mbox{and a.e.}~w \in \bD_i.
	\]
	Since $u(x_1,y_1,w)=u_w(x_1,y_1)$ and $u(x_1+\ep,y_1,w)=u_w(x_1+\ep,y_1)$, we conclude that 
	\[
	|u_*(\frac{\partial}{\partial x_i}) |^2(z_1,w)=|(u_w)_*(\frac{\partial}{\partial x_i})|^2(z_1) \mbox{ as $L^1$-functions}  \mbox{ for a.e.}~w \in \widehat{\bD_1}.
	\] 
	Similarly, 
	\[
	|u_*(\frac{\partial}{\partial y_i})|^2(z_1,w)=|(u_w)_*(\frac{\partial}{\partial y_i})|^2(w) \mbox{ as $L^1$-functions}  \mbox{ for a.e.}~w \in \widehat{\bD_1}.
	\]
	By \cite[Theorem 2.3.2 (2.3vi)]{KS}, 
\begin{align*}
	 	|\nabla u_w|^2 = |  (u_w)_*(\frac{\partial}{\partial x})|^2+ |(u_w)_*(\frac{\partial}{\partial y})|^2.
\end{align*} 
	Thus, (\ref{ks1.9}) follows  from the above three identities.

	For notational simplicity, for each $i \in \{1, \ldots, n\}$, we will now denote
	\begin{align}\label{eq:notion}
		\left| \frac{\partial u}{\partial x_i} \right|^2 := |u_*(\frac{\partial}{\partial x_i})|^2, \quad \left| \frac{\partial u}{\partial y_i} \right|^2 := |u_*(\frac{\partial}{\partial y_i})|^2.
	\end{align}
	
	Let $v$ be the unique harmonic map in $U$ with boundary values equal to those of $u$. We have a similar identity to \eqref{ks1.9}. More precisely, for any $i \in \{1, \dots, n\}$ and $w \in \widehat{\bD_i}$, we have
	\begin{align*}
		|\nabla v_w|^2 = |u_*(\frac{\partial}{\partial x_1})|^2(\cdot, w) + |v_*(\frac{\partial}{\partial y_1})|^2(\cdot, w)
	\end{align*}
	as $L^1$ functions on $\bD_i \simeq \bD_i \times \{w\}$ for a.e.~$w \in \widehat{\bD_i}$. We shall use the same notation for $\nu$ as in \eqref{eq:notion}.

	Applying  the Fubini-Tonelli Theorem, we express $E^v[U]$ and $E^u[U]$ as a sum of $n$-terms as follows:
	\begin{eqnarray*}
		E^v[U] & = & \sum_{i=1}^n   \int_{\bD^n} \left| \frac{\partial v}{\partial x_i} \right|^2 + \left| \frac{\partial v}{\partial y_i} \right|^2{\rm dvol}_0 \\
		& = &  \sum_{i=1}^n \int_{\widehat{\bD_i}} \left( \int_{\bD_i}  \left| \frac{\partial v}{\partial x_i} \right|^2 +\left| \frac{\partial v}{\partial y_i} \right|^2 \frac{idz_i \wedge d\bar z_i}{2}\right)  \widehat{\rm dvol}_0
		\\
		& = &  \sum_{i=1}^n \int_{\widehat{\bD_i}} \left( \int_{\bD_i} |\nabla v_w|^2 \frac{idz_i \wedge d\bar z_i}{2}\right) \widehat{\rm dvol}_0,
	\end{eqnarray*}
	and 
	\begin{eqnarray*}
		E^u[U] & = &   \sum_{i=1}^n \int_{\bD^n} \left| \frac{\partial u}{\partial x_i} \right|^2 + \left| \frac{\partial u}{\partial y_i} \right|^2{\rm dvol}_0 \\
		& = &  \sum_{i=1}^n \int_{\widehat{\bD_i}} \left( \int_{\bD_i}  \left| \frac{\partial u}{\partial x_i} \right|^2 +\left| \frac{\partial u}{\partial y_i} \right|^2 \frac{idz_i \wedge d\bar z_i}{2}\right) \widehat{\rm dvol}_0
		\\
		& = &  \sum_{i=1}^n \int_{\widehat{\bD_i}} \left( \int_{\bD_i} |\nabla u_w|^2 \frac{idz_i \wedge d\bar z_i}{2}\right) \widehat{\rm dvol}_0.
	\end{eqnarray*}
	Assume $E^v[U] < E^u[U]$.  Then    there exists some $i \in \{1, \dots, n\}$ such that
	\[
	\int_{\widehat{\bD_i}} \left( \int_{\bD_i} |\nabla v_w|^2 \frac{idz_i \wedge d\bar z_i}{2}\right) \widehat{\rm dvol}_0 < \int_{\widehat{\bD_i}} \left( \int_{\bD_i} |\nabla u_w|^2 \frac{idz_i \wedge d\bar z_i}{2}\right) \widehat{\rm dvol}_0.
	\]
	Thus, we conclude that there exists a subset $Z$ of $\widehat{\bD_i}$ with positive Lebesgue measure such that for any
	\[
	w_0:=(c_1, \dots, c_{i-1}, c_{i+1}, \dots, c_n) \in Z,
	\]
we have
	\begin{eqnarray*}	
		\int_{\bD_i} |\nabla v_{w_0}|^2 \frac{idz \wedge d\bar z}{2} 
		<    \int_{\bD_i}  |\nabla u_{w_0}|^2 \frac{idz \wedge d\bar z}{2}.
	\end{eqnarray*}
	This contradicts that $u_{w_0}$ is a harmonic map.  Thus, $E^u[U] = E^v[U]$ and $u|_{U}=v$ is  harmonic with respect to the Euclidean metric on $\bD^n$.

	\medspace

\medspace

\noindent {\bf Step 3:} {\it  $u$ is pluriharmonic.} Since $u$ is locally  harmonic with respect to some Euclidean metric,  the set ${\mathcal S}(u)$ of singular points of $u$  is  a closed  subset of $X$ of Hausdorff codimension  by \cref{gs}.

 Let  $p \in X\backslash \cS(u)$  and $\mathcal{P} \subset T_p^{1,0}(X)$ be any complex 1-dimensional subspace. By \cref{item:tangent}, there exists  some  \((H_1,\dots, H_{n-1})\in T^{\circ}\) such that $p\in H_1\cap\ldots\cap H_{n-1}$ and $H_1\cap\ldots\cap H_{n-1}$ is tangent to $\mathcal{P}$.   Write $\overline{\cR}:=H_1\cap\ldots\cap H_{n-1}$ and $\cR:=\overline{\cR}\backslash \Sigma$.   By the construction of $u$, its restriction $u|_{\cR}$ is the  unique pluriharmonic section $u_{\cR}:\cR\to \widetilde{\cR}\times_{\varrho_{\cR}}\cC$ of logarithmic energy growth. Thus,  we have
 \[
 \partial \bar \partial|_\mathcal{P} u(p)=\partial \bar \partial u_\cR(p)=0.
 \]
Since $p$ is an arbitrary point of $X \backslash \mathcal S(u)$, this proves that $ \partial \bar \partial u=0$ over $X\backslash \cS(u)$.  By \cref{phequiv},  $u$  is  pluriharmonic of logarithmic energy growth with respect to $(\overline{X},L)$.

\medspace

\noindent {\bf Step 4:} {\it $u$ is harmonic with respect to any K\"ahler metric $\omega$ on $X$.} Since harmonicity is a local property, it is sufficient to prove this claim locally. Pick any $x_0 \in X$. Let $(U; z_1, \ldots, z_n)$ be the coordinate neighborhood of $x_0$ introduced in \Cref{coornbhd}. Since $u|_U$ is harmonic with respect to the Euclidean metric on $U$, the singular subset $\cS(u)$ has Hausdorff codimension at least two. Let $v: U \rightarrow \cC$ be the unique harmonic map in $U$ with respect to $\omega$ with boundary values equal to those of $u$. Since $u$ is pluriharmonic, the restriction $u|_{U \backslash \cS(u)}$ is harmonic with respect to the metric $\omega$. Thus, the function $d^2(u,v)$ is subharmonic when restricted to $\cR(u) \cap \cR(v)$. Since $d^2(u,v)$ is bounded and $\cS(u) \cup \cS(v)$ is a closed subset in $U$ with Hausdorff codimension at least two, $d^2(u,v)$ is weakly subharmonic. By the maximum principle, and the fact that $d^2(u,v) = 0$ on $\partial U$, it follows that $d^2(u,v) = 0$ on $U$. This proves $u = v$, meaning that $u$ is harmonic with respect to $\omega$.  

\medspace

\noindent {\bf Step 5:} {\it $u$ is unique.} Let $\widetilde{v}: \widetilde{X} \rightarrow \mathcal{C}$ be another $\varrho$-equivariant pluriharmonic map of logarithmic energy growth, and $v: X \rightarrow \widetilde{X} \times_\varrho \mathcal{C}$ be its corresponding section (cf.~\Cref{sec:ems}). For $q \in X$, let $s \in \bU(q)$. The restriction $v_{Y_s}$ of $v$ is a pluriharmonic section of logarithmic energy growth with respect to $(\overline{Y}_{\!\! s}, L|_{\overline{Y}_{\!\! s}})$. By the uniqueness assertion of the inductive hypothesis, we conclude that $u_{Y_s} = v_{Y_s}$. Since $q$ is an arbitrary point in $X$, we conclude that $u = v$. This proves the uniqueness of $u$. 
\end{proof}

\section{Energy estimate  for pluriharmonic maps into Euclidean buildings}\label{sec:estimate}
In this section we will complete the proof of \cref{GS}.
\subsection{Local energy  estimate at infinity}
 In this subsection we prove \cref{main:energy} (cf. \cref{logestimate}). Let $X$, $\overline{X}$, $L$, $\Sigma$ and $\varrho$ be as   in \cref{thm:harmonicmaps}.   Set $T:=|L|^{\times (n-1)}$ and let $T^\circ$ be the Zariski open subset of $T$ defined in \cref{prop:Bertini}.
\begin{lem} \label{goodcoord}
Any smooth point $x_0$ in the divisor $\Sigma$ has an admissible coordinate neighborhood $(U;z_1,\ldots,z_n)$ centered at $x_0$ with $U\cap \Sigma=(z_1=0)$  such that for  any $z_*=(z_2,\ldots,z_n)\in \mathbb D^{n-1}$, the transverse disk $z\mapsto (z,z_*)$ is contained   in some complete intersection $\overline{\cR}_{z_*}:=H_1 \cap \dots \cap   H_{n-1}$, where $(H_1,\ldots, H_{n-1})\in T^\circ$.
\end{lem}
\begin{proof}
Since $x_0\in \Sigma$ is a smooth point, by \cref{prop:Bertini}, we can choose  $s_2,\ldots,s_n\in H^0(\overline{X}, L)$ such that 
\begin{enumerate}[label=(\alph*)]
\item $({\overline{Y}}_{\! s_2},\ldots, {\overline{Y}}_{\!\!  s_n})\in T^\circ$. In particular,  the hypersurfaces ${\overline{Y}}_{\! s_2},\ldots, {\overline{Y}}_{\!\!  s_n}$ are smooth, where ${\overline{Y}}_{\!\! s_i}:=s_i^{-1}(0)$. 
\item  \label{item:transverse2}The divisor $\sum_{i=2}^{n}{\overline{Y}}_{\!\! s_i}+\Sigma$ is normal crossing.
\item   $x_0\in   {\overline{Y}}_{\!\! s_2}\cap \ldots\cap {\overline{Y}}_{\!\!  s_n}$. 
\end{enumerate}  
Pick some $s_1\in H^0(\overline{X}, L)$ such that $x_0 \notin(s_1=0)$. Let $u_i:=\frac{s_i}{s_1}$. Then for any $i\in \{2,\ldots,m\}$, $u_i$ is a rational function on $\overline{X}$ that is  regular on some neighborhood $U$ of $x_0$. After shrinking $U$ if necessary,  we can assume that  there is a holomorphic function $v\in \sO(U)$ such that $dv(x_0)\neq 0$ and  $\Sigma\cap U=(v=0)$ . By \Cref{item:transverse2}, one has  $dv\wedge du_2\wedge\ldots\wedge du_n(x_0)\neq 0$. After possibly shrinking $U$, we may assume that
\begin{enumerate}[label=(\arabic*)]
\item $dv\wedge du_2\wedge\ldots\wedge du_n(x)\neq 0$ for all $x\in U$;
\item The map \begin{align*}
	\varphi: U&\to \bC^n\\
	x&\mapsto (v(x), u_2(x),\ldots, u_n(x))
\end{align*}
is biholomorphic to its image $\varphi(U)=\mathbb D_\ep^n$.
\end{enumerate}
Thus, the map $\varphi$ defines an admissible coordinate neighborhood $(U;z_1,\ldots,z_n;\varphi)$ of $U$ centering at $x_0$.  For any $\zeta:=(\zeta_2,\ldots,\zeta_n)\in \mathbb D_\ep^{n-1}$, the transverse disk $$\bD_\zeta:=\{(z,\zeta_2,\ldots,\zeta_n)\in \mathbb D^n_{\ep}\mid |z|<\ep\}$$ is contained in $U\cap (u_2-\zeta_2=0)\cap \ldots \cap(u_n-\zeta_n)$. The later is contained in $(s_2-\zeta_2s_1=0)\cap\ldots\cap (s_n-\zeta_n s_1=0)$.   Since $T^\circ$ is Zariski open in $T$,  one can shrink $\ep$ such that 
$$
\left(	\overline{Y}_{\!\! s_2 - \zeta_2 s_0},\ldots,\overline{Y}_{\!\! s_{n} - \zeta_{n} s_0}\right)\in T^\circ 	$$ 
for each $\zeta\in\bD_\ep^{n-1}$.  The lemma follows after we compose $\varphi$ with the rescaling 
\begin{align*}
\mathbb D_\ep^n&\to \mathbb D^n\\
(z_1,\ldots,z_n)&\mapsto (\frac{z_1}{\ep},\ldots,\frac{z_n}{\ep}).
\end{align*}
\end{proof}


\begin{proposition} \label{logestimate}
Let $X$, $\overline{X}$, $L$, $\Sigma$ and $\varrho$ be as   in \cref{thm:harmonicmaps}. Let $\tilde{u} : \widetilde{X} \to \Delta(G)$ be the $\varrho$-equivariant pluriharmonic map with logarithmic energy growth with respect to $(\overline{X}, L)$ constructed in \cref{thm:harmonicmaps}, and let $u$ be its corresponding section. For any smooth point $x_0 \in \Sigma$ and an admissible coordinate neighborhood $(U; z_1, \ldots, z_n)$ centered at $x_0$, as constructed in \cref{goodcoord}, there exists a constant $C > 0$ such that 
\begin{equation} \label{mp2'}
\left|\frac{\partial u}{\partial z_j}(z_1, z_2, \dots, z_n) \right|^2\leq \Lambda^2 \ \mbox{ for any } (z_1,\ldots,z_n)  \in  \bD^*_{\frac{1}{2}}\times \bD^{n-1}_{\frac{1}{2}},  \ \ \forall j=2, \dots, n, \
\end{equation}
\begin{equation} \label{mp3'}
0  \leq \int_{\bD_{r,\frac{1}{2}}\times \bD^{n-1}_{\frac{1}{2}}} \left|\frac{\partial  {u}}{\partial z_1} (z_1, z_2, \dots, z_n)\right|^2 {\rm dvol}_\omega  - \frac{L^2_\gamma}{2\pi} \log \frac{1}{r} \cdot \mathsf{Vol}\left(\mathbb D_{\frac{1}{2}}^{n-1} \right)\leq C, \ \ \ \forall \ 0<r< \frac{1}{2},
\end{equation} 
\begin{align} \label{eq:energy}
- \frac{L^2_\gamma}{2\pi} \log r \cdot \mathsf{Vol}\left(\mathbb D_{\frac{1}{2}}^{n-1} \right)\leq  \int_{\bD_{r,\frac{1}{2}}\times \bD^{n-1}_{\frac{1}{2}}} |\nabla u|_\omega^2 {\rm dvol}_\omega \leq - \frac{L^2_\gamma}{2\pi} \log r  \cdot \mathsf{Vol}\left(\mathbb D_{\frac{1}{2}}^{n-1} \right)+C, \ \ \ \forall \ 0<r< \frac{1}{2}.
\end{align}
\begin{align} \label{eq:energy3}
	- \frac{L^2_\gamma}{2\pi} \log r \cdot \mathsf{Vol}\left(\mathbb D_{\frac{1}{2}}^{n-1} \right)\leq  \int_{\bD_{r,\frac{1}{2}}\times \bD^{n-1}_{\frac{1}{2}}}  |\nabla u|_{\omega_P}^2 d\vol_{\omega_P} \leq - \frac{L^2_\gamma}{2\pi} \log r  \cdot \mathsf{Vol}\left(\mathbb D_{\frac{1}{2}}^{n-1} \right)+C,   \ \ \ \forall \ 0<r< \frac{1}{2}.
\end{align}
 Here
 \begin{itemize}
 	\item  $\omega:=\sum_{i=1}^{n}\frac{\sqrt{-1}}{2}dz_i\wedge d\bar{z}_i$  (resp. $\omega_P$ ) is the standard Euclidean metric (resp. Poincar\'e-type metric defined in \eqref{eq: Poin})  on $U^*:=U\backslash \Sigma$,   $d\vol_\omega$  (resp. $d\vol_{\omega_P}$) is the  volume form of $\omega$ (resp. $\omega_P$) on $U^*$, and $\mathsf{Vol}\left(\mathbb D_{\frac{1}{2}}^{n-1} \right)$   is the Euclidean volume of $\mathbb D_{\frac{1}{2}}^{n-1}$. 
 	\item  $\gamma \in \pi_1(X)$ is  the element corresponding to the loop  $\theta \mapsto (\frac{1}{2} e^{\sqrt{-1} \theta}, 0,\ldots,0)$ around the irreducible component  $\Sigma$ containing $x_0$;
 	\item    $L_\gamma$ is the translation length of $\varrho(\gamma)$ defined in \Cref{def:translation}.  
 \end{itemize}   
Moreover, the above energy  $\int_{\bD_{r,\frac{1}{2}}\times \bD^{n-1}_{\frac{1}{2}}} |\nabla u|^2 d\vol_{\omega}$ is \emph{finite}  provided that $\varrho(\gamma)\in G(K)$ is quasi-unipotent.
\end{proposition}

\begin{proof}
  In \cref{thm:harmonicmaps}, we prove that  $\tilde{u}$ is   harmonic with respect to any choice of a  K\"ahler metric on $\widetilde{X}$. 
By \cref{thm:Lcon},  $\tilde{u}$ is locally Lipschitz continuous with respect to the distance function on  $\widetilde{X}$ induced by the metric $\omega$.  
  Let $\Lambda>0$ be the Lipschitz constant of $\tilde{u}$ in $\pi_X^{-1}(\db\bD_{\frac{1}{2}}\times \overline{\bD_{\frac{1}{2}}}\times\cdots\times \overline{\bD_{\frac{1}{2}}})$.

Fix $z_*:=(z_{2*}, \dots,z_{n*})$, $w_*:=(w_{2*}, \dots,w_{n*}) \in \mathbb D_{\frac{1}{2}}^{n-1}$.   Then
\[
\delta^2_{z_*,w_*}(z):=d^2(\tilde{u}(z,z_*),  \tilde{u} (z,w_*))\leq \Lambda^2 |z_*-w_*|^2 \mbox{ \ for \  } |z|=\frac{1}{2}.
\] 
Let $\overline{\mathcal R}_{z_*}$ and $\overline{\mathcal R}_{w_*}$ be the  complete intersection curves in   Lemma~\ref{goodcoord}.  Denote $\cR_{z_*}:=\overline{\mathcal R}_{z_*}\cap X$ and $\cR_{w_*}:=\overline{\mathcal R}_{w_*}\cap X$.   Let $u_{ {\mathcal R}_{z_*}}$ and $u_{ {\mathcal R}_{w_*}}$  be induced maps as in (\ref{restrictionsection}) of the compositions of  $u$ and the inclusion maps $\cR_{z_*} \hookrightarrow X$ and $\cR_{z_*} \hookrightarrow X$ respectively.  Let   $\tilde{u}_{{\mathcal R}_{z_*}}$ and $\tilde{u}_{{\mathcal R}_{w_*}}$ be the corresponding equivariant maps from the universal covers  to $\Delta(G)$ as in   \eqref{restrictionsection}.
 By the construction of $\tilde{u}$ in \cref{thm:harmonicmaps},     $\tilde{u}_{ {\mathcal R}_{z_*}}$ and $\tilde{u}_{ {\mathcal R}_{w_*}}$ are harmonic maps of  logarithmic growth. Hence the function $\delta^2_{z_*,w_*}(z) =d^2(u(z_1,z_*), u(z_1,w_*))$ is a continuout subharmonic function satisfying 
\[
\lim_{|z| \rightarrow 0} \delta^2_{z_*,w_*}(z) +\ep \log |z| =-\infty.
\]
Thus, an argument used to prove  (\ref{maxprin}) also proves  \begin{equation} \label{maxprin2}
	\delta^2_{z_*,w_*}(z) \ \leq  \Lambda^2 |z_*-w_*|^2  \quad \forall z \in \bD^*_\frac{1}{2}. 
\end{equation}
It yields \eqref{mp2'}.

By \cref{thm:harmonicmaps}, $\tilde{u}$  has logarithmic energy growth with respect to $(\overline{X},L)$. By  \cref{def:log energy},   for  any fixed  $z_* \in \mathbb D_{\frac{1}{2}}^{n-1}$,  there exists a constant $C>0$ such that  we have
 \begin{equation} \label{ondisks}
- \frac{L^2_\gamma}{2\pi} \log r \leq E^{\tilde{u}_{ {\mathcal R}_{z_*}}}[\mathbb D_{r,\frac{1}{2}}] \leq  -\frac{L^2_\gamma}{2\pi} \log r+C
\end{equation}
for any $r\in (0,\frac{1}{2})$. 
Such constant  $C$ in (\ref{ondisks}) depends only on $L_\gamma$ and the Lipschitz estimate of $\tilde{u}_{ {\mathcal R}_{z_*}}$ on $\partial \mathbb D_{\frac{1}{2}}$.  Thus, $C$  is uniform for any $z_* \in \mathbb D_{\frac{1}{2}}^{n-1}$. 
Integrating (\ref{ondisks})
over $z_* \in \mathbb D_{\frac{1}{2}}^{n-1}$ while noting
\begin{align}\label{eq:energyuniform}
E^{\tilde{u}_{ {\mathcal R}_{z_*}}}[\mathbb D_{r,\frac{1}{2}}] = \int_{\mathbb D_{r,\frac{1}{2}} } \left|\frac{\partial u}{\partial z_1}\right|^2 (z,z_*) \  \frac{\sqrt{-1}dz \wedge d\bar z}{2},  
\end{align}
we conclude (\ref{mp3'}).

Since
\begin{eqnarray*}
\int_{\bD_{r,\frac{1}{2}}\times \bD^{n-1}_{\frac{1}{2}}}|\nabla u|^2  d\vol_{\omega} & = &  \int_{\bD_{r,\frac{1}{2}}\times \bD^{n-1}_{\frac{1}{2}}}\left(   \left|\frac{\partial u}{\partial z_1}\right|^2 + \sum_{j=2}^n  \left|\frac{\partial u}{\partial z_j}\right|^2
\right) d\vol_{\omega},
\end{eqnarray*}
the  assertion \eqref{eq:energy} follows from (\ref{mp2'}) and (\ref{mp3'}).  

Consider the Poincar\'e-type metric
\begin{align*} 
	\omega_P=\frac{\sqrt{-1}dz_1\wedge d\overline{z}_1}{|z_1|^2(\log |z_1|^2)^2}+\sum_{k=2}^{n} \sqrt{-1}dz_k\wedge d\overline{z}_k.
\end{align*} 
Denote by  $(P_{i\bar j})$ and $(P^{i\bar j})$ the components of this metric tensor  and its inverse. 
Note that
\begin{eqnarray*}
\int_{\bD_{r,\frac{1}{2}}\times \bD^{n-1}_{\frac{1}{2}}} |\nabla u|^2_{\omega_P} d\vol_{\omega_P} & = & \int_{\bD_{r,\frac{1}{2}}\times \bD^{n-1}_{\frac{1}{2}}}\left( P^{1\bar 1} \left|\frac{\partial u}{\partial z_1}\right|^2 + \sum_{j=2}^n P^{j\bar j}\left|\frac{\partial u}{\partial z_j}\right|^2
	\right) d\vol_{\omega_P}
	\\
	& = &  \int_{\bD_{r,\frac{1}{2}}\times \bD^{n-1}_{\frac{1}{2}}}\left( \left|\frac{\partial u}{\partial z_1}\right|^2  + \frac{1}{|z_1|^2(\log |z_1|^2)^2} \sum_{j=2}^n \left|\frac{\partial u}{\partial z_j}\right|^2
	\right) d\vol_{\omega_0}.
\end{eqnarray*}
Then \eqref{eq:energy3} follows from (\ref{mp2'}) and (\ref{mp3'}).

To prove the last claim, it then suffices to show that $L_\gamma=0$. Since the finiteness of local energy is preserved under finite unramified  covers, we can assume that $\varrho(\gamma)$ is unipotent.  Then there exists a Borel subgroup $B$  of $G$  such that $\varrho(\gamma)\in U(K)$, where $U$ is the unipotent radical of $B$.  Note that $U(K)$   fixes a sector-germ of the standard apartment $A$, which means that there exists a Weyl chamber $C^v$ of the apartment $A$ such that if $u$ in $U(K)$, then $u$ fixes $x+C^v$, for some $x$ in $A$. In particular, $\varrho(\gamma)$ fixes a  point  $y\in A$.  Consider the minimal closed convex $\varrho(\pi_1(X))$-invariant subset $\cC\subset \Delta(G)$ constructed in \cref{lem:ptinfty}.  By \cref{lem:proj}, the closest point projection map $\Pi: \Delta(G) \rightarrow \mathcal{C}$ is a  $G$-equivariant map, which implies that  $\varrho(\gamma)\Pi(y)=\Pi(\varrho(\gamma)y)=\Pi(y)$. By \eqref{eq:translation}, this implies that $L_\gamma=0$. The proposition is proved. 
\end{proof}

\subsection{Logarithmic energy growth (II)} \label{logenergy(II)}
In this subsection we  complete the proof of \cref{GS}. 
We shall give a more intrinsic definition of logarithmic energy growth than \cref{logenenergy} (cf. \cref{def:log energy}).  
 \begin{lem} \label{lem:generic}
 	Let $(\overline{X},\Sigma)$ be a log smooth pair, $L$ be a line bundle on $\overline{X}$. Assume that $V\subset |L|$ is a linear system which is base-point-free.  Then a generic hypersurface   $H$ in $V$ is smooth and $H+\Sigma$ is also simple normal crossing.   
 \end{lem}
 \begin{proof}
 	We write $\Sigma=\sum_{i=1}^{m}\Sigma_i$ into sum of irreducible components. For   $I\subset \{1,\ldots,m\}$, denote by  $\Sigma_I:= \bigcap_{i_k \in I} \Sigma_{i_k}$ which is a closed smooth subvariety of $\overline{X}$.    Then by the Bertini theorem, for each $I$ with $\dim \Sigma_I\geqslant 1$, there is a Zariski open set $V_I$ of $V$ such that every hypersurface $H\in V_I$ satisfies that $H$ and  $H\cap \Sigma_I$ are both smooth. Denote by $V':=\bigcap_{I}V_I$ where $I$ ranges over all subsets of $ \{1,\ldots,m\}$ such that $\dim \Sigma_I\geqslant 1$. Then $V'$ is a Zariski dense open set of $V$. It follows that every hypersurface $H\in V'$ is  smooth   and $H\cap \Sigma_I$ is  smooth for each   $\Sigma_I$  with $\dim \Sigma_I\geqslant 1$. This implies that $H\cup \Sigma$ is also simple normal crossing.  
 \end{proof}
\begin{lem} \label{thm:energy estimate}
Let $X$, $\overline{X}$, $L$, $\Sigma$ and $\varrho$ be as   in \cref{thm:harmonicmaps}. Let $\tilde{u} : \widetilde{X} \to \Delta(G)$ be the $\varrho$-equivariant pluriharmonic map with logarithmic energy growth with respect to $(\overline{X}, L)$ constructed in \cref{thm:harmonicmaps}, and let $u$ be its corresponding section. Choose any smooth point $x_0 \in \Sigma$. Let $(U;w_1,\ldots,w_n)$ be \emph{any} admissible coordinate neighborhood centered at $p$ such that $U\cap \Sigma=(w_1=0)$. Then     there exists a positive constant $C$ such that  for any $0<r<\frac{1}{2}$, and any $w_*:=(w_2,\ldots,w_n)\in \bD_{\frac{1}{2}}^{n-1}$,  one has
	\begin{equation} \label{Cbddef3} 
		0 \leq \int \limits_{{\mathbb D}_{r,\frac{1}{2}}} 
		\left| \frac{\partial u}{\partial w_1}(w_1,w_*) \right|^2 \frac{idw_1 \wedge d\bar w_1}{2}
		- \frac{L^2_\gamma}{2\pi}  \log \frac{1}{r}  \leq C.
	\end{equation}
	Here  $L_\gamma$ is the translation length of $\varrho(\gamma)$ with $\gamma \in \pi_1(X)$ corresponding to the loop $\theta \mapsto (\frac{1}{2} e^{i\theta},0,\ldots,0)$.  
\end{lem}

\begin{proof}
By Lemma~\ref{goodcoord}, we can choose  an admissible coordinate neighborhood $(V; z_1, \dots, z_n)$ centered at $p$ satisfying the properties therein, such that  
  $z_1=w_1$.  After shrinking $U$ if necessary, we may assume that  there is a constant $C>0$ such that  for any $j\in \{2,\ldots,n\}$, we have
  $$
  |  \frac{\partial z_j}{\partial w_1}(w_1,w_*)|\leq C
  $$ 
  for any $(w_1,w_*)\in U$.
  Then  by   \eqref{mp2'} and
  \begin{align*}
\frac{\partial u}{\partial w_1}(w_1,w_*) =\frac{\partial u}{\partial z_1}(z_1,z_*) \frac{\partial z_1}{\partial w_1}+\sum_{j=2}^{n}\frac{\partial u}{\partial z_j}(z_1,z_*) \frac{\partial z_j}{\partial w_1}=\frac{\partial u}{\partial z_1}(z_1,z_*)  +\sum_{j=2}^{n}\frac{\partial u}{\partial z_j}(z_1,z_*) \frac{\partial z_j}{\partial w_1},
  \end{align*}
     there is a constant $C_2>0$ such that   
   $$
   |\frac{\partial u}{\partial w_1}(w_1,w_*)|\leq |\frac{\partial u}{\partial z_1}(z_1,z_*) |+C,
   $$
     for any $(w_1,w_*)\in U$.
  Thus, \eqref{Cbddef3} follows from the same argument used in the proof of \cref{lem:ks}~(ii), replacing (\ref{fromjoel}) and (\ref{esjoel1}) with (\ref{mp2'}) and (\ref{mp3'}). We leave the details to the reader. 
\end{proof}

\begin{lem}\label{prop:unicity}
Let $X$, $\overline{X}$, $L$, $\Sigma$, and $\varrho$ be as in \cref{thm:harmonicmaps}. Let $\tilde{u} : \widetilde{X} \to \Delta(G)$ be the $\varrho$-equivariant pluriharmonic map with logarithmic energy growth with respect to $(\overline{X}, L)$ constructed in \cref{thm:harmonicmaps}. Assume that $\mu : \overline{X}_1 \to \overline{X}$ is a birational morphism such that $\mu|_{\mu^{-1}(X)} : \mu^{-1}(X) \to X$ is an isomorphism and $\Sigma_1 := \overline{X}_1 \backslash \mu^{-1}(X)$ is also a simple normal crossing divisor. If $L_1$ is a  sufficiently ample line bundle on $\overline{X}_1$, then $\tilde{u}$ also has logarithmic energy growth with respect to $(\overline{X}_1, L_1)$. 
\end{lem}
\begin{proof} 
	Consider the linear system $|\mu^*L|$ on $\overline{X}_1$. It is a  free linear system as $L$ is very ample.  Note that $$H^0(\overline{X}_1, \mu^*L)=H^0(\overline{X}, \mu_*(\mu^*L))=H^0(\overline{X}, L\otimes \mu_*(\sO_{\overline{X}_1}))=H^0(\overline{X},L),$$
	where the second equality is due to projection formula and the last equality follows from Zariski's main theorem  $\mu_*\sO_{\overline{X}_1}=\sO_{\overline{X}}$. It follows that
	\begin{align}\label{eq:isomorphism}
		\mu^*:H^0(\overline{X},L)\to H^0(\overline{X}_1, \mu^*L)   
	\end{align}
	is an isomorphism.  
	
Denote $T:=|L|^{\times(n-1)}$ and let $T^\circ$ be the Zariski open subset of $T$ constructed in \cref{prop:Bertini}.  Similarly, we define $T_1:=|\mu^*L|^{\times(n-1)}$  and let $T_1^\circ$ be the Zariski open subset of $T$  such that, 
for every  \((H_1,\dots, H_{n-1})\in T_1^{\circ}\), the hypersurfaces \(H_1,\dots, H_{n-1}\)  are smooth, and the divisor \(H_1+\dots+ H_{n-1}+\Sigma_1\)  is simple normal crossing.  By   \cref{lem:generic},  one can show that $T_1^\circ$ is a non-empty Zariski open subset of $T_1$. The isomorphism \eqref{eq:isomorphism} induces an isomorphism $i:T_1\to T$. Denote $T^{\circ\circ}:=T^\circ\cap i(T_1^\circ)$. It is a non-empty Zariski open subset of $T$. Moreover, by \cref{lem:generic} along with the same arguments in the proof of \cref{prop:Bertini}, for any $x_0\in X$, there exists \((H_1,\dots, H_{n-1})\in T^{\circ\circ}\) such that 
$$
x_0\in \overline{\cR}:=H_1\cap\cdots\cap H_{n-1}. 
$$
Denote $\cR:=\overline{\cR}\backslash \Sigma$. By \cref{thm:harmonicmaps},  $\tilde{u}_{\cR}:\widetilde{\cR}\to \cC$  is a $\varrho_{\cR}$-equivariant harmonic map with  logarithmic energy growth.   
	
	 By our construction of $T^{\circ\circ}$, it follows that  $\mu^*H_1,\ldots,\mu^*H_{n-1}$ are all smooth, and $\sum_{j=1}^{n-1}\mu^*H_j+\Sigma_1$ is simple normal crossing.  Thus, $\overline{\cR}_1:=\mu^*H_1\cap\cdots\cap \mu^*H_{n-1} $ is a smooth projective curve in $\overline{X}_1$.  Denote  $\cR_1:=\overline{\cR}_1\backslash \Sigma_1$. Then $\mu|_{\cR_1}:\cR_1\to \cR$ is an isomorphism.  
	
	We apply \cref{thm:harmonicmaps} again to construct another $\varrho$-equivariant  harmonic map $ \tilde{v}: \widetilde{X} \rightarrow \mathcal C$ of    logarithmic energy growth with respect to $(\overline{X}_1,L_1)$.  By the same proof of \cref{goodcoord}, there exists  an admissible coordinate neighborhood $(U;z_1,\ldots,z_n)$ centered at $x_0$ with $U\cap \Sigma_1=(z_1=0)$  such that  the transverse disk $z\mapsto (z,0,\ldots,0)$ is contained   in $\overline{\cR}_1$.  
	  It follows from    \cref{thm:energy estimate} that $\tilde{v}_{\cR}:\widetilde{\cR}\to \cC$  is a $\varrho_{\cR}$-equivariant harmonic map with  logarithmic energy growth.    By \cref{thm:Lefschetz}, we know  that  $\pi_1(\cR)\to \pi_1(X)$ is surjective. Therefore, $\varrho_\cR:\pi_1(\cR)\to G(K)$ also fixes $\cC$ and does not fix a point at infinity of $\cC$.      
	By the unicity property in \cref{lem:rs}, we conclude that ${u}_{\cR}= {v}_{\cR}$ where ${u}_{\cR}$ and ${v}_{\cR}$ are defined in \eqref{restrictionsection}.  Since $x_0$ is an arbitrary point in $X$,  it follows that    $u=v$ holds over the whole $X$. The lemma is proved. 	
\end{proof}

\begin{proposition}\label{thm:unicity}
%
Let $\overline{X}_1$ and $\overline{X}_2$ be two smooth projective compactifications of $X$ with $\Sigma_i:=\overline{X}_i\backslash X$ a simple normal crossing divisor. Let $L_1$ and $L_2$ be sufficiently ample line bundles on $\overline{X}_1$ and $\overline{X}_2$ respectively.    For $i=1,2$, let  $\tilde{u}_i: \widetilde{X} \rightarrow \mathcal C$  be the  unique $\varrho$-equivariant harmonic map of  logarithmic energy growth with respect to $(\overline{X}_i,L_i)$ constructed in \cref{thm:harmonicmaps}. Then $\tilde{u}_1=\tilde{u}_2$. 
\end{proposition}
\begin{proof}
	Since $\overline{X}_1$ is birational to $\overline{X}_2$,  	we can blow-up  the indeterminacy of the birational map $\overline{X}_1\dashrightarrow \overline{X}_2$  to obtain a birational morphism $\overline{X}_3\to \overline{X}_1$ such that we have
	\begin{equation*}
		\begin{tikzcd}
			& \overline{X}_3 \arrow[dr,"\mu_2"]\arrow[dl,"\mu_1"]&\\
			\overline{X}_1 \arrow[rr, dashed]&&\overline{X}_2
		\end{tikzcd} 
	\end{equation*}     
	Here $\mu_1$ and $\mu_2$ are both isomorphic over $X$.   We may assume that $\Sigma_3=\overline{X}_3\backslash X$ is also a simple normal crossing divisor. Fix a sufficiently ample line bundle $L_3$ on $\overline{X}_3$. By \cref{thm:harmonicmaps},  there is a unique $\varrho$-equivariant pluriharmonic map $\tilde{u}_3: \widetilde{X} \rightarrow \mathcal C$ of  logarithmic energy growth with respect to $(\overline{X}_3,L_3)$. Then by \cref{prop:unicity}, $\tilde{u}_1=\tilde{u}_3=\tilde{u}_2$. The proposition is proved. 
\end{proof} 

\Cref{prop:unicity} enables us to obtain the following energy estimate for the harmonic map.
\begin{proposition}[local energy estimate at each point]\label{thm:local estimate}
Let $X$, $\overline{X}$, $L$, $\Sigma$ and $\varrho$ be as   in \cref{thm:harmonicmaps}. Let $\tilde{u} : \widetilde{X} \to \Delta(G)$ be the $\varrho$-equivariant pluriharmonic map with logarithmic energy growth with respect to $(\overline{X}, L)$ constructed in \cref{thm:harmonicmaps}, and $u$ be its corresponding section.  For any holomorphic map $f:\bD\to \overline{X}$ such that $f^{-1}(\Sigma)\subset\{0\}$, we denote by $u_f:\bD^*\to \bD\times_{f^*\varrho}\cC$ the induced harmonic section of $u\circ f$ defined in (\ref{restrictionsection}) and  let $  \widetilde{u_f}:\widetilde{\bD^*}\to \cC$ be  the corresponding $f^* \varrho$-equivariant harmonic map of $u_f$. Then   there is a positive constant $C$ such that   for any $0<r_1<r_2<\frac{1}{2}$, one has	
	\begin{equation} \label{ondisks2}
		\frac{L^2_\gamma}{2\pi} \log \frac{r_2}{r_1} \leq E^{\widetilde{u_{f}}}[\mathbb D_{r_1,r_2}] \leq   \frac{L^2_\gamma}{2\pi} \log \frac{r_2}{r_1} +C,
	\end{equation} 
	where $L_\gamma$ is the translation   length of $\varrho(\gamma)$ with $\gamma \in \pi_1(X)$ corresponding to the loop $\theta\mapsto f(\frac{1}{2}e^{i\theta})$.  
\end{proposition}
\begin{proof}
	We can shrink $\bD$ such that $f|_{\bD^*}:\bD^*\to X$ is an embedding.   We can take an embedded desingularization for the image $C:=f(\bD)$   to obtain a birational morphism $\mu:\overline{X}_1\to \overline{X}$ such that 
	\begin{enumerate}[label*=\rm (\alph*)]
		\item $\mu^{-1}(\Sigma)=\Sigma_1$  is a simple normal crossing divisor. 
		\item \label{item:blow} $\mu$ is an isomorphism over $X$. 
		\item The strict transform $C_1$ of $C$ is smooth, and intersects with $\Sigma_1$    transversely.     In particular, $x_0:=C_1\cap \Sigma_1$ is a smooth point of $\Sigma_1$. 
	\end{enumerate} 
	 Thus, we can take an admissible coordinate neighborhood $(U;z_1,\ldots,z_n)$ centered at $x_0$ such that $U\cap \Sigma_1=(z_1=0)$ and $C_1=(z_2=\cdots=z_n=0)$.  Let $f_1:\bD\to \overline{X}_1$ be the lift of $f$. Then we can reparametrize $\bD$ such that $f_1(z)=(z^k,0,\ldots,0)$. 
	 
	        By \eqref{Cbddef3}, there exists a positive constant $C$ such that for any $0<r_1<r_2<\frac{1}{2}$,  one has	\begin{equation} \label{Cbddef4} 
		0 \leq \int \limits_{{\mathbb D}_{r_1,r_2}} 
		\left| \frac{\partial \tilde{u}}{\partial z_1}(z_1,0,\ldots,0) \right|^2 \frac{\sqrt{-1}dz_1 \wedge d\bar z_1}{2}
		- \frac{L^2_{\gamma_0}}{2\pi}  \log \frac{r_2}{r_1}  \leq C.
	\end{equation}
	Here  $L_{\gamma_0}$ is the translation length of $\varrho(\gamma_0)$ with $\gamma_0 \in \pi_1(X)$ corresponding to the loop $\theta\mapsto (re^{i\theta},0,\ldots,0)$.  Since 
	$$
\left|\frac{d\widetilde{u_{f_1}}}{dz}(z)\right|^2=\left|kz^{k-1}\frac{\partial \tilde{u}}{\partial z_1}(z^k,0,\ldots,0)\right|^2, 
	$$
	 then  for any $0<r_1<r_2<\frac{1}{2}$, one has
	$$
 E^{\widetilde{u_{f_1}}}[\mathbb D_{r_1,r_2}]=\int \limits_{{\mathbb D}_{r_1,r_2}}\left|	\frac{d\widetilde{u_{f_1}}}{dz}(z)\right|^2\frac{\sqrt{-1}dz \wedge d\bar z}{2} =	 k\int \limits_{{\mathbb D}_{r_1^k,r_2^k}} 
	\left| \frac{\partial \tilde{u}}{\partial z_1}(z_1,0,\ldots,0) \right|^2 \frac{\sqrt{-1}dz_1 \wedge d\bar z_1}{2}.
	$$
	Let $u_{f_1}:\bD^*\to \bD^*\times_{f_1^*\varrho}\cC$ be the induced section of $u\circ f_1$ defined in \eqref{restrictionsection}. By \Cref{item:blow}, we have $u_{f_1}=u_f$. The above equality implies 
	\begin{equation} \label{Cbddef5}
		k^2\frac{L_{\gamma_0}^2}{2\pi} \log \frac{r_2}{r_1} \leq  E^{\widetilde{u_{f}}}[\mathbb D_{r_1,r_2}] \leq k^2\frac{L_{\gamma_0}^2}{2\pi} \log \frac{r_2}{r_1}+Ck^2.
	\end{equation}
	for any $0<r_1<r_2<\frac{1}{2}$. 
	For the loop $\gamma\in \pi_1(X)$ defined by  $\theta\mapsto f_1(\frac{1}{2}e^{i\theta})$, the translation length $L_\gamma$ of $\varrho(\gamma)$ is equal to $kL_{\gamma_0}$.  \eqref{Cbddef5} implies  \eqref{ondisks2}. The theorem is proved. 
\end{proof}
By \cref{thm:local estimate}, we can revise  \cref{logenenergy} as follows. 
\begin{dfn}[logarithmic energy growth (II)]\label{def:log energy}
	Let $X$ be a smooth quasi-projective variety, $G$ be a semi-simple algebraic group over a non-archimedean local field $K$,  and  let   $\varrho: \pi_1(X) \rightarrow G(K)$  be a Zariski dense representation.  A $\varrho$-equivariant  harmonic map $ \tilde{u}: \widetilde{X} \rightarrow \Delta(G)$ has \emph{logarithmic energy growth} if for any holomorphic map $f:\bD^*\to X$ with no essential singularity at the origin (i.e. for some,   thus any, smooth projective compactification $\overline{X}$ of $X$, $f$ extends to a holomorphic map $\bar{f}:\bD\to \overline{X}$),  there is a positive constant $C$ such that  for any $r\in (0, \frac{1}{2})$,  one has
	\begin{equation} \label{ondisks3}
		- \frac{L^2_\gamma}{2\pi} \log {r} \leq E^{ u_f}[\mathbb D_{r,\frac{1}{2}}] \leq  -\frac{L^2_\gamma}{2\pi} \log r+C,
	\end{equation} 
	where $L_\gamma$ is the translation   length of $\varrho(\gamma)$ with $\gamma \in \pi_1(X)$ corresponding to the loop $\theta\mapsto f(\frac{1}{2}e^{i\theta})$.  
\end{dfn}
In summary, we have the following result,  which proves the second assertion in   \cref{main:existence} and \cref{main:functorial}.
\begin{thm}\label{thm:functorial}
 The pluriharmonic map $\tilde{u}$ constructed in \cref{thm:harmonicmaps} has logarithmic energy growth in the sense of \cref{def:log energy}.  Moreover, if $f:Y\to X$ is a morphism from another smooth quasi-projective variety $Y$, then for the section $u_f:Y\to \widetilde{Y}\times_{f^*\varrho}\, \cC$ defined in \eqref{restrictionsection}, the corresponding map $\widetilde{u_f}$ is a $f^*\varrho$-equivariant pluriharmonic map of logarithmic energy growth.  Moreover, $\widetilde{u_f}$ is harmonic with respect to any K\"ahler metric compatible with the complex structure of $X$.
\end{thm} 
\begin{proof}	The first assertion follows from \cref{thm:local estimate}. The fact that $u_f$ is pluriharmonic can be deduced from the definition of pluriharmonic. Furthermore, consider any holomorphic map $g:\bD^* \to Y$ with no essential singularity at the origin. Then $f\circ g:\bD^* \to X$ has no essential singularity at the origin.

Denote by $L_\gamma$ is the translation   length of $f^*\varrho(\gamma)$ with $\gamma \in \pi_1(Y)$ corresponding to the loop $\theta\mapsto g(\frac{1}{2}e^{i\theta})$.  Then $L_\gamma$  is the translation   length of $\varrho(\gamma')$ with $\gamma' \in \pi_1(X)$ corresponding to the loop $\theta\mapsto f\circ g(\frac{1}{2}e^{i\theta})$.    By \eqref{ondisks2} there is a positive constant $C$ such that  for any $r\in (0,\frac{1}{2})$, one has
\begin{equation*} 
	- \frac{L^2_\gamma}{2\pi} \log {r} \leq E^{  u_{f\circ g}}[\mathbb D_{r,\frac{1}{2}}] \leq  -\frac{L^2_\gamma}{2\pi} \log r+C.
\end{equation*} 
The harmonicity of $u_f$ with respect to any K\"ahler metric $\omega$ can be established using the same method  in Step 4 of the proof of Theorem~\ref{thm:harmonicmaps}.
\end{proof}

\section{Pluriharmonic maps and logarithmic symmetric differentials} 
Let $X$ be a smooth quasi-projective variety and  and let $G$ be a semisimple algebraic group over  a non-archimedean local field $K$. Assume that $\varrho: \pi_1(X)\to G(K)$ is  a Zariski dense representation.  By \cref{GS},  there is a $\varrho$-equivariant pluriharmonic map $\tilde{u}:\widetilde{X}\to \Delta(G)$, that is locally Lipschitz and has logarithmic energy growth.  In this section we will construct logarithmic symmetric  differentials on $X$ using this pluriharmonic map $u$.  The construction we presented here is close to that in \cite{Kli13} (cf.  \cite{Eys04,Kat97,Zuo96} for other slightly different construction). 

\subsection{Finite \'etale cover and logarithmic symmetric differential} 
	\begin{dfn}[Galois morphism]
		A covering map $\gamma: X \rightarrow Y$ of varieties is called \emph{Galois  with group $G$} if there exists a finite group $G \subset \operatorname{Aut}(X)$ such that $\gamma$ is isomorphic to the quotient map.
	\end{dfn} 
	\begin{lem}\label{descends}
		Let $\bar{f} : (\overline{X}, \Sigma_X) \to (\overline{Y}, \Sigma_Y)$ be a surjective morphism between two log smooth pairs of dimension $n$. Assume that the restriction of $\bar{f}$ to $X$ is étale and Galois, with Galois group $G$. If $H^0(\overline{X}, {\rm Sym}^k \Omega_{\overline{X}}(\log \Sigma_X)) \neq 0$ for some positive integer $k$, then $H^0(\overline{Y}, {\rm Sym}^m \Omega_{\overline{Y}}(\log \Sigma_Y)) \neq 0$ for some positive integer $m$.
	\end{lem}
	\begin{proof}
		Let $\overline{X} \stackrel{\mu}{\to} \overline{X}_1 \stackrel{\bar{f}_1}{\to} \overline{Y}$ be the Stein factorization of $\bar{f}$. Then $\mu$ is a birational morphism onto a projective normal variety $\overline{X}_1$, and the restriction of $\mu$ over $X$ is an isomorphism. We will identify $X_1 := \mu(X)$ with $X$ abusively. By Zariski’s Main Theorem in the equivariant setting (cf.~\cite[Theorem 3.8]{GKP13}), $\bar{f}_1$ is Galois with group $G$. 
		Denote by $\Sigma_Y^{\rm sing}$ the singular locus of $\Sigma_Y$, which is a closed subset  of $\overline{Y}$ of codimension at least two. Let $\overline{Y}^\circ := \overline{Y} \backslash \Sigma_Y^{\rm sing}$ and $\overline{X}_1^\circ := \bar{f}_1^{-1}(Y^\circ)$. Then $\overline{X}_1^\circ$ is smooth, and $\Sigma_{X_1}^\circ := \overline{X}_1^\circ \backslash X_1$ is a smooth divisor in $\overline{X}_1^\circ$. Moreover, it follows from the proof of \cite[Lemma A.12]{Den22} that at any $x \in \Sigma_{X_1}^\circ$, there are admissible coordinate neighborhoods $(\Omega_x; x_1, \ldots, x_n)$ centered at $x$, with $\Sigma^\circ_{X_1} \cap \Omega_x = (x_1 = 0)$, and an admissible coordinate neighborhood $(\Omega_y; y_1, \ldots, y_n)$ centered at $\bar{f}_1(x)$, with $\Sigma_Y \cap \Omega_y = (y_1 = 0)$, such that
		\begin{align}\label{eq:ramified}
			\bar{f}_1(x_1, \ldots, x_n) = (x_1^k, x_2, \ldots, x_n).
		\end{align}

		Let $\Xi$ be the exceptional locus of $\mu$. Then $\mu(\Xi)$ is a closed subset of $\overline{X}_1$ of codimension at least two. The closed subset $\Upsilon := \cup_{g \in G} g.\mu(\Xi)$ of $\overline{X}_1$ also has codimension at least two.
		
		By assumption, there exists a non-zero $P \in H^0(\overline{X}, {\rm Sym}^k \Omega_{\overline{X}}(\log \Sigma_X))$ for some positive integer $k$. Since $\mu$ is an isomorphism over $\overline{X}_1^\circ \backslash \Upsilon$, $P$ induces a logarithmic symmetric differential on $(\overline{X}_1^\circ, \Sigma_{X_1}^\circ)|_{\overline{X}_1^\circ \backslash \Upsilon}$. By the Hartogs theorem, such a logarithmic symmetric differential extends to a logarithmic symmetric differential $P_0 \in H^0\big(\overline{X}_1^\circ, {\rm Sym}^k \Omega_{\overline{X}_1^\circ}(\log \Sigma_{X_1}^\circ)\big)$. We define
$
		Q := \prod_{g \in G} g^* P,$
		which is a non-zero $G$-invariant logarithmic symmetric differential in $H^0\big(\overline{X}_1^\circ, {\rm Sym}^{k|G|} \Omega_{\overline{X}_1^\circ}(\log \Sigma_{X_1}^\circ)\big)$,  as
		\[
		g : (\overline{X}_1^\circ, \Sigma_{X_1}^\circ) \to (\overline{X}_1^\circ, \Sigma_{X_1}^\circ)
		\]
		is an automorphism of the log pair $(\overline{X}_1^\circ, \Sigma_{X_1}^\circ)$ for any $g\in G$.  By the local description of $\bar{f}_1$ in \eqref{eq:ramified}, $Q$ descends to a logarithmic symmetric differential
		\[
		R \in H^0\big(\overline{Y}^\circ, {\rm Sym}^{|G|k} \Omega_{\overline Y}(\log \Sigma_Y)|_{\overline{Y}^\circ}\big),
		\]
		such that $\bar{f}_1^*R=Q$.
		Since $\overline{Y} \backslash \overline{Y}^\circ$ has codimension at least two, by the Hartogs theorem again, $R$ extends to a non-zero logarithmic symmetric differential in
		\[
		H^0\big(\overline{Y}, {\rm Sym}^{|G|k} \Omega_{\overline{Y}}(\log \Sigma_Y)\big).
		\]
		The lemma is proved.
	\end{proof}
	 
	\subsection{Constructing logarithmic symmetric differentials}
Let $\overline{X}$ be a smooth projective compactification of $X$ such that $\Sigma=\overline X \backslash X$ is a simple normal crossing divisor.  We fix a  smooth K\"ahler metric $\overline{\omega}$ on $\overline{X}$, and let $\omega$ be its restriction on $X$.  By \cref{GS}, $\tilde{u}$ is harmonic with respect to $\omega$.  Let $u:X\to \widetilde{X}\times_\varrho \Delta(G)$ be the corresponding section of $\tilde{u}$ defined in \Cref{sec:ems}.  Recall that  $|\nabla {u}|_\omega^2\in L^1_{\rm loc}( {X})$ is the energy density function in \cref{sec:harmonic}.    By \cref{rem:Litschitz},  $|\nabla u|_\omega^2$ is moreover locally bounded as $\tilde{u}$ is locally Lipschitz.   
 

Fix now an apartment $A$ in $ \Delta(G)$, which is isometric to $\mathbb{R}^{N}$. Here $N$ is the $K$-rank of $G$.    Let $W\subset {\rm Isom}(A)$ be the \emph{affine Weyl group} of $\Delta(G)$. The \emph{vectorial Weyl group} $W^v:=W\cap\GL(A)$ is  a finite  group generated by reflections.   Note that $W=W^v\ltimes \Lambda$, where    $\Lambda$ is a lattice acting on $A$ by translations.  For the root system $\Phi=\{\alpha_1,\ldots,\alpha_m\}\subset A^*-\{0\}$ of $\Delta(G)$, one has
\begin{align*}
	 \{w^*\alpha_1,\ldots,w^*\alpha_m\}=\{\alpha_1,\ldots,\alpha_m\}\quad \mbox{for any}\ w\in W^v.
\end{align*}
 In other words, the action of $W^v$ on $\Phi$ is a permutation.    It follows that 
  \begin{align}\label{eq:coin}
 	\{w^*d\alpha_1,\ldots,w^*d\alpha_m\}=\{d\alpha_1,\ldots,d\alpha_m\}\quad \mbox{for any}\ w\in W.
 \end{align} 
Here each $d\alpha_i$ is a linear real one-form  on $A$. 

For any regular point $x\in \cR(u)$ of $u$ (cf.~Definition~\ref{def:sing}),     one can choose a   simply-connected open neighborhood  $U$ of $x$   such that
\begin{itemize}
	\item  the inverse image $\pi_X^{-1}(U)=\coprod_{i\in I}U_i$ is a union of disjoint open sets in $\widetilde{X}$,  each of which is mapped isomorphically onto 
	$U$ by $\pi_X:\widetilde{X}\to X$.
	\item For some $i\in I$,  there is an apartment $A_i$ of $\Delta(G)$ such that $u({U}_i)\subset A_i$.
\end{itemize}
Since  $\tilde{u}$ is $\varrho$-equivariant and  $G(K)$ acts transitively on the set of apartments of $\Delta(G)$, for any other  $U_j$,  $u(U_j)$ is contained in some other apartment $A_j$.    
For each $j\in I$,  we  choose $g_j\in G(K)$ such that $g_j(A_j)=A$. 
We denote   $u_j=g_j \tilde{u}\circ (\pi_X|_{U_j})^{-1} : U\to  A$.  By the pluriharmonicity of $\tilde{u}$,   each $\alpha_k\circ u_j$  is a pluriharmonic function on $U$, and thus $\partial \alpha_k\circ u_j$   is a holomorphic 1-form on $U$.   
\begin{lem}\label{lem:welldef}
For each $i,j\in I$, the two sets of holomorphic 1-forms $\{\partial \alpha_1\circ u_i,\ldots,\partial \alpha_m\circ u_i\}$ and $\{\partial \alpha_1\circ u_j,\ldots,\partial \alpha_m\circ u_j\}$ coincide. 
\end{lem} 
\begin{proof}
	Choose $\gamma\in \pi_1(X)$ such that $\gamma$ maps $U_i$ to $U_j$ isomorphically.  Since $\tilde{u}$ is $\varrho$-equivariant,  one has $ \varrho(\gamma)\tilde{u}\circ (\pi_X|_{U_i})^{-1}=\tilde{u}\circ (\pi_X|_{U_j})^{-1}$,  and thus
	\begin{align}\label{eq:transform}
		u_j= g_j\varrho(\gamma) g_i^{-1}u_i. 
	\end{align} 
	We write $g:=g_j\varrho(\gamma) g_i^{-1}\in G(K)$. Then \eqref{eq:transform} implies that     $u_i(U)\subset A\cap g^{-1}A$.  By \cite[Corollary 4.2.25]{KP23} and \cite[Axiom 4.1.4 (A 1)]{KP23}, there exists $w\in  W$  such that $wx=gx$ for any $x\in A\cap g^{-1}A$.   This implies that    $u_j=wu_i$.   We conclude that 
	\begin{align*}
	 \{\partial \alpha_1\circ u_j,\ldots,\partial \alpha_m\circ u_j\} = \{\partial \alpha_1\circ w u_i,\ldots,\partial \alpha_m\circ wu_i\}= \{\partial \alpha_1\circ u_i,\ldots,\partial \alpha_m\circ u_i\},
	\end{align*}
	where the last equality follows from \eqref{eq:coin}. The lemma is proved. 
\end{proof} 
By \cref{lem:welldef}, $	 \{\partial \alpha_1\circ u_j,\ldots,\partial \alpha_m\circ u_j\}$ defines a well-defined multi-valued holomorphic 1-form on $\cR(u)$, denoted by $\{\omega_1,\ldots,\omega_m\}$.    Let $T$ be a formal variable. Then we can write 
\begin{align}\label{eq:construction}
\prod_{k=1}^{m}\big(T-\omega_j\big)=:T^m+\sigma_1T^{m-1}+\cdots+\sigma_m,
\end{align}
such that  $\sigma_k\in H^0(\cR(u),{\rm Sym}^k\Omega_X|_{\cR(u)})$.   
\begin{proposition}\label{prop:log} For any $k\in \{1,\dots m\}$,  
 $\sigma_{k}$ extends to a logarithmic symmetric differential $H^0(\overline{X}, \Sym^k\Omega_{\overline{X}}(\log \Sigma))$.  Moreover, if $\tilde{u}$ is not constant,  there exists some $k$ such that $\sigma_{k}\neq 0$. 
\end{proposition}
\begin{proof}
	By \cite[Theorem 6.4]{GS92}, $\cS(u)$ is  a closed subset of   $X$  of Hausdorff codimension at least two. Since $u$ is locally Lipschitz,    for any $x\in X$, there are  a neighborhood $\Omega_x$ of $x$ and a constant $C_x$ such that   $|\nabla u|_\omega \leq C_x$   on $\Omega_x$.  Note that  there is a uniform constant $C_0>0$ such that
	\begin{align}\label{eq:k}
		|\sigma_k|_\omega\leq C_0 |\nabla u|_\omega^k \quad \mbox{ over } \quad \cR(u).
	\end{align} Hence  over $\Omega_x\cap \cR(u)$, one has 
	$$
	|\sigma_{k}|_{\omega}\leq C_0 |\nabla u|_\omega^k\leq C_0C^k_x.
	$$
By  the result on removable singularity in \cite[Lemma 3.(ii)]{Shi68},  $\sigma_{k}$ extends to a  holomorphic symmetric form in $H^0(X, \Sym^k\Omega_{X})$, which we still denote by $\sigma_k$.   
	
Choose any point $x$ in the smooth locus of $\Sigma$.    By \eqref{eq:energy} in \cref{logestimate}, there is an admissible coordinate  neighborhood $(U;z_1,\ldots,z_n)$ centered at $x$  with $\Sigma=(z_1=0)$, and  a constant $C_1>0$  such that one has
\begin{align}  \label{eq:energy2}
	- \frac{L^2_\gamma}{2\pi} \log r \cdot \mathsf{Vol}\left(\mathbb D_{\frac{1}{2}}^{n-1} \right)\leq  \int_{\bD_{r,\frac{1}{2}}\times \bD^{n-1}_{\frac{1}{2}}} |\nabla u|_{\omega_0}^2 {\rm dvol}_0 \leq - \frac{L^2_\gamma}{2\pi} \log r  \cdot \mathsf{Vol}\left(\mathbb D_{\frac{1}{2}}^{n-1} \right)+C, \ \ \ \forall \ 0<r< \frac{1}{2}.
\end{align} 
Here $\omega_0:=\sqrt{-1}\sum_{i=1}^{n}\frac{dz_i\wedge d\bar{z}_i}{2}$,  $d\vol_0$  is the  volume form of $\omega_0$  on $U^*:=U\backslash \Sigma$,  and $\mathsf{Vol}\big(\mathbb D_{\frac{1}{2}}^{n-1} \big)$   is the Euclidean volume of $\mathbb D_{\frac{1}{2}}^{n-1}$. 
 Note that  
 \begin{align*} 
 	|\sigma_k|_{\omega_0}\leq C_0 |\nabla u|_{\omega_0}^k \quad \mbox{ over } \quad \cR(u).
 \end{align*}   Thus,   \eqref{eq:k} implies  that  there is a constant $C>0$  such that one has
\begin{align*}  
-C \log r  \leq  \int_{\bD_{r,\frac{1}{2}}\times \bD^{n-1}_{\frac{1}{2}}} |\sigma|_{\omega_0}^{\frac{2}{k}} {\rm dvol}_0 \leq -C \log r   +C, \ \ \ \forall \ 0<r< \frac{1}{2}.
\end{align*} On $U^*$, we write 
$
\sigma_{k}(z)=\sum_{|\alpha|=k}\tau_{\alpha}(z)dz^{\alpha},
$  where $\alpha=(\alpha_1,\ldots,\alpha_n)\in \mathbb{N}^n$ with $|\alpha|:=\sum_{i=1}^{n}\alpha_i$, and $dz^\alpha:=dz_1^{\alpha_1}\cdots dz_n^{\alpha_n}$.  Then $\tau_{\alpha}$  are holomorphic functions over $U^*$.  
 It follows that for each $\alpha$, we have
$$
\int_{\bD_{r,\frac{1}{2}}\times \bD^{n-1}_{\frac{1}{2}}} |\tau_\alpha(z)|^{\frac{2}{k}} idz_1\wedge d\overline{z}_1\wedge\cdots\wedge idz_n \leq -C \log r   +C, \ \ \ \forall \ 0<r< \frac{1}{2}.
$$ 
We now prove that $\tau_{\alpha}(z)$ extends to a meromorphic function over  $U$ for each $\alpha$.  
We fix  even  $m>0$. Then
\begin{align*}
F(r):=\int_{\bD_{r,\frac{1}{2}}\times \bD^{n-1}_{\frac{1}{2}}} |z_1|^{m-1}|\tau_{\alpha}|^{\frac{2}{k}} idz_1\wedge d\overline{z}_1\wedge\cdots\wedge idz_n\wedge d \overline{z}_n 
 \leq -C \log r   +C, \ \ \ \forall \ 0<r< \frac{1}{2}.
\end{align*}   
It follows that for any  $r\in (0,\frac{1}{2})$, we have
\begin{align*}
\int_{\bD_{r,\frac{1}{2}}\times \bD^{n-1}_{\frac{1}{2}}} |z_1|^{m}|\tau_{\alpha}|^{\frac{2}{k}} idz_1\wedge d\overline{z}_1\wedge\cdots\wedge idz_n\wedge d \overline{z}_n 
&= -\int_{r}^{\frac{1}{2}}tF'(t)dt
\\&=rF(r)+\int_{r}^{\frac{1}{2}}F(t)dt-\frac{1}{2}F(\frac{1}{2})\\&\leq -Cr\log r+Cr-C\int_{r}^{\frac{1}{2}}\log t dt-\frac{1}{2}F(\frac{1}{2})+(\frac{1}{2}-r)C.
\end{align*}
This yields \begin{align*}
	\int_{\bD^*_{\frac{1}{2}}\times \bD^{n-1}_{\frac{1}{2}}} |z_1|^{m}|\tau_{\alpha}|^{\frac{2}{k}} idz_1\wedge d\overline{z}_1\wedge\cdots\wedge idz_n\wedge d \overline{z}_n <+\infty.
\end{align*}
By \cref{meromorphic} below we conclude that $z_1^{\frac{km}{2}}  \tau_\alpha$, hence    $\tau_\alpha$   extends to a meromorphic function over $\mathbb D^n$. Thus,  there exists some $\ell\in \bZ$ such that    $\tau_{\alpha}(z)=z_1^\ell  b_\alpha(z)$ such that $b_\alpha(z)\in \sO(U)$ which is not identically equal to zero on $\Sigma$. 

Take a point $y=(0,y_2,\ldots,y_n)\in \Sigma \cap U$ such that $b_\alpha(y)\neq 0$. Then for some $\ep>0$ one has $|b_\alpha(z)|^{\frac{2}{k}}\geqslant C_3$  over  $$V:= \{(z_1,\ldots,z_n)\in \bD^*_{\frac{1}{2}}\times \bD^{n-1}_{\frac{1}{2}}\mid  |z_1|<\ep, |z_2-y_2|<\ep,\ldots, |z_n-y_n|<\ep \} $$
for some constant $C_3>0$.  We shall switch to the Poincar\'e-type metric $\omega_P$ defined in \eqref{eq: Poin} on $U^*$ and apply \eqref{eq:energy3}. By the construction of $\sigma_k$,  we have
\begin{align*} 
	|\sigma_k|_{\omega_P}\leq C_0 |\nabla u|_{\omega_P}^k \quad \mbox{ over } \quad \cR(u).
\end{align*} 
Since   
$$
|\sigma_k(z)|_{\omega_P} \geq  |\tau_\alpha(z) dz_\alpha|_{\omega_P}=|\tau_\alpha(z)| |z_1|^{2\alpha_1}(\log |z_1|^2)^{2\alpha_1}, 
$$
then by \eqref{eq:energy3}, there exists a constant $C_4>0$ such that  one has
\begin{align} \nonumber
& C_3\mathsf{Vol}\left(\mathbb D_{\frac{1}{2}}^{n-1} \right)\int_{r}^{\frac{1}{2}}   \frac{t^{\frac{2\ell+2\alpha_1}{k}}d\log t}{|\log t|^2} = \\\nonumber
& C_3\int_{V\cap \bD_{r,\frac{1}{2}}\times \bD^{n-1}_{\frac{1}{2}}}|z_1|^{\frac{2\ell+2\alpha_1}{k}} \frac{idz_1\wedge d\overline{z}_1}{|z_1|^2(\log |z_1|)^2}\wedge  idz_2\wedge d \overline{z}_2\wedge  \cdots\wedge idz_n\wedge d \overline{z}_n \\\label{eq:contr}
& \leq \int_{V\cap  \bD_{r,\frac{1}{2}}\times \bD^{n-1}_{\frac{1}{2}}}  |\sigma_k|_{\omega_P}^2 d\vol_{\omega_P}\leq C_4\log \frac{1}{r}+C_4, \quad \forall \quad 0<r<\frac{1}{2}.
  \end{align}  
  If $\ell<-\alpha_1$,  then there exists $\ep>0$ such that  $t^{\frac{2\ell+2\alpha_1}{k}}\geq 2|\log t|^3$ for $0<t<\ep$. It follows that there exists a constant $C_5>0$ with 
$$
 \int_{r}^{\frac{1}{2}}   \frac{t^{\frac{2\ell+2\alpha_1}{k}}d\log t}{|\log t|^2}\geqslant  \log^2 r -C_5, \quad \forall \quad 0<r<\ep.
$$
By \eqref{eq:contr}, we have
$$
C_3\mathsf{Vol}\left(\mathbb D_{\frac{1}{2}}^{n-1} \right)( \log^2 r -C_5)\leq C_4\log \frac{1}{r}+C_4, \quad \forall \quad 0<r<\ep.
$$
for any $0<r<\ep_2$,  
which yields a contradiction.  Thus,  $\ell+\alpha_1\geqslant 0$, which implies that 
$$
\sigma_{k}\in 	H^0(\overline{X}^\circ, \Sym^k\Omega_{\overline{X}}(\log \Sigma)|_{\overline{X}^\circ}). 
$$
Here we denote by $\overline{X}^\circ:=\overline{X} \backslash \cup_{j\neq i} \Sigma_i\cap \Sigma_j$ whose complement has codimension at least two in $\overline{X}$. By the Hartogs theorem, it extends to a logarithmic symmetric form on $\overline{X}$.  The first claim is proved. 

If $u$ is not constant, then there is some connected open set $U\subset X$   such that the pluriharmonic map $u_i:U\to A$ defined above is not constant. As $G$ is semisimple, its root system $\{\alpha_1,\ldots,\alpha_m\}$  generates $A^*$.  Thus, the mutivalued holomorphic 1-form $\{\omega_1,\ldots,\omega_m\}$ constructed above  is non zero.  By   \eqref{eq:construction}, $\sigma_k\neq 0$ for some $k\in \{1,\ldots, m\}$.  We prove the second claim. The proposition is proved. 
\end{proof}
The following lemma is the criterion on the meromorphicity of functions in terms of $L^p$-boundedness.
\begin{lem}\label{meromorphic}
	Let $f$ be a  holomorphic function on $(\mathbb D^*)^\ell\times \mathbb D^{n-\ell}$ such that $$\int_{(\mathbb D^*)^\ell\times \mathbb D^{n-\ell}}|f(z)|^pidz_1\wedge d\overline{z}_1\wedge\cdots\wedge i dz_n\wedge d\overline{z}_n\leq C,$$ for some real $0<p<\infty$ and some positive constant $C$. Then $f$ extends to a meromorphic function on $\mathbb D^n$. 
\end{lem}
\begin{proof}
	 Since $|f(z)|^p$ is plurisubharmonic on $(\mathbb D^*)^k\times \mathbb D^\ell$, by the mean value inequality, for any 
	  $
	 z=(z_1,\ldots,z_n)\in (\bD^*_{\frac{1}{2}})^\ell\times (\bD_{\frac{1}{2}})^{n-\ell}$   
	  one has
	 $$
|f(z)|^p\leq \frac{4^{n-\ell}}{\pi^n \prod_{i=1}^{\ell}|z_i|^2} \int_{\Omega_z}|f(\zeta)|^pid\zeta_1\wedge d\overline{\zeta}_1\wedge\cdots\wedge  i d\zeta_n\wedge d\overline{\zeta}_n\leq \frac{4^{n-\ell}C}{\pi^n \prod_{i=1}^{\ell}|z_i|^2}
	 $$
	 where $$\Omega_{z}:= \{(\zeta_1,\ldots,\zeta_n) \in (\mathbb D^*)^\ell\times \mathbb D^{n-\ell}\mid | \zeta_{i}-z_{i} \mid < |z_i| \text { for } i \leq \ell ;\left|\zeta_{i}-z_{i}\right| < \frac{1}{2} \text { for } i>\ell \}.$$
Thus, there is a constant $C_0>0$ such that 
	 	 $$
	 |f(z)|\leq   C_0\prod_{i=1}^{\ell}|z_i|^{-\frac{2}{p}} 
	 $$
for any 
$
z=(z_1,\ldots,z_n)\in (\bD^*_{\frac{1}{2}})^\ell\times (\bD_{\frac{1}{2}})^{n-\ell}$. Hence $\prod_{i=1}^{\ell}z_i^{\lceil \frac{2}{p}\rceil}f(z) $ is  bounded  over    $(\bD^*_{\frac{1}{2}})^\ell\times (\bD_{\frac{1}{2}})^{n-\ell}$. By the Riemann extension theorem, it  extends to a holomorphic function over $\mathbb D^n$.  The lemma is proved. 
\end{proof}
 
\begin{thm}\label{sym form}
 Let $(\overline{X},\Sigma)$ be a log smooth pair. Let $K$ be a   non-archimedean local field $K$.   If $\varrho: \pi_1(X)\to\GL_N(K)$ is an unbounded representation, then   we have 
 	\begin{align}\label{eq:sym}
 		H^0(\overline{X}, \Sym^k\Omega_{\overline{X}}(\log \Sigma))\neq 0
 	\end{align} 
 	for some positive integer $k$. 
 \end{thm}
 
 \begin{proof}
 	\noindent {\it Step 1: Assume that $\varrho$ is reductive.} By \cref{descends}, to prove \eqref{eq:sym}, we are free to replace $X$ by its finite étale covers. We denote by $G$ the Zariski closure of $\varrho$, which is assumed to be reductive. Let $G^0$ be the identity component of $G$.  After replacing $X$ by a finite étale cover corresponding to the finite index subgroup $\varrho^{-1}(\varrho(\pi_1(X))\cap G^0(K))$ of $\pi_1(X)$, we can assume that the Zariski closure $G$ of $\varrho$ is connected. Hence the radical $R(G)$ of $G$ is a torus, and the derived group $D(G)$ is semisimple. Write $T := G/D(G)$ and $G' = G/R(G)$. Then $G'$ is semisimple and $T$ is a torus. Moreover, the natural morphism 
 	\[
 	G \to  G' \times T
 	\] 
 	is an isogeny. We may assume that $G'$ and $T$ are split over $K$  after we replace $K$ by a finite extension.  Denote by $\varrho':\pi_1(X)\to   G'(K)\times T(K)$ the composed morphism of $\varrho$ and $G(K)\to T(K)\times G'(K)$. Then it is also Zariski dense.  
 	
 	Since we assume that  the image of $\varrho(\pi_1(X))$ is unbounded, it follows that the image of $\varrho'$ is also  unbounded (see e.g. \cite[Lemma 2.2.10]{KP23}).  
 	Let $p_1: G'(K) \times T(K) \to G'(K)$ and $p_2: G'(K) \times T(K) \to T(K)$ be the projection maps. Then   representations $\sigma_1 := p_1 \circ \varrho'$ and $\sigma_2 := p_2 \circ \varrho'$ are both Zariski dense.
 	
 	Assume first that $\sigma_1 : \pi_1(X) \to G'(K)$ is unbounded. By \cref{thm:harmonicmaps}, there is a locally Lipschitz $\sigma_1$-equivariant pluriharmonic map $\tilde{u}:\widetilde{X} \to \Delta(G')$ which has logarithmic energy growth. Note that $\tilde{u}$ is not constant; otherwise, its image point would be fixed by $\sigma_1(\pi_1(X))$, and the subgroup of $G'(K)$ fixing a point of $\Delta(G')$ is compact, which contradicts our assumption. Thus, \eqref{eq:sym} follows from \cref{prop:log}.
 	
 	Now assume that $\sigma_1 : \pi_1(X) \to G'(K)$ is  bounded. Then the image of $\sigma_2 : \pi_1(X) \to T(K)$ is unbounded and must be infinite. Since $T(K)$ is abelian, it follows that $\sigma_2$ induces a morphism $H_1(X,\mathbb{Z}) \to T(K)$ with infinite image. In particular, by the universal coefficient theorem, we conclude that $H^1(X,\mathbb{C})$ is infinite.  
 	
 	\begin{claim}\label{claim2}
 		$H^0(\overline{X},\Omega_{\overline{X}}(\log \Sigma))\neq 0$.
 	\end{claim}
 	
 	\begin{proof}[Proof of \cref{claim2}]
 		By the theory of mixed Hodge structures,  one has an isomorphism 
 		$$
 		H^1(X,\mathbb{C}) \simeq 	H^0(\overline{X}, \Omega_{\overline{X}}(\log \Sigma)) \oplus H^{0,1}(\overline{X}). 
 		$$   
 		Since $H^1(X,\mathbb{C})$ is infinite, either $H^0(\overline{X}, \Omega_{\overline{X}}(\log \Sigma))$ or $H^{0,1}(\overline{X})$ is non-zero. In the latter case, by Hodge duality, $H^0(\overline{X}, \Omega_{\overline{X}})$ and thus $H^0(\overline{X}, \Omega_{\overline{X}}(\log \Sigma))$ are non-zero.   
 	\end{proof}
 	
 	In summary, we have proved that $H^0(\overline{X}, \Sym^k \Omega_{\overline{X}}(\log \Sigma)) \neq 0$ for some positive integer if $\varrho$ is reductive.    
 	
 	\medskip
 	
 	\noindent {\it Step 2: General case.} Let $\varrho^{ss} : \pi_1(X) \to \GL_N(\bar{K})$ be the semisimplification of $\varrho$. It follows that $\varrho^{ss}$ is reductive. Since $\pi_1(X)$ is finitely generated, there exists a finite extension $L$ of $K$ such that $\varrho^{ss} : \pi_1(X) \to \GL_N(L)$. Note that $\varrho^{ss}$ is also unbounded (see e.g. \cite[Lemma 3.5]{DY23}). Applying the result from Step 1, we conclude \eqref{eq:sym}. The theorem is proved. 
 \end{proof}  
 
\section{Proof of \cref{main}}
\subsection{On Simpson's integrality conjecture}
 In \cite{Sim92}, Simpson conjectured that for any smooth projective variety $X$, a rigid representation    $\varrho:\pi_1(X)\to\GL_N(\bC)$ is conjugate to an \emph{integral} one, \emph{i.e.} a representation $\pi_1(X)\to\GL_N(O_k)$ where $k$ is a number field and $O_k$ denotes the ring of integers of $k$.  This is known as \emph{Simpson's integrality conjecture}. In \cite[Corollary 1.8]{Kli13}, Klingler proved Simpson's conjecture for compact Kähler manifolds that do not admit symmetric differentials. In this subsection, we extend Klingler's theorem to smooth quasi-projective varieties. 
\begin{thm}\label{rigid}
	Let $(\overline{X},\Sigma)$ be a log smooth pair. Assume that $H^0(\overline{X}, \Sym^k \Omega_{\overline{X}}(\log \Sigma))=0$ for every positive integer $k$.  Then  for any positive integer $N$, 
	 each semisimple representation $\varrho:\pi_1(X)\to\GL_N(\bC)$ is \emph{rigid} and \emph{integral}.    Moreover, $\varrho$ is a complex direct factor of a $\bZ$-variation of Hodge structure. 
\end{thm}
\begin{proof}
	\noindent {\it Step 1: Any reductive representation is rigid.} 
	Let $R_\mathrm{B}(X, N)$ be the representation scheme of $\pi_1(X)$ into $\GL_N$, which is an affine scheme of finite type defined over $\mathbb{Q}$ (cf. \cite{LM85} for the definition). For any field $K$, we have $R_\mathrm{B}(X, N)(K) = \mathrm{Hom}(\pi_1(X), \GL_N(K))$. Note that $\GL_N$ acts on $R_{\rm B}(X,N)$ by conjugation. Denote by $\pi: R_\mathrm{B}(X, N) \to M_{\rm B}(X, N)$ the GIT quotient, which is a surjective morphism of affine schemes of finite type defined over $\mathbb{Z}$.
	
	If $M_{\rm B}(X, N)$ is a positive-dimensional affine scheme, then there exists a $\bar{\mathbb{Q}}$-morphism $\psi: M_{\rm B}(X, N) \to \mathbb{A}^1$ whose image is Zariski dense. Since $\pi$ is surjective, we can find a closed irreducible curve $C \subset R_\mathrm{B}(X, N)$ defined over $\bar{\mathbb{Q}}$ such that $\psi \circ \pi|_{C}: C \to \mathbb{A}^1$ is generically finite. We may take an open subset $U \subset \mathbb{A}^1$ over which the morphism $\psi \circ \pi|_{C}: C \to \mathbb{A}^1$ is finite. 
	
	Let $k$ be a finite extension of $\mathbb{Q}$ such that $C$ is defined over $k$, and $\psi \circ \pi|_{C}$ is a morphism of $k$-schemes. Let $\mathfrak{p}$ be a non-archimedean place of $k$, and    $k_\kp$ be its completion. Then $k_\kp$ is a non-archimedean local field of characteristic zero. Take $x \in U(k_\kp)$ and $y \in C(\overline{k_\kp})$ over $x$. Then $y$ is defined over some finite extension of $k_\kp$, with its degree controlled by the degree of $\psi \circ \pi|_{C}$.   Note that there are only finitely many such field extensions. Hence there exists a finite extension $L$ of $k_\kp$ such that the points over $U(k_\kp)$ are all contained in $C(L)$.  Since $U(k_\kp)$ is unbounded, the image $\psi \circ \pi(C(L)) \subset \mathbb{A}^1(L)$ is unbounded.
	
	Let $R_0$ be the set of all bounded representations in $R_\mathrm{B}(X, N)(L)$. By a theorem of Yamanoi (\cite[Lemma 4.2]{Yam10}), $M_0 = \pi(R_0)$ is compact in $M_\mathrm{B}(X, N)(L)$ with respect to the analytic topology, implying that $\psi(M_0)$ is bounded in $\mathbb{A}^1(L)$. Accordingly, there exists some $\tau \in C(L)$ such that $\tau: \pi_1(X) \to \GL_N(L)$ is unbounded. By \cref{sym form}, we have $H^0(\overline{X}, \Sym^k \Omega_{\overline{X}}(\log \Sigma)) \neq 0$ for some positive integer $k$. This leads to a contradiction, proving that $M_{\rm B}(X, N)$ is zero-dimensional. Hence any representation $\varrho: \pi_1(X) \to \GL_N(\mathbb{C})$ is rigid.
	
	\medskip
	
	\noindent {\it Step 2: Any rigid representation is integral.} 
	Let $\varrho: \pi_1(X) \to \GL_N(\mathbb{C})$ be a semisimple  representation. By Step 1, it is rigid. Thus, after conjugation, there exists a number field $k$ such that $\varrho: \pi_1(X) \to \GL_N(k)$. Let $\kp$ be a non-archimedean place of $k$, and let $k_\kp$ be its completion. By assumption and \cref{sym form}, the extension $\pi_1(X) \to \GL_N(k_\kp)$ of $\varrho$ is bounded for each non-archimedean place  $\kp$ of $k$. Therefore, $\varrho$ factors through 
$
	\pi_1(X) \to \GL_N(O_k),$
	where $O_k$ is the ring of integers of $k$. Thus, $\varrho$ is integral.
	
	\medskip
	
	\noindent {\it Step 3: $\varrho$ is a complex direct factor of a $\mathbb{Z}$-VHS.} 
	Let $\varrho: \pi_1(X) \to \GL_N(O_k)$ be as in Step 2. For every embedding $\sigma: k \to \mathbb{C}$, the composition $\sigma \circ \varrho: \pi_1(X) \to \GL_N(\mathbb{C})$ is semisimple and rigid. By \cite{Moc06}, $\sigma \circ \varrho$ underlies a complex variation of Hodge structure for each embedding $\sigma: k \to \mathbb{C}$. The conditions in \cite[Proposition 7.1 and Lemma 7.2]{LS18} are satisfied, and we apply \cite[Proposition 7.1]{LS18} to conclude that $\varrho$ is a complex direct factor of a $\mathbb{Z}$-variation of Hodge structure. 
	The theorem is thus proved.
\end{proof}
 
\begin{rem}
	The above proof gives a new proof of  the rigidity part \cite{Ara02} in the projective case.  In the first version of the present paper on arXiv, we used Uhlenbeck's compactness in gauge theory to prove such result. However, we felt that it would be more interesting to establish \cref{rigid} from the theory of harmonic maps to Bruhat-Tits buildings, as it provides a unified approach to both rigidity and integrality.
	
	It is worth noting that non-abelian Hodge theory in the archimedean setting cannot be entirely avoided. Specifically, in Step 3, we rely on Mochizuki's theorem in \cite{Moc06}, whose proof is based on harmonic maps to symmetric spaces.
	
	Recently, Esnault and Groechenig \cite{EG18} proved that a cohomologically rigid local system over a quasi-projective variety with finite determinant and quasi-unipotent local monodromies at infinity is also integral. 
\end{rem}

\subsection{Proof of \cref{main}}
Let us prove \cref{main}.
\subsubsection{The case of characteristic zero}    
\begin{proof}[Proof of \cref{main} for ${\rm char}\, \bK=0$] 
	Since $\pi_1(X)$ is finitely generated, there exists a subfield $k \subset \bK$ such that ${\rm tr. deg.}(k/\bQ) < \infty$ and $\tau(\pi_1(X)) \subset \GL_{N}(k)$. We can choose an embedding $k \to \bC$, and thus assume that $\tau: \pi_1(X) \to \GL_{N}(\bC)$.
	
	Thanks to \cref{descends}, to prove the theorem, we are free to replace $X$ by finite étale covers. Let $\sigma: \pi_1(X) \to \GL_N(\bC)$ be the semisimplification of $\tau$. If $\sigma(\pi_1(X))$ is finite, then after replacing $X$ by a finite étale cover, we can assume that $\sigma(\pi_1(X))$ is trivial. In other words, $\tau(\pi_1(X))$ is contained in some unipotent group $U \subset \GL_N(\bC)$. Then, there exists a sequence of normal subgroups
	\[
	U = U_0 \supset U_1 \supset \cdots \supset U_s = \{1\}
	\]
	such that each $U_i/U_{i+1}$ is commutative. Since $\tau(\pi_1(X))$ is infinite, after replacing $X$ by a finite étale cover, there exists some $i$ such that $\tau(\pi_1(X)) \subset U_i$ and the natural map $\tau': \pi_1(X) \to U_i/U_{i+1}$ induced by $\tau$ has infinite image. Since $U_i/U_{i+1}$ is abelian, $\tau'$ factors through $H_1(X, \bZ) \to U_i/U_{i+1}$. In other words, $H_1(X, \bZ)$ is infinite. By the universal coefficient theorem, $H^1(X, \bC)$ is also infinite. By \cref{claim2}, we have $H^0(\overline{X}, \Omega_{\overline{X}}(\log \Sigma)) \neq 0$. The theorem is proved if $\sigma$ has finite image.
	
	Now, assume $\sigma$ has infinite image. We assume by contradiction that
	\[
	H^0(\overline{X}, \Sym^k \Omega_{\overline{X}}(\log \Sigma)) = 0
	\]
	for all $k > 0$. By \cref{rigid}, $\sigma$ is a direct factor of a semisimple representation $\varrho: \pi_1(X) \to \GL_m(\bZ)$ underlying a $\mathbb{Z}$-variation of Hodge structure. Let
	\[
	\Phi: X \to \sD/\Gamma
	\]
	be the corresponding period map, where $\sD$ is the period domain and $\Gamma = \varrho(\pi_1(X))$ is the monodromy group, which acts discretely on $\sD$. By Malcev's theorem, we can replace $X$ by a finite étale cover such that $\Gamma$ is torsion-free. Since $\varrho$ has infinite image, $\Phi$ has  positive-dimensional image. By a theorem of Griffiths \cite{Gri70}, there is a Zariski open subset $X_1 \subset \overline{X}$ containing $X$ such that $\Phi$ extends to a proper holomorphic map $X_1 \to \sD/\Gamma$. Its image $Z$ is thus a proper subvariety of $\sD/\Gamma$. By a theorem of Sommese \cite[Proposition IV]{Som78} (or \cite{DY23} for a new proof), there exists:
	\begin{enumerate}[label=\rm (\alph*)]
		\item a proper bimeromorphic map $\nu: Y \to Z$ from a smooth quasi-projective variety $Y$,
		\item a proper birational morphism $\mu: X_2 \to X_1$ from a smooth quasi-projective variety $X_2$,
		\item an algebraic and surjective morphism $f: X_2 \to Y$,
	\end{enumerate}
	such that we have the following commutative diagram:
	\[
	\begin{tikzcd}
		X_2 \arrow[r, "\mu"] \arrow[d, "f"] & X_1 \arrow[d] & \\
		Y \arrow[r, "\nu"] & Z &
	\end{tikzcd}
	\]
	Take a smooth projective compactification $\overline{Y}$ of $Y$ such that $\Sigma_Y = \overline{Y} - Y$ is a simple normal crossing divisor. Then $Y \to Z \to \sD/\Gamma$ is a generically immersive and horizontal map. By \cite{Bru18, BC20}, we know that the logarithmic cotangent bundle $\Omega_{\overline{Y}}(\log \Sigma_Y)$ is big. Therefore, there exists a  positive integer $k$ such that
	\[
	H^0(\overline{Y}, \Sym^k \Omega_{\overline{Y}}(\log \Sigma_Y)) \neq 0.
	\]
	Take a smooth projective compactification $\overline{X}_2$ of $X_2$ such that:
	\begin{itemize}
		\item $\Sigma_2 = \overline{X}_2 \backslash X_2$ is a simple normal crossing divisor.
		\item $f$ extends to a surjective morphism $\overline{f}: \overline{X}_2 \to \overline{Y}$ with $\bar{f}^{-1}(\Sigma_Y) \subset \Sigma_2$.
		\item $\mu$ extends to a birational morphism $\bar{\mu}: \overline{X}_2 \to \overline{X}$ with $\bar{\mu}^{-1}(\Sigma_X) \subset \Sigma_2$.
	\end{itemize}
 	We pull back a  non-zero logarithmic symmetric differential in $H^0(\overline{Y}, \Sym^k \Omega_{\overline{Y}}(\log \Sigma_Y))$ via $\overline{f}$ to obtain a non-trivial element $P \in H^0(\overline{X}_2, \Sym^k \Omega_{\overline{X}_2}(\log \Sigma_2))$. Let $\Xi$ be the exceptional locus of $\bar{\mu}$. Then $\bar{\mu}(\Xi)$ has codimension at least two in $\overline{X}$ since $\mu$ is birational. Thus, $P$ induces a section $P_0 \in H^0(\overline{X} \backslash \mu(\Xi), \Sym^k \Omega_{\overline{X}}(\log \Sigma)|_{\overline{X} \backslash \Xi})$. By Hartogs' theorem, $P_0$ extends to a non-trivial logarithmic symmetric differential in $H^0(\overline{X}, \Sym^k \Omega_{\overline{X}}(\log \Sigma))$. The theorem is proved in the case where ${\rm char}\, \bK = 0. $
\end{proof}

\subsubsection{The case of positive characteristic}  
\begin{proof}[Proof of \cref{main} for ${\rm char}\, \bK>0$.]
	We can assume that $\bK$ is algebraically closed after replacing $\bK$ with its algebraic closure. Let $p = {\rm char}\, \bK$. Let $R_{\rm B}(\pi_1(X),\GL_{N})$ be the representation scheme of $\pi_1(X)$ into $\GL_N$, which is of finite type and defined over $\bZ$.  
	Note that $R_{\rm B}(\pi_1(X),\GL_{N})(\bK)$ can be identified with the set $\operatorname{Hom}\left(\pi_1(X), \GL_N(\bK)\right)$. Consider the base change $R_{\bF_p} := R_{\rm B}(\pi_1(X),\GL_{N}) \times_{\spec \bZ}\spec \bF_p$, which is an affine $\bF_p$-scheme of finite type.  
	We note that $\GL(N, \bF_p)$ acts on $R_{\bF_p}$ via conjugation. Using Seshadri’s extension of geometric invariant  quotient  theory for schemes, we can take the GIT quotient of $R_{\bF_p}$ by $\GL(N, \bF_p)$, denoted by $M_{\bF_p}$. Then $M_{\bF_p}$ is also an affine $\bF_p$-scheme of finite type. Note that the $\bK$-points $M_{\bF_p}(\bK)$ are identified with the conjugacy classes of semisimple representations $\pi_1(X) \to \GL_N(\bK)$.   
	\medspace
	
	\noindent \textit{Case 1: $M_{\bF_p}$ is positive dimensional.} Since the morphism $\pi_p: R_{\bF_p} \to M_{\bF_p}$ is surjective between affine $\bF_p$-schemes of finite type, we can find an irreducible affine curve $C_o \subset R_{\bF_p}$ defined over $\bar{\bF}_p$ such that $\pi_p(C_o)$ is positive dimensional. Let $\overline{C}$ be the compactification of the normalization $C$ of $C_o$, and let $\{P_1,\ldots,P_\ell\}= \overline{C}\backslash C$. One can find a positive integer $m$ such that $\overline{C}$ is defined over $\bF_q$ with $q=p^m$, and $P_i\in \overline{C}(\bF_q)$ for each $i$. 
	
	By the universal property of the representation scheme, $C$ gives rise to a representation $\varrho_C:\pi_1(X)\to \GL_N(\bF_q[C])$, where $\bF_q[C]$ is the coordinate ring of $C$. Consider the discrete valuation $v_i:\bF_q(C) \to \bZ$ defined by $P_i$, where $\bF_q(C)$ is the function field of $C$. Let $\widehat{\bF_q(C)}_{v_i}$ be the completion of $\bF_q(C)$ with respect to $v_i$. Then we have $\big(\widehat{\bF_q(C)}_{v_i}, v_i\big) \simeq \big(\bF_q((t)), v\big)$, where $\big(\bF_q((t)), v\big)$ is the formal Laurent field of $\bF_p$ with the valuation $v$ defined by $v(\sum_{i=m}^{+\infty}a_i t^i) = \min \{i \mid a_i \neq 0\}$. Let $\varrho_i:\pi_1(X) \to \GL_N(\bF_q((t)))$ be the extension of $\varrho_C$ with respect to $\big(\widehat{\bF_q(C)}_{v_i}, v_i\big)$.
	
	\begin{claim}\label{lem:simple}
		There exists some $i \in \{1, \ldots, \ell\}$ such that $\varrho_i:\pi_1(X) \to \GL_N(\bF_q((t)))$ is unbounded.
	\end{claim}
	
	\begin{proof}
		Assume for the sake of contradiction that $\varrho_i$ is bounded for each $i$. Then after replacing $\varrho_i$ by some conjugation, we have $\varrho_i(\pi_1(X)) \subset \GL_N(\bF_q[[t]])$. For any matrix $A \in \GL_N(B)$ where $B$ is an $\bF_p$-algebra, we denote by $\chi(A)=T^N + \sigma_1(A)T^{N-1} + \cdots + \sigma_N(A)$ its characteristic polynomial with $\sigma_i(A)\in B$ the coefficients. Then $\sigma_j(\varrho_C(\gamma)) \in \bF_q[C]$ for every $\gamma \in \pi_1(X)$.  
		
		Since we have assumed that $\varrho_i(\pi_1(X)) \subset \GL_N(\bF_q[[t]])$ for every $i$, it follows that $\sigma_j(\varrho_i(\gamma)) \in \bF_q[[t]]$ for each $i \in \{1,\ldots,\ell\}$ and $j \in \{1,\ldots,N\}$. Therefore, by the definition of $\varrho_i$, $v_i\big(\sigma_j(\varrho_C(\gamma))\big) \geq 0$ for each $i$. It follows that $\sigma_j(\varrho_C(\gamma))$ extends to a regular function on $\overline{C}$, which is thus constant. This implies that for any $\{\eta_i:\pi_1(X) \to \GL_N(K_i)\}_{i=1,2}$ with    such that ${\rm char}\, K_i =   p$ and $\eta_i \in C(K_i)$, we have $\chi(\eta_1(\gamma)) = \chi(\eta_2(\gamma))$ for each $\gamma \in \pi_1(X)$. It yields $[\eta_1] = [\eta_2]$. Hence $\pi_p(C_o)$ is a point, leading to a contradiction.     
	\end{proof}
	
	\cref{lem:simple} together with \cref{sym form} imply the existence of non-trivial logarithmic symmetric differentials in $H^0(\overline{X}, \Sym^k\Omega_{\overline{X}}(\log \Sigma))$. We have thus proved the theorem when $M_{\bF_p}$ is positive dimensional. 
	
	\medspace
	
	\noindent \textit{Case 2: $M_{\bF_p}$ is zero dimensional.} We will prove that this case cannot occur. First, assume that $\tau:\pi_1(X) \to \GL_N(\bK)$ is semisimple. It follows that $\tau$ is conjugate to some $\varrho':\pi_1(X) \to \GL_N(\bar{\bF}_p)$. Since $\pi_1(X)$ is finitely generated, we have $\varrho'(\pi_1(X))\subset \GL_N(\bF_q)$ for some $q=p^m$.  Since $\GL_N(\bF_q)$ is a finite group, it follows that $\varrho'(\pi_1(X))$, hence $\tau(\pi_1(X))$, is finite. This leads to a contradiction. Hence the semisimplification  of $\tau$ must have finite image.  
	
	After replacing $X$ by a finite étale cover, we can assume that $\tau(\pi_1(X))$ is contained in the subgroup of strictly upper-triangular matrices in $\GL_N(\bK)$, which is a successive extension of $\mathbb{G}_{a,\bK}$. Hence $\tau(\pi_1(X))$ is a successive extension of finitely generated subgroups of $\mathbb{G}_{a,\bK}$, all of which are finite. It follows that $\tau(\pi_1(X))$ is finite, leading again to a contradiction. Thus, $M_{\bF_p}$ cannot be zero dimensional.
	
	The proof of the theorem is accomplished. 
\end{proof}

 

 \appendix

\section{Pluriharmonic maps from a quasi-projective surface}\label{sec:pluriharmonic}

In a series of remarkable papers \cite{Moc07b,Moc06}, Mochizuki  proves the existence of a  pluriharmonic metrics on flat vector bundles over smooth quasi-projective varieties.  These metrics correspond to infinite energy pluriharmonic  maps into  symmetric spaces of noncompact type   
by the Donaldson-Corlette theorem (cf.~\cite{Don87,Cor88}).   
The key step in Mochizuki's argument is to show that the harmonic metric over a quasi-projective surface is actually pluriharmonic.  The existence of pluriharmonic metrics on a  higher dimensional  smooth quasi-projective variety follows from an inductive argument on the dimension.  In this appendix, we generalize Mochizuki's argument to prove the following.  


\begin{thmx} \label{thm:pu}
	Let $(\overline{X},\Sigma)$ be a log smooth pair with $\dim \overline{X}=2$,  $Y$ be  a Riemannian manifold with strongly nonpositive curvature or a Euclidean building, and  $\rho:\pi_1(X) \rightarrow \mathsf{Isom}(Y)$ be an isometric action on $Y$. Endow $X$  with a Poincar\'e-type K\"ahler metric $g$ defined in \Cref{sec:metric}.  Then a  
	$\varrho$-equivariant harmonic  map $\tilde u:\widetilde X \rightarrow Y$ with logarithmic growth with respect to $g$  is pluriharmonic.
\end{thmx}

Note that  symmetric space of noncompact type has  strongly nonpositive curvature (cf.~\cite[Corollary 5.5]{Loh90}). Thus, Theorem~\ref{thm:pu} includes these cases which have   already been proved by Mochizuki (cf.~\cite[Proposition 11.20]{Moc06}).

The notion of harmonic maps of logarithmic energy growth has been discussed in \cite{DMrs} and \cite{DMks}.  Loosely speaking, this means that the energy density function of $u$ grows like $\frac{1}{r}$ along a disk transverse to a   $\Sigma$.  For the purpose of this appendix, it suffices to know that $u$ satisfies the energy estimates listed in Section~\ref{weneed}. We established this in \cite{DMks}. 

We will assume for the majority of the appendix that $(\overline{X},\Sigma)$ is  a log smooth pair with $\dim \overline{X}=2$, and that the target space $Y$ is either a Riemannian manifold $M$ of strongly nonpositive curvature or a Euclidean building $\Delta(G)$.   In Section~\ref{sec:pu} and Section~\ref{sec:bddm}, we treat the two cases $Y=M$ or $Y=\Delta(G)$ separately.

\subsection{Pairing of forms} We will use the following notation.  
Let   $M$  be a smooth Riemannian manifold and $TM\otimes \bC$ be its complexified tangent bundle.
For a smooth map $u: \widetilde X \rightarrow M$, let
$E := u^*(TM \otimes \bC)$.  Decompose the pullback of the Levi-Civita connection  as
\[
\nabla=\nabla'+\nabla''
\]
where 
\[
\nabla':C^{\infty}(E) \rightarrow \Omega^{1,0}(E), \ \ \ \nabla'':C^{\infty}(E) \rightarrow \Omega^{0,1}(E).
\]
In turn, $\nabla'$ and $\nabla''$ induce differential operators
\[
\partial_E: \Omega^{p,q}(E) \rightarrow \Omega^{p+1,q}(E), \ \ \
\bar \partial_E: \Omega^{p,q}(E) \rightarrow \Omega^{p,q+1}(E)
\]
where
\begin{eqnarray*}
	\partial_E(\phi \otimes s) & = &  \partial \phi \otimes s+(-1)^{p+q}\phi \otimes \nabla'_Es
	\\
	\bar \partial_E(\phi \otimes s) & = &  \bar \partial \phi \otimes s+(-1)^{p+q}\phi \otimes \nabla''_Es.
\end{eqnarray*}
Let $\{s^i\}$ be a  local frame of $E$.  For 
\begin{eqnarray*}
	\psi  =  \psi_i \otimes s^i \in \Omega^{p,q}(E) \ \mbox{and }  \
	\xi  =  \xi_i \otimes s^i \in \Omega^{p',q'}(E)
\end{eqnarray*} 
we set
\[ 
\{\psi, \xi \}= \langle s^i, s^j \rangle
\psi_i \wedge \bar \xi_j
\in \Omega^{p+q',q+p'}
\] 
where $ \langle \cdot, \cdot \rangle$ is the sesquilinear extension to $TM \otimes \bC$ of the Riemannian metric on $M$. 

\begin{rem} \label{globallydefined}
Consider the case when $\tilde u:\widetilde X \rightarrow Y=\Delta(G)$ is a harmonic map into a building.  Let $x \in \cR(u)$ and let $\mathcal N$ and $A$ be as in Definition~\ref{def:sing}.  Isometrically identify $\varphi:\mathbb R^N \simeq A$ and view the restriction $ u_\varphi:=\tilde u|_{\mathcal N}$ as a map into $\mathbb R^N$.  Thus $\bar \partial u_\varphi= \frac{\partial u_\varphi}{\partial \bar z^\alpha} d\bar z^\alpha$ is a $(0,1)$-form with values in $\bC^N$.  Note that $\bar \partial u_\varphi$ is independent of the choice of the isometric identification $A \simeq \mathbb R^N$ up to rotation.  Therefore, the $(1,1)$-form  
$\{\bar \partial u_\varphi, \bar \partial u_\varphi\}$ is independent of the choice of the isometric identification $\mathbb R^N \simeq A$.  For this reason, we  henceforth denote $\{\bar \partial u_\varphi, \bar \partial u_\varphi\}$ simply as $\{\bar \partial u, \bar \partial u\}$.  This function is well-defined on the regular set $\mathcal R(u)$ which is an open set in $\widetilde X$ of codimension 2.   By the local Lipschitz regularity of $\tilde{u}$, $\left|\{\bar \partial u, \bar \partial u\} \right|$ is an integrable function on any compact subdomain of $\widetilde X$, and we will henceforth interpret it as a locally $L^1$-function defined a.e.~on $\widetilde X$. 
\end{rem}

\subsection{Cut-off functions} \label{sec:cutoff}

Denote $\mathbb D_{z^i}$ to indicate that the complex coordinate $z^i$  parameterizes $\mathbb D$.

Let $P \in \Sigma_i \cap \Sigma_j$ for  $i \neq j$, and let $V_P$ be a neighborhood of $P$ containing no other crossings.
Choose holomorphic  trivializations $e_i$ (resp. $e_j$)  of $\sO _{\overline{X}}(\Sigma_i)$ (resp. $\sO _{\overline{X}}(\Sigma_j)$)  on 
$V_P$ and define $z^1$ (resp. $z^2$) by setting
$\sigma_i=z^1e_i, \mbox{ (resp. }\sigma_j=z^2e_j)
$. 
Let $h_j$ be a Hermitian metric  on $ \sO_{\overline{X}}(\Sigma_j)$ such that $|e_j|_{h_j}=1$ in $V_P$ for any crossing $P$.

Let $h$ be a Hermitian metric on $\overline{X}$,
not necessarily K\"ahler, such that the following holds:
\begin{itemize}
	\item[(i)] The metric $h$ is the Euclidean metric in a neighborhood $V_P$ of every crossing $P$, i.e.~
	\[
	h|_{V_P} =dz^1d\bar z^1+ dz^2d\bar z^2. 
	\]
	By rescaling $\sigma_1$ and $\sigma_2$ if necessary, we can assume without  loss of generality that 
	$\overline{\mathbb D}_{z^1} \times \overline{\mathbb D}_{z^2} \subset V_P.
	$
	\item[(ii)] The metric $h$ induces the orthogonal decomposition $T \overline{X}|_{\Sigma_j} = T \Sigma_j \oplus N \Sigma_j$ and  under the natural isomorphism 
	\[
	N \Sigma_j\simeq \sO _{\overline{X}}(\Sigma_j)|_{\Sigma_j},
	\]
	the restriction of $h$ to $N \Sigma_j$ is same as $ h_j$.
\end{itemize}

By scaling the metric $h$ if necessary, we can assume that the restriction of the exponential map  
\[
\exp : N \Sigma_j \subset T \overline{X}|_{\Sigma_j} \rightarrow \overline{X}
\] 
to 
	${\mathcal D}_j  =   \{\nu \in  N\Sigma_j: |\nu|_{h_j} < 1\} $
	defines a diffeomorphism.     We  identity $\mathcal D_j$  as a neighborhood of $\Sigma_j$ in $\overline{X}$; i.e.~$
	{\mathcal D}_j \simeq \exp({\mathcal D}_j) \subset \overline{X}.
	$
	Let $\mathcal D_j^*=\mathcal D \backslash \Sigma_j$.
	
	Fix a non-increasing, non-negative smooth function  $\eta:[0,\infty) \rightarrow [0,1]$ 
	satisfying
	\[
	\eta(x)=1 \mbox{ for } 0 \leq x \leq \frac{1}{2},  \ \ \eta(x) = 0 \mbox{ for } \frac{2}{3} \leq x <\infty.
	\]
	For $N \in {\mathbb N}$, define a cut-off function 
	\[
	\chi_N: X \rightarrow [0,1], \ \ \ \chi_N =\left\{
	\begin{array}{ll}
		\displaystyle{\prod_{j=1}^L  \eta \left(N^{-1}\log |\sigma_j|_{h_j}^{-2} \right) } & \mbox{in } \displaystyle{\bigcup_{j=1}^L {\mathcal D}_j^*}
		\\
		1 & \mbox{otherwise}.
	\end{array}
	\right.
	\]

	\subsection{Neighborhood of divisors} \label{sec:nbhd}
	 	We follow the notation of Sections~\ref{sec:metric} and~\ref{sec:cutoff}.
	The restriction of the normal bundle $N\Sigma_j \rightarrow \Sigma_j$ to $\mathcal D_j$ defines a disk bundle
	\begin{equation} \label{diskbundle}
		\pi_j : \mathcal D_j \rightarrow \Sigma_j.
	\end{equation}
	
	We now consider a finite collection of sets near the divisor 
	of the following two types:
	\begin{itemize}
		
		\item A  set   of type (A) admits a local unitary trivialization
		\begin{equation} \label{ref:typeA}
			\pi_j^{-1}(\Omega) \simeq  \Omega \times {\mathbb D}_{z^2}, \ \ 
		\end{equation}
		of $\pi_j:  \bar {\mathcal D}_j \rightarrow \Sigma_j$ where $\Omega \subset \Sigma_j$ is a contractible open subset of  $\Sigma_j$ containing no crossings.  
		With $\sigma_j$ the canonical section of $\sO_{\overline{X}}(\Sigma_j)$ as before, define a  function $\zeta$ on $\Omega \times \overline{\mathbb D}$ by
		$\sigma_j = \zeta e.
		$ 
		Thus, $\zeta$ is holomorphic with respect to  the complex structure on $\overline{X}$ and 
		$(\zeta, z^2)
		$
		define   holomorphic coordinates on  a set $\Omega \times \overline{\mathbb D}$ of type (A).
		\item  
		A set of type (B) 
		is a open set $\mathbb D_{z^1} \times \mathbb D_{z^2} \subset V_P
		$
	where $V_P$ be an  open set as in Section~\ref{sec:cutoff} containing a single crossing $P \in \Sigma_i \cap \Sigma_j$ ($i\neq j$).
	By the property (i) of the hermitian metric $h$ (cf.~Section~\ref{sec:cutoff}),  
	$(z^1,z^2)$ are {\it holomorphic coordinates} with respect to the   complex  structure on $\overline{X}$.
	Furthermore, with the identification
	$\overline{\mathbb D}_{z^1} \simeq \overline{\mathbb D}_{z^1} \times \{0\} \subset \Sigma_i$ (resp. $\overline{\mathbb D}_{z^2} \simeq  \{0\} \times \overline{\mathbb D}_{z^2}\subset \Sigma_j$),  
	$\pi_j^{-1}(\overline{\mathbb D}_{z^1}) \simeq \overline{\mathbb D}_{z^1} \times \overline{\mathbb D}_{z^2} \ \ \ (\mbox{resp. }\pi_i^{-1}(\overline{\mathbb D}_{z^2}) \simeq \overline{\mathbb D}_{z^1} \times \overline{\mathbb D}_{z^2})
	$ 
	is a local unitary trivialization of $\pi_j:  \overline{\mathcal D}_j \rightarrow \Sigma_j$ (resp. $\pi_i: \bar {\mathcal D}_i \rightarrow \Sigma_i$).  
\end{itemize}
\begin{dfn}
Fix a smooth K\"ahler metric $\overline{\omega}$ on $\overline{X}$. 	 We define  a  Kähler form on $\overline{X} \backslash \bigcup_{i \neq j} \Sigma_i$ by 
\begin{align}\label{gsigma}
	 	g_{\Sigma_j}:= C \overline{\omega}-\frac{\sqrt{-1}}{2} \sum_{i \neq j} \partial \bar{\partial} \log \log \left|\sigma_i\right|_{h_i}^{-2}.
\end{align} 
\end{dfn} 

Define $g_{\Sigma_j}$ to be the restriction to $\Sigma_j \backslash \bigcup_{i \neq j} \Sigma_i$ of the Kähler metric associated to this Kähler form. This is a smooth metric on $\Sigma_j$ away from the crossings.
We will use the following volume estimates for the Poincar\'e-type K\"ahler metric $g$ defined in \eqref{griffithsmetric}.  For more details, we refer to \cite[Section 3]{DMks}.
\begin{itemize}
	\item
	In a set of type (A), we write $z^2=re^{i\theta}$ in polar coordinates.  We have
	\begin{equation} \label{volcomp}
		d\mbox{vol}_g =d\mbox{vol}_P \left(1+O\left(\frac{1}{(-\log r^2+\alpha)^2} \right)\right)
	\end{equation}
	where $\alpha=\alpha(\zeta)$ is a smooth function. 
	\[
	d\mbox{vol}_P = d\mbox{vol}_{g_j}  \wedge \frac{dz^2\wedge d\overline{z}^2}{-2i r^2(-\log r^2+\alpha)^2} 
	\]
	and  $g_j$ is the restriction to $\Sigma_j$ of the K\"ahler metric $g_{\sigma_j}$ defined  in \eqref{gsigma}.
	\item In a set of type (B), we write $z^1=\varrho e^{i\phi}$ and $z^2=re^{i\theta}$ in polar coordinates.  We have
	\begin{equation} \label{volcompB}
		d\mbox{vol}_g =d\mbox{vol}_P \left(1+O\left(\frac{1}{(\log r^2)^2} \right) + O\left(\frac{1}{(\log r^2)^2} \right)\right),
	\end{equation}
	where 
	\[
	d\mbox{vol}_P = \frac{dz^1\wedge d\overline{z}^1}{-2i \varrho^2(\log \varrho^2)^2}  \wedge \frac{dz^2\wedge d\overline{z}^2}{-2ir^2(\log r^2)^2}.
	\]
\end{itemize}

\subsection{Energy estimates for harmonic maps of logarithmic growth} \label{weneed}

Let  $L_j$ be the translation length of $\varrho(\gamma_j)$ where $\gamma_j$ is the element of  $\pi_1(X)$ corresponding to a loop around the irreducible component $\Sigma_j$ of the divisor $\Sigma$.
Throughout this paper, the $\rho$-equivariant harmonic map $\tilde{u}$ in Theorem~\ref{thm:pu}  are assumed to satisfy the following estimates:
\begin{itemize}
	\item[(i)] In the set $\Omega \times {\mathbb D}^*_{\frac{1}{4}}$ away from a crossing 
	where $(z^1, \zeta=s e^{i\eta})$ are the holomorphic coordinates on $\Omega \times {\mathbb D}$, \begin{eqnarray*} \label{weneedA}
		\int_{\Omega \times \bar {\mathbb D}^*_{\frac{1}{4}} }  \left| \frac{\partial u}{\partial z^1} \right|^2 dz^1 \wedge d\bar z^1 \wedge \frac{d\zeta \wedge d\bar \zeta}{s^2(-\log s^2)^2} & < & \infty
		\\
		\int_{\Omega \times \bar {\mathbb D}^*_{\frac{1}{4}} }  \left( \left| \frac{\partial u}{\partial \zeta} \right|^2 - \frac{L_j}{16\pi s^2} \right) dz^1 \wedge d\bar z^1 \wedge d\zeta \wedge d\bar \zeta    & < & \infty
		\\
		\int_{\Omega \times \bar {\mathbb D}^*_{\frac{1}{4}} }  \left| \frac{\partial u}{\partial \zeta} \right|^2 dz^1 \wedge d\bar z^1 \wedge \frac{d\zeta \wedge d\bar \zeta}{(-\log s^2)^2}    & < & \infty
		\\
		\int_{\Omega \times \bar {\mathbb D}^*_{\frac{1}{4}} }   \left| \frac{\partial u}{\partial s} \right|^2  dz^1 \wedge d\bar z^1 \wedge d\zeta \wedge d\bar \zeta    & < & \infty
		\\
		\int_{\Omega \times \bar {\mathbb D}^*_{\frac{1}{4}} }
		\left(   \left| \frac{\partial u}{\partial \eta} \right|^2 -\frac{L_j^2}{4\pi} 
		\right) 
		dz^1 \wedge d\bar z^1 \wedge \frac{d\zeta \wedge d\bar \zeta}{s^2}  & < & \infty \\
		\int_{\Omega \times \overline{\mathbb D}^*_{\frac{1}{4}}}
		\left| \frac{\partial u}{\partial \eta} \right|^2 
		dz^1 \wedge d\bar z^1 \wedge \frac{d\zeta \wedge d\bar \zeta}{s^2(-\log s^2)^2} 
		& < & \infty.
	\end{eqnarray*}

	\item[(ii)] In the set  $\bar {\mathbb D}^*_{\frac{1}{4}}  \times  \bar {\mathbb D}^*_{\frac{1}{4}}$ at a crossing where $(z^1=\varrho e^{i\phi},z^2=re^{i\theta})$  are the holomorphic coordinates on ${\mathbb D} \times {\mathbb D}$:
	
	\begin{eqnarray*} \label{weneedB}
		\int_{\bar {\mathbb D}^*_{\frac{1}{4}}  \times  \bar {\mathbb D}^*_{\frac{1}{4}}}
		\left( 
		\left| \frac{\partial u}{\partial z^1} \right|^2 -  \frac{L_i^2}{16\pi \varrho^2}
		\right)  dz^1 \wedge d\bar z^1\wedge \frac{dz^2 \wedge d\bar z^2}{r^2 (-\log r^2)^2} & < & \infty 
		\\
		\int_{\bar {\mathbb D}^*_{\frac{1}{4}}  \times  \bar {\mathbb D}^*_{\frac{1}{4}}}
		\left(  \left| \frac{\partial u}{\partial z^2} \right|^2 -   \frac{L_j^2}{16\pi r^2} \right)  \frac{dz^1 \wedge d\bar z^1}{\varrho^2 (-\log \varrho^2)^2} \wedge dz^2 \wedge d\bar z^2 & < & \infty
		\\
		\int_{\bar {\mathbb D}^*_{\frac{1}{4}}  \times  \bar {\mathbb D}^*_{\frac{1}{4}}}
		\left| \frac{\partial u}{\partial \varrho} \right|^2  dz^1 \wedge d\bar z^1 \wedge \frac{dz^2 \wedge d\bar z^2}{r^2(-\log r^2)^2} & < & \infty
		\\
		\int_{\bar {\mathbb D}^*_{\frac{1}{4}}  \times  \bar {\mathbb D}^*_{\frac{1}{4}}}
		\left| \frac{\partial u}{\partial r} \right|^2  \frac{dz^1 \wedge d\bar z^1}{\varrho^2 (-\log \varrho^2)^2}  \wedge dz^2 \wedge d\bar z^2 & < & \infty
		\\
		\int_{\bar {\mathbb D}^*_{\frac{1}{4}}  \times  \bar {\mathbb D}^*_{\frac{1}{4}}}
		\left(  \left| \frac{\partial u}{\partial \phi} \right|^2 - \frac{L_j^2}{4\pi}
		\right)  \frac{dz^1 \wedge d\bar z^1}{\varrho^2 } \wedge \frac{dz^2 \wedge d\bar z^2}{r^2 (-\log r^2)^2} & < & \infty
		\\
		\int_{\bar {\mathbb D}^*_{\frac{1}{4}}  \times  \bar {\mathbb D}^*_{\frac{1}{4}}}
		\left( \left| \frac{\partial u}{\partial \theta} \right|^2 -  \frac{L_i^2}{4\pi}
		\right) 
		\frac{dz^1 \wedge d\bar z^1}{\varrho^2 (-\log \varrho^2)^2} \wedge \frac{dz^2 \wedge d\bar z^2}{r^2} & < & \infty.
	\end{eqnarray*}
	
\end{itemize}

\begin{rem}
	In \cite{DMks}, we constructed a $\rho$-equivariant harmonic map satisfying the above estimates (cf.~ \cite[Theorem 6.6 and Theorem 6.7]{DMks}) under the assumption that $\rho$ is {\it proper}; i.e.
	the sublevel sets of the function  $\delta: \tilde{X} \rightarrow [0,\infty)$ defined by
	\[
	\delta(P)=\max \{d(\rho(\lambda)P,P): \lambda \in \Lambda\}.
	\]  are bounded in $Y$.
\end{rem}

\subsection{Technical results}

\label{sec:appendix}

We will prove the  technical results needed in the proof of Theorem~\ref{thm:pu}.  The  arguments presented here are similar to those contained in \cite{Moc07b}. We  include all the details  for the sake of completeness.

\begin{lem} \label{lemma:prelim}
	Let $V=\bar {\mathbb D}_{\frac{1}{4}}^*  \times  \bar {\mathbb D}_{\frac{1}{4}}^*$ be a set  at  a crossing (cf.~Section~\ref{weneed}~(ii))
	and  $(z^1=\varrho e^{i\phi},z^2=r e^{i\theta})$ be holomorphic coordinates in $V$.  
	If   $\{F_N\}_{N=1}^{\infty}$ is a sequence of functions defined on $V$ satisfying the following:
	\begin{itemize}
		\item[(a)] $\displaystyle{|F_N(z^1,z^2)| \leq \frac{c}{(-\log r^2)^2}}$ for some constant $c>0$ independent of $N$,
		\item[(b)] $\displaystyle{ c_0:=\int_V   F_N(z^1,z^2)
			\frac{dz^1 \wedge d\bar z^1}{\varrho^2} \wedge \frac{dz^2 \wedge d\bar z^2}{r^2}}$ is  independent of $N$, and 
		
		\item[(c)] for any  $z^2 \in {\mathbb D}_{\frac{1}{4}}^*$ with $|z^2|=r$, 
		$F_N(z^1,z^2)=0$  for $N$ sufficiently large,
	\end{itemize}
	then
	\[
	\lim_{N \rightarrow \infty}  \int_V F_N(z^1,z^2) \left| \frac{\partial u}{\partial z^1}\right|^2 dz^1 \wedge d\bar z^1 \wedge \frac{dz^2 \wedge d\bar z^2}{r^2} =\frac{c_0 L_i^2}{16\pi}.
	\]
\end{lem}

\begin{proof}
	We first rewrite
	\begin{eqnarray} 
		\lefteqn{
			\int_V F_N(z^1,z^2) \left| \frac{\partial u}{\partial z^1}\right|^2 dz^1 \wedge d\bar z^1 \wedge \frac{dz^2 \wedge d\bar z^2}{r^2}
		} \nonumber 
		\\
		& = &  
		\frac{L_i^2}{16\pi} \int_V   F_N(z^1,z^2)
		\frac{dz^1 \wedge d\bar z^1}{\varrho^2} \wedge \frac{dz^2 \wedge d\bar z^2}{r^2}
		\nonumber \\
		&  & +\int_V F_N(z^1,z^2) \left( \left| \frac{\partial u}{\partial z^1} \right|^2 -
		\frac{L_i^2}{16\pi \varrho^2} \right) dz^1 \wedge d\bar z^1 \wedge \frac{dz^2 \wedge d\bar z^2}{r^2}. 
	\end{eqnarray}
	The first term is equal to $\frac{c_0 L_i^2}{16\pi}$  by assumption (b).
	For the second term of (\ref{int-(i)}), we first rewrite the integral as
	\[
	\int_0^{\frac{1}{4}} \left( \int_{{\mathbb D}^*_{\frac{1}{4}} }\int_0^{2\pi}F_N(z^1,z^2) \left(   \left| \frac{\partial u}{\partial z^1} \right|^2 -\frac{L_i^2}{16\pi \varrho^2} 
	\right) \varrho d\phi \wedge \frac{dz^2 \wedge d\bar z^2}{-2ir^2}  \right)d\varrho.
	\]
	By assumption (a), the integral inside the bracket, i.e.~the function
	\begin{equation}  \label{fncFN'}
		r \mapsto \int_{{\mathbb D}^*_{\frac{1}{4}} }\int_0^{2\pi}F_N(z^1,z^2) \left(   \left| \frac{\partial u}{\partial z^1} \right|^2 -\frac{L_i^2}{16\pi \varrho^2} 
		\right) \varrho d\phi \wedge \frac{dz^2 \wedge d\bar z^2}{-2ir^2},
	\end{equation}
	is bounded from above (independently of $N$) by a non-negative function
	\[
	r \mapsto c \int_{{\mathbb D}^*_{\frac{1}{4}} }\int_0^{2\pi} \left(   \left| \frac{\partial u}{\partial z^1} \right|^2 -\frac{L_i^2}{16\pi \varrho ^2}
	\right)  \varrho d\phi \wedge \frac{dz^2 \wedge d\bar z^2}{-2ir^2(-\log r^2)^2}.
	\]
	The above is non-negative by the definition of $L_i$ and integrable over the interval $[0,\frac{1}{4}]$ by Section~\ref{weneed}~(ii).
	Furthermore, 
	the function (\ref{fncFN'}) converges to 0 for each  $r \in (0,\frac{1}{4})$ by assumption (c).  Thus,   Lebesgue's dominated  convergence theorem implies   the result.
\end{proof}


\begin{proposition} \label{integrability}
	If $\{\chi_N\}$ is the sequence of cut-off functions defined in Section~\ref{sec:cutoff}, then
	\begin{equation} \label{ibp1}
		\lim_{N\rightarrow \infty} \int_X \partial \overline{\partial} \chi_N \wedge \{ \bar\partial u, \bar\partial u \}<\infty.
	\end{equation}
\end{proposition}

\begin{proof}

	Let $V$  be either a set $\Omega \times {\mathbb D}^*_{\frac{1}{4}}$ away from the crossings (cf.~Section~\ref{weneed}~(i)) or a set 
	$\bar {\mathbb D}_{\frac{1}{4}}^*  \times  \bar {\mathbb D}_{\frac{1}{4}}^*$  at  a crossing (cf.~Section~\ref{weneed}~(ii)).
	Since   $\partial \bar \partial \chi_N$  is supported in the finite union of such sets for sufficiently large $N$,  it suffices to prove 
	\begin{equation} \label{ibp}
		\lim_{N\rightarrow \infty} \int_{V} \partial \overline{\partial} \chi_N \wedge \{ \bar\partial u, \bar\partial u \}<\infty
	\end{equation}
	for either    $V=\Omega \times {\mathbb D}^*_{\frac{1}{4}}$ or $V=\bar {\mathbb D}_{\frac{1}{4}}^*  \times  \bar {\mathbb D}_{\frac{1}{4}}^*$.
	Throughout this proof of (\ref{ibp}), we will use $c$ to denote a generic positive constant that may change from line to line but is independent of $N \in {\mathbb N}$.
	
	First, consider   the subset
	$V=\bar {\mathbb D}_{\frac{1}{4}}^*  \times  \bar {\mathbb D}_{\frac{1}{4}}^*$ near  a crossing with local holomorphic coordinates $(z^1=\varrho e^{i\phi},z^2=r e^{i\theta})$.  In $V$ and  for $N$ sufficiently large,   
	\begin{eqnarray*}
		\chi_N(z^1,z^2)  
		& = & \eta \left(-N^{-1} \log \varrho^2 \right) \eta \left(-N^{-1} \log r^2 \right). 
	\end{eqnarray*}
	The support of $\eta'\left(-N^{-1} \log \varrho^2 \right) $ and $\eta''\left(-N^{-1} \log \varrho^2 \right) $ is contained in 
	\begin{equation} \label{sptphi1}
		W_N:=\left\{\frac{1}{2} \leq -N^{-1} \log \varrho^2 \leq \frac{2}{3}\right\} 
	\end{equation}
	and the support of $\eta'\left(-N^{-1} \log r^2 \right) $ and $\eta''\left(-N^{-1} \log r^2 \right)$ is contained in 
	\begin{equation} \label{sptphi2}
		V_N:=\left\{\frac{1}{2} \leq -N^{-1} \log r^2 \leq \frac{2}{3}\right\}.
	\end{equation}
	Therefore,
	\begin{eqnarray} 
		(-\log \varrho^2) \left| \frac{\eta'(-N^{-1} \log \varrho^2)}{N} 
		\right|
		\leq  c, & &  (-\log r^2)\left| \frac{\eta'(-N^{-1} \log r^2)}{N} 
		\right|
		\leq c \nonumber  \\
		(-\log \varrho^2)^2 \left| \frac{\eta'' (-N^{-1} \log \varrho^2)}{N^2} \right| \leq  c,
		& & 
		(-\log r^2)^2 \left| \frac{\eta'' (-N^{-1} \log r^2)}{N^2} \right| \leq  c.
		\label{cleverbd!} 
	\end{eqnarray}
	We have
	\begin{eqnarray}
		\partial \overline{\partial} \chi_N & = &  \eta(-N^{-1}\log \varrho^2)\eta''(-N^{-1}\log r^2) \frac{dz^2 \wedge d\overline{z}^2}{N^2r^2}
		\nonumber 
		\\
		& & 
		\
		+
		\eta''(-N^{-1}\log \varrho^2) \eta(-N^{-1}\log r^2)\frac{dz^1 \wedge d\overline{z}^1}{N^2 \varrho^2}
		\nonumber
		\\
		& & \ \ +\eta'(-N^{-1}\log \varrho^2) \eta'(-N^{-1}\log r^2) \frac{dz^1 \wedge d\overline{z}^2}{N^2 z^1\overline{z}^2}
		\nonumber 
		\\
		& & \ \ \ +\eta'(-N^{-1}\log r^2) \eta'(-N^{-1}\log \varrho^2) \frac{dz^2 \wedge d\overline{z}^1}{N^2 \overline{z}^1z^2}.
		\label{sum'}
	\end{eqnarray}
	Using  (\ref{sum'}),  we write the integral of  (\ref{ibp}) as the sum $(i)+(ii)+(iii)+(iv)$
	where
	\begin{eqnarray*}
		(i) &= & 
		\int_V  \eta(-N^{-1}\log \varrho^2)\eta''(-N^{-1}\log r^2) \frac{dz^2 \wedge d\overline{z}^2}{N^2r^2}   \wedge \{ \bar\partial u, \bar\partial u \}\\
		(ii) &= & 
		\int_V  \eta''(-N^{-1}\log \varrho^2)\eta(-N^{-1}\log r^2) \frac{dz^1 \wedge d\overline{z}^1}{N^2\varrho^2}   \wedge \{ \bar\partial u, \bar\partial u \}\\
		(iii) &= & 
		\int_V  \eta'(-N^{-1}\log \varrho^2)\eta'(-N^{-1}\log r^2) \frac{dz^1 \wedge d\overline{z}^2}{N^2z^1\overline{z}^2}   \wedge \{ \bar\partial u, \bar\partial u \}
		\\
		(iv) &= & 
		\int_V  \eta'(-N^{-1}\log \varrho^2)\eta'(-N^{-1}\log r^2) \frac{dz^2 \wedge d\overline{z}^1}{N^2\overline{z}^1z^2}   \wedge \{ \bar\partial u, \bar\partial u \}.
	\end{eqnarray*}
	
	First, consider the integral $(i)$.  Using the identity
	\[ 
	\langle  \frac{\partial  u}{\partial{ \bar z^\alpha}},  \frac{\partial  u}{ {\partial{\bar  z^\beta}}} \rangle  d\bar z^\alpha \wedge dz^\beta = h_{i\bar j}  \frac{\partial  u^i}{\partial{ \bar z^\alpha}} 
	\overline{
		\frac{\partial  u^j}{ {\partial{\bar  z^\beta}}} 
	}
	d\bar z^\alpha \wedge dz^\beta =\{\bar \partial u, \bar \partial u\},
	\]
	we have
	\begin{equation}   \label{int-(i)}
		(i) =  \int_V  
		\eta(-N^{-1}\log \varrho^2) \frac{\eta''(-N^{-1}\log r^2)}{N^2} 
		\left| \frac{\partial u}{\partial z^1} \right|^2 dz^1\wedge d\bar z^1 \wedge \frac{dz^2 \wedge d\bar z^2}{r^2}.
	\end{equation}
	We now check that  
	\[
	F_N(z^1,z^2) = \eta(-N^{-1}\log \varrho^2) \frac{\eta''(-N^{-1}\log r^2)}{N^2}
	\]
	satisfies the assumptions (a), (b) and (c) of  Lemma~\ref{lemma:prelim}.
	First, $F(z^1,z^2)$ satisfies assumption (a) of Lemma~\ref{lemma:prelim} by (\ref{cleverbd!}).
	Next, we will check that the function $F_N(z^1,z^2)$ also satisfies assumption (b) of Lemma~\ref{lemma:prelim}.
	Indeed, after  a change of variables, 
	\begin{equation} \label{stvar}
		t=-N^{-1}\log \varrho \ \mbox{ and } \ s=-N^{-1}\log r,
	\end{equation}
	we obtain
	\[
	\frac{dz^1 \wedge d\overline{z}^1}{N\varrho^2} =-2i \frac{d\varrho \wedge d\phi}{N\varrho} = 2i dt \wedge d\phi \ \mbox{ and } \ \frac{dz^2 \wedge d\overline{z}^2}{Nr^2} = -2i \frac{dr \wedge d\theta}{Nr} =  2i ds \wedge d\theta.
	\]
	Thus, 
	\begin{eqnarray*}  \label{int-(i)1}
		\lefteqn{  c_0:=
			\int_V 
			\eta(-N^{-1}\log \varrho^2)  \frac{\eta''(-N^{-1}\log r^2)}{N^2}  \frac{dz^1 \wedge d\overline{z}^1}{\varrho^2} \wedge \frac{d\overline{z}^2 \wedge dz^2}{r^2}
		} 
		\nonumber  \\
		& = & c \int_0^{\frac{1}{3}} \eta(2t) dt \int_{\frac{1}{4}}^{\frac{1}{3}}  \eta''(2s) ds= c \int_0^{\frac{1}{3}} \eta(2t) dt \cdot \left(  \eta'(\frac{2}{3}) - \eta'(\frac{1}{2}) \right) =0.
	\end{eqnarray*}
	Finally,  $F_N(z^1,z^2)$ satisfies assumption (c) of Lemma~\ref{lemma:prelim} by (\ref{sptphi2}).
	By applying Lemma~\ref{lemma:prelim}, we conclude
	\begin{equation} \label{ito0}
		\lim_{N \rightarrow \infty}|(i)| =0.
	\end{equation}
	The same argument also implies
	\[
	\lim_{N \rightarrow \infty}|(ii)| =0.
	\]
	We will now bound  the term $(iii)$.    Indeed, we can rewrite
	\begin{eqnarray}
		|(iii)|&= &
		\left| \int_V \frac{ \eta'(-N^{-1}\log \varrho^2)}{N}\frac{\eta'(-N^{-1}\log r^2)}{N} \frac{dz^1 \wedge d\overline{z}^2}{z^1 \overline{z}^2}   \wedge \{ \bar\partial u, \bar\partial u \} \right|
		\nonumber  \\
		&\leq   & 
		\int_V
		\left|  \langle \frac{\partial u}{\partial \overline{z}^1}, \frac{\partial  u}{\partial \overline{z}^2}\rangle  \right| 
		\frac{ \eta'(-N^{-1}\log \varrho^2)}{N}\frac{\eta'(-N^{-1}\log r^2)}{N}  
		\frac{dz^1 \wedge d\overline{z}^1}{\varrho} \wedge \frac{dz^2 \wedge d \bar z^2}{r} 
		\nonumber  \\
		& \leq &   \int_{V_N}
		\left|  \frac{\partial u}{\partial \overline{z}^1}\right|^2
		\left( \frac{\eta'(-N^{-1}\log \varrho^2)}{N}  \right)^2
		dz^1 \wedge d\overline{z}^1 \wedge \frac{dz^2 \wedge d \bar z^2}{r^2}
		\nonumber  \\
		& & \ \ +  \int_{W_N}
		\left| \frac{\partial  u}{\partial \overline{z}^2} \right|^2 
		\left( \frac{ \eta'(-N^{-1}\log r^2)}{N}\right)^2  
		\frac{dz^1 \wedge d\overline{z}^1}{\varrho^2} \wedge dz^2 \wedge d \bar z^2.
		\label{iiiint}
	\end{eqnarray}
	
	For the first integral on the right hand side of (\ref{iiiint}), we let
	\[
	F_N(z^1,z^2) =\chi_{V_N}
	\left( \frac{ \eta'(-N^{-1}\log r^2)}{N}\right)^2 
	\]
	where $\chi_{V_N}$ is the characteristic function of $V_N$.
	First,  $F_N(z^1,z^2)$ satisfies assumption (a) of Lemma~\ref{lemma:prelim} by (\ref{cleverbd!}),.
	Next, we check that it satisfies assumption (b) of Lemma~\ref{lemma:prelim}.  Indeed, using the substitution (\ref{stvar}), 
	\begin{eqnarray*}
		c_0 & := &  \int_V \chi_{V_N} \left( \frac{\eta'(-N^{-1}\log r^2)}{N}  \right)^2 \frac{dz^1 \wedge d\overline{z}^1}{\varrho^2} \wedge \frac{dz^2 \wedge d \bar z^2}{r^2} 
		\\
		& = &  c   \int_{\frac{1}{4}}^{\frac{1}{3}} dt \  \int_{\frac{1}{4}}^{\frac{1}{3}}   (\eta'(2s))^2  ds.
	\end{eqnarray*}
	Finally,  $F_N(z^1,z^2)$ satisfies assumption (c) of Lemma~\ref{lemma:prelim} by (\ref{sptphi2}).
	Thus, the second integral on the right hand side of (\ref{iiiint}) limits to $\frac{c_0L_i^2}{16\pi}$ as $N \rightarrow \infty$ by Lemma~\ref{lemma:prelim}.  Analogously, the second integral on the right hand side of (\ref{iiiint}) limits to $\frac{c_0L_j^2}{16\pi}$ as $N \rightarrow \infty$. 
	Thus, we have shown 
	\begin{equation} \label{iii0}
		\lim_{N \rightarrow \infty}|(iii)| \leq \frac{c_0(L_i^2 +L_j^2)}{16\pi}.
	\end{equation}
	Same argument shows 
	\[
	\lim_{N \rightarrow \infty}|(iv)| \leq \frac{c_0(L_i^2 +L_j^2)}{16\pi}.
	\]
	Summing the limits of $(i)$, $(ii)$, $(iii)$ and $(iv)$, we conclude that (\ref{ibp}) is satisfied  in the case   $V=\bar {\mathbb D}_{\frac{1}{4}}^*  \times  \bar {\mathbb D}_{\frac{1}{4}}^*$.

	Next, consider  set $V=\Omega \times {\mathbb D}^*_{\frac{1}{4}}$ away from the crossings with holomorphic coordinates  $(z^1,\zeta=r e^{i\theta})$.   
	In $V$ and for sufficiently large $N$, 
	\begin{equation}\label{setb0} 
		\chi_N(z^1,\zeta) =\eta \left(N^{-1} \log b|\zeta|^{-2} \right). 
	\end{equation}
	We compute
	\begin{eqnarray}
		\partial \overline{\partial} \chi_N & = & \frac{\eta''(N^{-1} \log b|\zeta|^{-2})}{N^2} \left(\frac{d\zeta \wedge d\overline{\zeta}}{|\zeta|^2}
		+  \frac{\partial b \wedge \overline{\partial} b}{b^2}- \frac{\partial b \wedge d\overline{\zeta}}{b \bar \zeta}- \frac{d\zeta \wedge \overline{\partial} b}{\zeta b}\right)\nonumber \\
		&  &  \ +\frac{\eta'(N^{-1} \log b|\zeta|^{-2})}{N}  \partial \overline{\partial} \log b. \label{sum}
	\end{eqnarray}
	The support of $\eta'$ and $\eta''$ is contained in 
	\begin{eqnarray*} 
		\left\{\frac{1}{2} \leq N^{-1} \log b|\zeta|^{-2} \leq \frac{2}{3}\right\},
	\end{eqnarray*}
	which  is contained in the set
	\begin{equation}\label{sprt}
		V_N=\Omega \times {\mathbb D}_{z^2,c_1e^{-\frac{N}{3}}, c_2e^{-\frac{N}{4}}}
	\end{equation}
	for appropriate constants $c_1$ and $c_2$ depending only on $b$. 
	Therefore,
	\begin{eqnarray} 
		\left| \frac{\eta'(N^{-1} \log b|\zeta|^{-2})}{N} 
		\right|
		& \leq &  \frac{c}{\log br^{-2}}
		\leq  \frac{c}{-\log r^2}, \label{cleverbd1}\\
		\left| \frac{\eta'' (N^{-1} \log b|\zeta|^{-2})}{N^2} \right|& \leq & \frac{c}{(\log br^{-2})^2}
		\leq  \frac{c}{(-\log r^2)^2}.
		\label{cleverbd2} 
	\end{eqnarray}
	Using  (\ref{sum}),  we write the integral of  (\ref{ibp}) as the sum $(I)+(II)+(III)+(IV)+(V)$.  For the integral $(I)$, we write
	\begin{eqnarray*}
		|(I)| 
		&  = &   
		\left|
		\int_V \frac{\eta'' (N^{-1} \log b|\zeta|^{-2})}{N^2}  \frac{d\zeta \wedge d\overline{\zeta}}{|\zeta|^2} \wedge \{ \bar\partial u, \bar\partial u \}   
		\right|
		\\
		& \leq &
		c\int_{V_N}  \left| \frac{\partial u}{\partial \bar z^1} \right|^2  dz^1 \wedge d\bar z^1 \wedge \frac{d\zeta \wedge d\bar \zeta}{r^2 (-\log r^2)^2}\ \ \ (\mbox{by (\ref{cleverbd2}))}.
	\end{eqnarray*}
	By Section~\ref{weneed}~(i), 
	\[
	\int_V     \left| \frac{\partial     u}{\partial \bar z^1} \right|^2  dz^1 \wedge d\bar z^1 \wedge \frac{d\zeta \wedge d\bar \zeta}{r^2 (-\log r^2)^2}<\infty.
	\]
	Thus, Lebesgue's dominated  convergence Theorem implies
	\begin{equation} \label{I0}
		\lim_{N \rightarrow \infty} (I) = 0.
	\end{equation}
	For the integral $(II)$, we write
	\begin{eqnarray*}
		|(II)| 
		& = &
		\left|   \int_V \frac{\eta'' (N^{-1} \log b|\zeta|^{-2})}{N^2} \frac{\partial b \wedge \overline{\partial} b}{b^2} \wedge \{ \bar\partial u, \bar\partial u \}  \right| 
		\\
		& \leq & c   \int_{V_N}  \left|\frac{\partial u}{\partial \overline{z}^1}\right |^2  dz^1 \wedge d\bar z^1 \wedge \frac{d\zeta \wedge d\bar \zeta}{(-\log r^2)^2}+ \int_{V_N} \left |\frac{\partial u}{\partial \overline{\zeta}}\right |^2  dz^1 \wedge d\bar z^1 \wedge \frac{d\zeta \wedge d\bar \zeta}{(-\log r^2)^2}  \\
		& &  \ \ \  (\mbox{by}  \ \frac{\partial b}{b}=O(1)\ \mbox{and }  (\ref{cleverbd2})).
	\end{eqnarray*}
	By Section~\ref{weneed}~(i), we can apply an analogous argument to (\ref{I0}) to conclude
	\[
	\lim_{N \rightarrow \infty} (II)=0.
	\] 
	In order to estimate $(III)$, notice that
	\begin{eqnarray*} \label{cancel1}
		d\overline{\zeta} \wedge \{\bar \partial u, \bar\partial u \}=d\overline{\zeta} \wedge\left(\left|\frac{\partial u}{\partial \overline{z}^1} \right|^2d\overline{z}^1\wedge d{z}^1
		+ \langle \frac{\partial u}{\partial \overline{z}^1}, \frac{\partial u}{\partial \overline{\zeta}}\rangle d\overline{z}^1\wedge d\zeta \right).
	\end{eqnarray*}
	Thus,
	\begin{eqnarray*}
		|(III)| & = &  \left|     \int_V \frac{\eta'' (N^{-1} \log b|\zeta|^{-2})}{N^2} \frac{\partial b \wedge d\overline{\zeta}}{b\overline{\zeta}} \wedge \{\bar \partial u, \bar\partial u \}     \right|    
		\\
		& \leq & 
		c     \int_{V_N} \left(\left|\frac{\partial u}{\partial \overline{z}^1} \right|^2
		+\left| \langle \frac{\partial u}{\partial \overline{z}^1}, \frac{\partial u}{\partial \bar\zeta}\rangle \right| \right) dz^1 \wedge d\bar z^1 \wedge \frac{d\zeta \wedge d\bar \zeta}{r(-\log r^2)^2} 
		\\
		& & 
		\ \ \  ({\mbox{by}  \ \frac{\partial b}{b}=O(1) \mbox{ and } (\ref{cleverbd2}))}\\
		\\
		& = & 
		c     \int_{V_N}   r \left(\left|\frac{\partial u}{\partial \overline{z}^1} \right|^2
		+\left| \langle \frac{\partial u}{\partial \overline{z}^1}, \frac{\partial u}{\partial \bar\zeta} \rangle\right| \right)\ dz^1 \wedge d\bar z^1 \wedge \frac{d\zeta \wedge d\bar \zeta}{r^2(-\log r^2)^2} 
		\\
		& \leq & 
		c     \int_{V_N}  \left(  \left|\frac{\partial u}{\partial \overline{z}^1} \right|^2 
		+  r^2\left| \frac{\partial u}{\partial \bar\zeta} \right|^2 \right)dz^1 \wedge d\bar z^1 \wedge \frac{d\zeta \wedge d\bar \zeta}{r^2(-\log r^2)^2} \ \   \mbox{(by Cauchy-Schwartz)}.
	\end{eqnarray*}
	By Section~\ref{weneed}~(i), we can apply an analogous argument to (\ref{I0}) to conclude
	\[
	\lim_{N \rightarrow \infty} (III)=0.
	\] 
	Similarly,  
	\[
	\lim_{N \rightarrow \infty} (IV)=0.
	\] 
	We thus conclude
	\[
	\lim_{N \rightarrow \infty} |(I)|+|(II)|+|(III)|+|(IV)|=0.
	\]
	Next,
	\begin{eqnarray*}
		(V) & = &
		\int_V \frac{\eta'(N^{-1} \log b|\zeta|^{-2})}{N}   \partial  \overline{\partial} \log b \wedge \{ \bar\partial u, \bar\partial u \}
		\\   
		& = &
		\int_V \frac{\eta'(N^{-1} \log b|\zeta|^{-2})}{N}   \frac{\partial^2 \log b} { \partial  z^1  \partial \overline{z}^1} dz^1 \wedge d{\bar z}^1 \wedge \left|\frac{\partial u}{\partial \bar \zeta} \right|^2  d \bar\zeta \wedge d{\zeta}\\
		& & \ +
		\int_V \frac{\eta'(N^{-1} \log b|\zeta|^{-2})}{N}   \frac{\partial^2 \log b} { \partial \zeta  \partial \bar\zeta} d\zeta \wedge d{\bar \zeta} \wedge \left|\frac{\partial u}{\partial \bar z^1} \right|^2  d \overline{z}^1 \wedge dz^1\\
		& & \ \ +
		\int_V \frac{\eta'(N^{-1} \log b|\zeta|^{-2})}{N}   \frac{\partial^2 \log b} { \partial z^1  \partial \bar\zeta} dz^1 \wedge d{\bar \zeta} \wedge <\frac{\partial u}{\partial \bar z^1},  \frac{\partial u}{\partial \bar \zeta}>  d \overline{z}^1 \wedge d\zeta\\
		& & \ \ \ +
		\int_V \frac{\eta'(N^{-1} \log b|\zeta|^{-2})}{N}   \frac{\partial^2 \log b} { \partial \zeta  \partial  \overline{z}^1} d \zeta \wedge d \overline{z}^1 \wedge <\frac{\partial u}{\partial \bar \zeta}, \frac{\partial u}{\partial \overline{z}^1}>  d \bar\zeta \wedge d{z}^1\\
		& =: & (V)_1+(V)_2+(V)_3+(V)_4.
	\end{eqnarray*}
	We estimate
	\begin{eqnarray*}
		|(V)_2| 
		& \leq  &  c
		\int_V
		\left|  \frac{\eta'(N^{-1} \log b|\zeta|^{-2})}{N}  \right| \left| \frac{\partial^2 \log b} { \partial  z^1  \partial \overline{z}^1} \right| \left|\frac{\partial u}{\partial \bar z^1} \right|^2 dz^1 \wedge   d \overline{z}^1 \wedge d\zeta \wedge d\bar \zeta\nonumber 
		\\
		& \leq  &  c
		\int_{V_N}
		\left|\frac{\partial u}{\partial \bar z^1} \right|^2 dz^1 \wedge   d \overline{z}^1 \wedge \frac{d \zeta \wedge d\bar \zeta}{(-\log r^2)}\nonumber \\
		&& \left(\mbox{by} \  \frac{\partial^2 \log b} { \partial \zeta  \partial \bar \zeta}=O(1) \ \mbox{ and }  (\ref{cleverbd1})\right)
		\\
		|(V)_3| & \leq & 
		\int_V \left| \frac{\eta'(N^{-1} \log b|\zeta|^{-2})}{N}\right|  \left| \frac{\partial^2 \log b} { \partial z^1  \partial \bar\zeta} \right| \left| \frac{\partial u}{\partial \bar z^1}\right|\left| \frac{\partial u}{\partial \bar \zeta}\right| dz^1 \wedge   d \overline{z}^1 \wedge d{\bar \zeta} \wedge d\zeta \\
		& \leq & c
		\int_{V_N} 
		\frac{1}{(-\log r^2)}\left|\frac{\partial u}{\partial \bar z^1} \right|\left|\frac{\partial u}{\partial \bar \zeta} \right| dz^1 \wedge   d \overline{z}^1 \wedge d{\bar \zeta} \wedge d\zeta \nonumber 
		\\
		&& \left(\mbox{by}  \ \frac{\partial^2 \log b} { \partial \zeta  \partial \bar \zeta}=O(1) \ \mbox{and } \  (\ref{cleverbd1})\right)\\
		& \leq   &  c
		\int_{V_N}
		\left( \left|\frac{\partial u}{\partial \bar z^1} \right| ^2+\frac{1}{(-\log r^2)^2} \left|\frac{\partial u}{\partial \bar \zeta} \right|^2 \right) dz^1 \wedge   d \overline{z}^1 \wedge d\zeta \wedge d\bar \zeta\nonumber 
	\end{eqnarray*}
	and similarly 
	\begin{eqnarray*}
		|(V)_4|  & \leq   &  c
		\int_{V_N}
		\left( \left|\frac{\partial u}{\partial z^1} \right| ^2+\frac{1}{(-\log r^2)^2} \left|\frac{\partial u}{\partial \zeta} \right|^2 \right) dz^1 \wedge   d \overline{z}^1 \wedge d\zeta \wedge d\bar \zeta.   \end{eqnarray*}
	With these estimates, we can argue as in  the proof of (\ref{I0}) to  conclude
	\[
	\lim_{N \rightarrow \infty} (V)_2+(V)_3+(V)_4=0.
	\]
	We are left  to compute
	\[
	(V)_1 =
	\int_V \frac{\eta'(N^{-1} \log b|\zeta|^{-2})}{N}   \frac{\partial^2 \log b} { \partial  z^1  \partial \overline{z}^1} dz^1 \wedge d{\bar z}^1 \wedge \left|\frac{\partial u}{\partial \bar \zeta} \right|^2  d \bar\zeta \wedge d{\zeta}.
	\]   
	First,  use the identity
	\[
	\frac{\partial^2 \log b} {  \partial z^1  \partial \overline{z}^1}(z^1,\zeta)=\frac{\partial^2 \log b} { \partial z^1  \partial \overline{z}^1}(z^1,0)+O(r)
	\]
	to write 
	\[
	(V)_1=(V)_{1a}+(V)_{1b}.
	\]
	We estimate 
	\begin{eqnarray*}
		|(V)_{1b}|& = &   \int_V \left| \frac{\eta'(N^{-1} \log b|\zeta|^{-2})}{N}  \right|  \left|\frac{\partial u}{\partial \bar \zeta} \right|^2O(\zeta) dz^1 \wedge d{\bar z}^1 \wedge   d \zeta \wedge d \bar \zeta\\   
		& \leq  &  c
		\int_{V_N}  
		\frac{r}{(-\log r^2)}\left|\frac{\partial u}{\partial \bar \zeta} \right|^2dz^1 \wedge d{\bar z}^1 \wedge   d \zeta \wedge d \bar \zeta \  \mbox{ (by (\ref{cleverbd1})}).
	\end{eqnarray*}
	Thus, we can argue as in  the proof of (\ref{I0}) to  conclude 
	\[
	\lim_{N \rightarrow \infty} (V)_{1b} =0.
	\]
	Furthermore,
	\begin{eqnarray*}
		(V)_{1a} & = & 
		\int_V \frac{\eta'(N^{-1} \log b|\zeta|^{-2})}{N}   \frac{\partial^2 \log b} { \partial z^1  \partial \overline{z}^1}(z^1,0) dz^1 \wedge d{\bar z}^1 \wedge \left|\frac{\partial u}{\partial \bar \zeta} \right|^2  d \bar\zeta \wedge d{\zeta}
		\\
		& = &
		\int_V \frac{\eta'(N^{-1} \log b|\zeta|^{-2})}{N}   \frac{\partial^2 \log b} { \partial z^1  \partial \overline{z}^1}(z^1,0) dz^1 \wedge d{\bar z}^1 \wedge \left(\left|\frac{\partial u}{\partial \bar \zeta} \right|^2- \frac{L_j^2}{16\pi r^2}\right) d \bar\zeta \wedge d{\zeta}
		\\
		&  &
		\ \ + \  \frac{L_j^2}{4\pi}  \int_V \frac{\eta'(N^{-1} \log b|\zeta|^{-2})}{Nr^2}    \frac{\partial^2 \log b} { \partial z^1  \partial \overline{z}^1}(z^1,0) dz^1 \wedge d{\bar z}^1 \wedge d \bar\zeta \wedge d{\zeta}.
	\end{eqnarray*}
	The first  term on the right hand side above can be estimated by 
	\begin{eqnarray*}
		\lefteqn{
			\left| \int_V \frac{\eta'(N^{-1} \log b|\zeta|^{-2})}{N}    \frac{\partial^2 \log b} { \partial z^1  \partial \overline{z}^1}(z^1,0) \left(\left|\frac{\partial u}{\partial \bar \zeta} \right|^2-\frac{L_j^2}{4\pi r^2}\right) dz^1 \wedge d{\bar z}^1 \wedge  d \bar\zeta \wedge d{\zeta}  \right|
		}
		\\
		& \leq  &  c
		\int_{V_N}  
		\left(\left|\frac{\partial u}{\partial \bar \zeta} \right|^2-\frac{L_j^2}{4\pi r^2}\right) dz^1 \wedge d{\bar z}^1 \wedge   \frac{d \bar\zeta \wedge d{\zeta}}{(-\log r^2)} \  \left(\mbox{by}  \ \frac{\partial^2 \log b} { \partial z^1  \partial \overline{z}^1} =O(1) \ \mbox{and} \  (\ref{cleverbd1})\right).
	\end{eqnarray*} 
	With these estimates, we can argue as in  the proof of (\ref{ito0}) to  conclude \begin{eqnarray*}
		\lim_{N \rightarrow \infty} (V)_{1} & = &  \lim_{N \rightarrow \infty} (V)_{1a} + (V)_{1b} \\
		& = & \frac{L_j^2}{4\pi}  \lim_{N \rightarrow \infty}\int_V \frac{\eta'(N^{-1} \log b|\zeta|^{-2})}{Nr^2}   \frac{\partial^2 \log b} { \partial z^1  \partial \overline{z}^1}(z^1,0)dz^1 \wedge d{\bar z}^1\wedge d \bar\zeta \wedge d{\zeta}\\
		&=&\frac{L_j^2}{4\pi}  \int_{\Omega} \frac{\partial^2 \log b} { \partial z^1  \partial \overline{z}^1}(z^1,0) d{z}^1\wedge d{\bar z}^1   \cdot \lim_{N \rightarrow \infty} \int_{\overline{{\mathbb D}}^*_{\frac{1}{4}}}
		\frac{\eta'(N^{-1} \log b|\zeta|^{-2})}{Nr^2}  d \bar\zeta \wedge d{\zeta}\\
		&=&\frac{L_j^2}{4\pi i} \int_{\Omega} \Theta(\sO_{\overline{X}}(\Sigma_j))
		\cdot  \lim_{N \rightarrow \infty}  \int_0^{\frac{1}{4}}
		\frac{\eta'(-N^{-1} \log br^2)}{Nr}dr\\
		&=&\frac{L_j}{4\pi i} \int_{\Omega} \Theta(\sO_{\overline{X}}(\Sigma_j)).
	\end{eqnarray*}
	In the above $ \Theta(\sO_{\overline{X}}(\Sigma_j))$  denotes the curvature of the hermitian metric $h_j$ on the line bundle $\sO_{\overline{X}}(\Sigma_j)$. The estimates for $(I)$, $(II)$, $(III)$, $(IV)$ and $(V)$ imply that (\ref{ibp}) also holds for $V=\Omega \times {\mathbb D}^*_{\frac{1}{4}}$ away from the crossings.
\end{proof}



\begin{proposition} \label{abc}
	Assume
	\begin{equation} \label{hessianintegrable}
	\int_X |\partial_E \bar \partial u|^2 <\infty.
	\end{equation}
	If  $\{\chi_N\}$ is the sequence of cut-off functions defined in Section~\ref{sec:cutoff}, then
	\[
	\lim_{N \rightarrow \infty}  \int_X d \chi_N \wedge  \{\bar\partial_{}\partial u, \partial u -\overline{\partial} u\}   = 
	0.
	\]
\end{proposition}


\begin{proof}
	Let $V$  be either a set $\Omega \times {\mathbb D}^*_{\frac{1}{4}}$ away from the crossings (cf.~Section~\ref{weneed}~(i)) or a set 
	$\bar {\mathbb D}_{\frac{1}{4}}^*  \times  \bar {\mathbb D}_{\frac{1}{4}}^*$  at  a crossing (cf.~Section~\ref{weneed}~(ii)).  Since the support of $d\chi_N$ is covered by such a set $V$, it is sufficiently to prove
	\begin{equation} \label{limitis0}
		\lim_{N \rightarrow \infty}  \int_V d \chi_N \wedge  \{\bar\partial_{}\partial u, \partial u -\overline{\partial} u\}   = 
		0.
	\end{equation}
	Thus, the rest of the proof is devoted to proving (\ref{limitis0}).  For the sequel, the constant $c>0$ is an arbitrary constant independent of the parameter $N$.  
	First, consider the set
	$V=\bar {\mathbb D}_{\frac{1}{4}}^*  \times  \bar {\mathbb D}_{\frac{1}{4}}^*$ at  a crossing with local holomorphic coordinates $(z^1=\varrho e^{i\phi},z^2=r e^{i\theta})$ (cf.~Section~\ref{weneed}~(ii)).
	We have
	\begin{eqnarray*}
		\partial u - \overline{\partial} u  
		& = &
		\left( \frac{\partial u}{\partial z^1 }dz^1- \frac{\partial u}{\partial \overline{z}^1 }d\overline{z}^1 \right) + \left( \frac{\partial u}{\partial z^2 }dz^2- \frac{\partial u}{\partial \overline{z}^2 }d\overline{z}^2 \right)
		\\
		& = &    i \left(   \frac{\partial u}{\partial \varrho} \varrho d\phi - \frac{\partial u}{\partial \phi} \frac{d\varrho}{\varrho} \right)+ i \left(   \frac{\partial u}{\partial r} r d\theta - \frac{\partial u}{\partial \theta} \frac{dr}{r} \right)
	\end{eqnarray*}
	and
	\begin{eqnarray*}
		d\chi_N & = &  -\eta(-N\log \varrho^2) \frac{\eta'(-N^{-1}\log r^2)}{N}  \frac{2dr}{r}  -\frac{\eta'(-N^{-1}\log \varrho^2)}{N}  \eta(-N\log r^2)  \frac{2d\varrho}{\varrho}.
	\end{eqnarray*}
	Thus,
	\begin{eqnarray} \label{goto00}
		\lefteqn{
			\int_V d \chi_N \wedge  \{\bar\partial_{}\partial u, \partial u -\overline{\partial} u\} }
		\nonumber \\
		& = &   
		- \frac{2}{N} \int_V   \eta(-N^{-1}\log \varrho^2) \eta'(-N^{-1}\log r^2)  \frac{dr}{r} \wedge \{ \partial_{} \bar \partial u, \frac{\partial u}{\partial z^1 }dz^1- \frac{\partial u}{\partial \overline{z}^1 }d\overline{z}^1  \} 
		\nonumber \\
		& & 
		\      - \frac{2i}{N}  \int_V  \eta(-N^{-1}\log \varrho^2) \eta'(-N^{-1}\log r^2) \frac{dr}{r}  \wedge \{ \partial_{} \bar \partial u,   \frac{\partial u}{\partial r} rd\theta \}  \nonumber \\
		& & 
		\ \ \  -\frac{2}{N}  \int_V  \eta'(-N^{-1}\log \varrho^2) \eta(-N^{-1}\log r^2)   \frac{d\varrho}{\varrho} \wedge \{ \partial_{} \bar \partial u,  \frac{\partial u}{\partial z^2 }dz^2- \frac{\partial u}{\partial \overline{z}^2 }d\overline{z}^2 \} 
		\nonumber  \\
		& & \ \ \ \  - \frac{2i}{N}    \int_V  \eta'(-N^{-1}\log \varrho^2) \eta(-N^{-1}\log r^2) \frac{d\varrho}{\varrho}  \wedge \{ \partial_{} \bar \partial u,   \frac{\partial u}{\partial \varrho} \varrho d\phi \}  \\
		& = &(i)+(ii)+(i')+(ii'). \nonumber 
	\end{eqnarray}
	
	We will show that all the terms $(i)$, $(ii)$, $(i')$ and $(ii')$ go to 0 as $N \rightarrow \infty$. We start with $(i)$.
	Note that  $|\eta(-N^{-1}\log \varrho^2)|$ has  support in $\varrho \geqslant e^{-\frac{N}{3}}$ and $|\eta'(-N^{-1}\log r^2)|$ has  support in $ e^{-\frac{N}{3}} \leq r \leq e^{-\frac{N}{4}}$  (cf.~(\ref{sptphi1})). Thus, the integrand of $(i)$ has support in 
	\[
	D_N:=  {\mathbb D}_{z^2,e^{-\frac{N}{3}}, \frac{1}{4}} \times {\mathbb D}_{z^1,e^{-\frac{N}{3}}, e^{-\frac{N}{4}}}.
	\]
	We estimate
	\begin{eqnarray}
		\lefteqn{|(i)|
			\leq 
			c  \int_{D_N} \left|  \frac{ \eta'(-N^{-1}\log r^2)}{N} \right|   |\partial_{} \bar \partial u|  \left| \frac{\partial u}{\partial z^1}\right|  \varrho d\varrho \wedge d\phi \wedge \frac{r dr \wedge d\theta}{r}
		}
		\nonumber   \\
		& \leq &
		c \left(\int_{D_N}  |\partial_{} \bar \partial u|^2  dz^1 \wedge d\bar z^1 \wedge 
		dz^2 \wedge d\bar z^2
		\right)^{\frac{1}{2}} \nonumber
		\\
		& & \ \times 
		\left(  \int_{D_N}  \left( \frac{ \eta'(-N^{-1}\log r^2)}{N} \right)^2 \left| \frac{\partial u}{\partial z^1}\right|^2 dz^1 \wedge d\bar z^1\wedge \frac{dz^2 \wedge d\bar z^2}{r^2}
		\right)^{\frac{1}{2}} \nonumber\\
		&& \ \ \ \ \ \ \mbox{ (by Cauchy-Schwartz and (\ref{cleverbd!}))}. 
		\label{sades}
	\end{eqnarray}
	The first integral above limits to 0 as $N \rightarrow \infty$ by assumption~(\ref{hessianintegrable}), volume estimate (\ref{volcompB}) and  Lebesgue's dominated convergence theorem.  The limit as $N \rightarrow \infty$ of the second integral exists by Lemma~\ref{lemma:prelim} by  following the proof  of (\ref{iii0}). Thus
	$
	\lim_{N \rightarrow \infty} (i) =0.
	$
	An analogous argument shows $
	\lim_{N \rightarrow \infty} (i') =0.
	$

	Next,  \begin{eqnarray*}
		|(ii)| 
		& \leq &    
		c \int_{V}   \left| \frac{ \eta'(-N^{-1}\log r^2)}{N} \right| | \bar\partial_{}\partial u|
		\left| \frac{\partial u}{\partial r} \right|
		dz^1  \wedge d\bar z^1 \wedge
		\frac{dz^2 \wedge d\bar z^2}{r} 
		\\
		& \leq &    
		c \left(  \int_{D_N} 
		| \bar\partial_{}\partial u|^2 
		dz^1  \wedge d\bar z^1 \wedge
		\frac{dz^2 \wedge d\bar z^2}{r^2 (-\log r^2)^2}
		\right)^{\frac{1}{2}}  
		\\
		&  & \ \times
		\left(  \int_{D_N}  
		\left|\frac{\partial u}{\partial r} \right|^2 
		dz^1  \wedge d\bar z^1 \wedge
		dz^2 \wedge d\bar z^2 
		\right)^{\frac{1}{2}}\\
		&  & (\mbox{by Cauchy-Schwartz and (\ref{cleverbd!})}).
	\end{eqnarray*}
	The first integral limits to 0 as $N \rightarrow \infty$ by  assumption~(\ref{hessianintegrable}), volume estimate (\ref{volcompB}) and  Lebesgue's dominated   convergence Theorem.  The second integral also limits to 0 by Section~\ref{weneed}~(ii) and Lebesgue's dominated  convergence theorem. Thus, $
	\lim_{N \rightarrow \infty} (ii) =0$, and an analogous argument shows 
	$\lim_{N \rightarrow \infty} (ii') =0$.

	Next, consider a set $V=\Omega \times {\mathbb D}^*_{\frac{1}{4}}$ away from the crossings   with holomorphic coordinates  $(z^1,\zeta=r e^{i\theta})$.  
	Since
	\[
	\partial u - \overline{\partial} u =  \left( \frac{\partial u}{\partial z^1 }dz^1- \frac{\partial u}{\partial \overline{z}^1 }d\overline{z}^1 \right)+ i \left(   \frac{\partial u}{\partial r} r d\theta - \frac{\partial u}{\partial \theta} \frac{dr}{r} \right),
	\]
	we have
	\begin{eqnarray} \label{goto0}
		\lefteqn{
			\int_X d \chi_N \wedge  \{\bar\partial_{}\partial u, \partial u -\overline{\partial} u\} }
		\\
		& = &   
		i \int_Xd\chi_N \wedge \{ \bar\partial_{}\partial u, \frac{\partial u}{\partial r} r d\theta \} 
		+
		i \int_Xd\chi_N \wedge \{ \bar\partial_{}\partial u,  \frac{\partial u}{\partial \theta}  \frac{dr}{r} \} 
		\nonumber  \\
		& & \ +   \int_Xd\chi_N \wedge \{ \bar\partial_{}\partial u, \frac{\partial u}{\partial z^1 }dz^1- \frac{\partial u}{\partial \overline{z}^1 }d\overline{z}^1 \}  \nonumber \\
		& = & (I)+(II)+(III) \nonumber
	\end{eqnarray}
	where  the integrals $(I)$, $(II)$, and $(III)$ are estimated below. Let
	\[
	G_N:=\Omega \times{\mathbb D}_{z^1,c_1e^{-\frac{N}{3}}, c_2e^{-\frac{N}{4}}}.
	\]
	Since, 
	\begin{eqnarray*}
		d\chi_N 
		& = & -\frac{\eta'(N^{-1}\log br^{-2})}{N} \left( \frac{2dr}{r} - \frac{db}{b} \right), \end{eqnarray*}
	integral $(I)$ is bounded by
	\begin{eqnarray*}
		\left| (I) \right| 
		& = & 
		\left|  \int_V  \frac{\eta'(N^{-1}\log br^{-2})}{N} \left( \frac{2dr}{r} - \frac{db}{b}    \right) \wedge \{ \bar\partial_{}\partial u, \frac{\partial u}{\partial r} r d\theta \}  \right|  
		\\
		& \leq &c \int_V   |\overline{\partial}_{} \partial u| \left| \frac{\partial u}{\partial r} \right| \varrho d\varrho \wedge d\phi \wedge \frac{r dr \wedge d\theta}{r (-\log r^2)}  \ \  \left(\mbox{since} \ \frac{db}{b}=O(1) \right)
		\\
		& \leq  & c \left(  \int_{G_N} |\overline{\partial}_{} \partial u|^2  d z^1 \wedge d\bar z^1 \wedge  \frac{d \zeta \wedge d\bar \zeta}{r^2 (-\log r^2)^2} \right)^{\frac{1}{2}} \\
		& & \ \ \times  \left( \int_{G_N}   \left| \frac{\partial u}{\partial r} \right|^2 
		d z^1 \wedge d\bar z^1 \wedge d\zeta \wedge d\bar \zeta  \right)^{\frac{1}{2}} \ \  \mbox{(by Cauchy-Schwartz)}.
	\end{eqnarray*}
	The first integral limits to 0 by assumption~(\ref{hessianintegrable}), volume estimate (\ref{volcomp}) and Lebesgue's dominated  convergence theorem.  The second integral also limits to 0 by Section~\ref{weneed}~(i) (with $s=r$) and  Lebesgue's dominated  convergence theorem. Thus, $
	\lim_{N \rightarrow \infty} (I) =0$.
	
	Next, we estimate $(II)$.  {\it This is the term for which  the modified Siu's  Bochner formula is crucial.}  Indeed, we hightlight the cancellation $\frac{dr}{r} \wedge \frac{dr}{r}=0$ below:
	\begin{eqnarray*}
		|(II)| 
		& =& 
		\left|  \int_V \frac{\eta'(N^{-1}\log br^{-2})}{N} \left( \frac{2dr}{r} - \frac{db}{b}  \right) \wedge \{ \bar\partial_{}\partial u, \frac{\partial u}{\partial \theta} \frac{dr}{r} \}  \right|
		\\
		& \leq &c \int_V  \left| \frac{\eta'(N^{-1}\log br^{-2})}{N} \right| |\overline{\partial}_{} \partial u| \left| \frac{\partial u}{\partial \theta} \right|  \varrho d\varrho \wedge d\phi \wedge \frac{r dr \wedge d\theta}{r} \\
		& & \ \ \  \ \left(\mbox{since} \ \frac{db}{b}=O(1)  \mbox{ and } \frac{dr}{r} \wedge \frac{dr}{r}=0\right)
		\\   & \leq & 
		c \left( \int_{G_N}   | \bar\partial_{}\partial u|^2  dz^1 \wedge d\bar z^1 \wedge d\zeta \wedge d\bar \zeta \right)^{\frac{1}{2}}
		\\
		& & \ \ \times \left(\int_{G_N}    \left| \frac{\partial u}{\partial \theta} \right|^2 dz^1 \wedge d\bar z^1 \wedge \frac{d\zeta \wedge d\bar \zeta}{r^2 (-\log r^2)^2} \right)^{\frac{1}{2}}  \   \mbox{ (by Cauchy-Schwartz)}.
	\end{eqnarray*}
	The first integral limits to 0 by assumption~(\ref{hessianintegrable}), volume estimate (\ref{volcomp}) and  Lebesgue's dominated  convergence theorem.  The second integral also limits to 0 by Section~\ref{weneed}~(i) (with $r=s$ and $\theta=\eta$) and Lebesgue's dominated  convergence theorem. Thus, $
	\lim_{N \rightarrow \infty} (II) =0$.
	
	Finally, 
	\begin{eqnarray*}
		|(III)| 
		& = & 
		\left|  \int_V  \frac{\eta'(N^{-1}\log br^{-2})}{N} \left( \frac{2dr}{r} - \frac{db}{b} \right) \wedge
		\{ \bar\partial_{}\partial u, \frac{\partial u}{\partial z^1 }dz^{1}- \frac{\partial u}{\partial \overline{z}^1 }d\overline{z}^1 \}  \right|
		\\
		& \leq & 
		c\int_{G_N}  | \bar\partial_{}\partial u| \left|  \frac{\partial u}{\partial z^1}\right| dz^1 \wedge d\bar z^1 \wedge \frac{d\zeta \wedge d\bar \zeta}{r (-\log r^2)} \\
		&&  
		\ \ \  \left(\mbox{since} \ \frac{db}{b}=O(1) \ \mbox{and by} \ (\ref{cleverbd1}) \right)
		\\  & \leq & 
		c \left( \int_{G_N}   | \bar\partial_{}\partial u|^2 dz^1 \wedge d\bar z^1 \wedge d\zeta \wedge d\bar \zeta \right)^{\frac{1}{2}} 
		\\
		& & \ \times \left( \int_{G_N} \left|  \frac{\partial u}{\partial z^1}\right|^2   dz^1 \wedge d\bar z^1 \wedge \frac{d\zeta \wedge d\bar \zeta}{r^2 (-\log r^2)^2} \right)^{\frac{1}{2}} \ \ \ \  \mbox{ (by Cauchy-Schwartz)}.
	\end{eqnarray*}
	The first integral limits to 0 by  assumption~(\ref{hessianintegrable}), volume estimate (\ref{volcomp}) and  Lebesgue's dominated  convergence theorem.  The second integral also limits to 0 by Section~\ref{weneed}~(i) (with $r=s$) and  Lebesgue's dominated  convergence theorem. Thus, $
	\lim_{N \rightarrow \infty} (III) =0$.

	We now conclude that  (\ref{goto0}) $\rightarrow 0$ as $N \rightarrow \infty$, which combined with the fact that (\ref{goto00}) $\rightarrow 0$ as $N \rightarrow \infty$
	implies (\ref{limitis0}).  This concludes the proof of Lemma~\ref{abc}.
\end{proof}

\subsection{Proof of Theorem~\ref{thm:pu} (I)} \label{sec:pu}
In this section, we let $Y$ be a Riemannian manifold with strongly nonpositive curvature.

\begin{lem} \label{integrable}
	Assume that the harmonic map $\tilde{u}$ of Theorem~\ref{thm:pu} maps into a Riemannian manifold $M$  with strongly nonpositive curvature.  Then
	\[
	\int_X \left|\partial_{E} \bar \partial u \right|^2  \omega^2 < \infty.
	\]
\end{lem}

\begin{proof}
	The  Siu-Sampson's Bochner formula  (cf.~\cite{Sam85}) is 
	\begin{eqnarray}
		\partial  \bar \partial \{\bar \partial  u, \bar \partial u \}
		& = &    2\left(\left|\partial_{E} \bar \partial u \right|^2 +Q_0\right) \omega^2 \label{sampson} 
	\end{eqnarray}
	where
	\begin{equation} \label{qnot}
		Q_0=-2g^{\alpha \bar \delta} g^{\gamma \bar \beta} R_{ijkl}
		\frac{\partial u^i}{\partial z^\alpha}  \frac{\partial u^k}{\partial \bar z^\beta}  \frac{\partial u^j}{\partial z^\gamma}\frac{\partial u^l}{\partial \bar z^\delta} \geq 0.
	\end{equation}
	In the expression for $Q_0$, we use local coordinates $(z^\alpha)$ of $X$ and $(y^i)$ of $Y$.     If $Y=\Delta(G)$ is a building, then (\ref{sampson}) is valid for $x \in \cR(u)$ with $Q_0=0$. 
	Multiply by $\chi_N$, integrate it over $X$, and 
	apply integration by parts to conclude
	\[
	2 \int_X\left( \left|\partial_{E} \bar \partial u \right|^2 +Q_0 \right) \chi_N \omega^2  =
	\int_X    \partial  \bar \partial \{\bar \partial u, \bar \partial u \} \chi_N
	\\
	=
	\int_X  \{\bar \partial u, \bar \partial u \}  \wedge \partial  \bar \partial \chi_N.
	\]
	The limit of the right hand side above as $N \rightarrow \infty$ is bounded by  Proposition~\ref{integrability}.  This proves Lemma~\ref{integrable}.
\end{proof}

 We are now in position to finish the proof of Theorem~\ref{thm:pu} when $Y=M$ is a Riemannian manifold of strongly nonpositive curvature. To do so, we need the following variation of the Siu-Sampson-Mochizuki Bochner formula  for a harmonic map $u: \widetilde X \rightarrow M$ found in \cite{DMIn}: 
\[
\left(4\left|{\partial}_{E} \bar \partial u\right|^2+Q_0   \right)  \omega^2 = d \{\bar \partial_{E}  \partial u,  \bar \partial u -  \partial u\}.
\]
where $Q_0$ as in (\ref{qnot}). 
By Lemma~\ref{integrable}, we can  integrate the above equality and apply Proposition~\ref{abc}.  Thus, we obtain
\begin{eqnarray*}
	\int_X \left(4\left|{\partial}_{E} \bar \partial u\right|^2+Q_0   \right)  \omega^2  & = & \int_X d  \{ \bar\partial \partial u, \partial u -\overline{\partial} u\} \\
	& = & \lim_{N \rightarrow \infty} \int_X \chi_N d  \{ \bar\partial \partial u, \partial u -\overline{\partial} u\}
	\\
	& = & -\lim_{N \rightarrow \infty} \int_X d\chi_N \wedge  \{ \bar\partial \partial u, \partial u -\overline{\partial} u\}
	\\
	& = & 0.
\end{eqnarray*}
Since $Q_0\geq 0$ by assumption, 
$Q_0=|\partial_{E} \bar  \partial u|=0$.  Thus, we conclude $\partial_{E} \bar  \partial u=0$;  in other words, $u$ is pluriharmonic.

\subsection{Proof of Theorem~\ref{thm:pu} (II)} \label{sec:bddm}

In this section, we let $Y$ be a  Euclidean building $\Delta(G)$.  Unlike Section~\ref{sec:pu}, special care must be taken because of the presence of the singular set.  
\begin{lem} \label{Bintegrable'}
	For  $\chi_N:X \rightarrow [0,1]$ as in Section~\ref{sec:cutoff},
	\[
	\int_X \partial \bar \partial \{ \bar \partial u, \bar \partial u\}  \chi_N 
	= \int_X   \{ \bar \partial u, \bar \partial u\} \wedge \bar \partial \partial \chi_N.
	\]
\end{lem}

\begin{proof}
	Let $\Omega_1$ be the support of $\chi_N$ which is relatively compact.	With $\psi_i$ defined as in Theorem~\ref{gs}, we have
	\begin{eqnarray*}
		\lefteqn{ \int_X \partial \bar \partial \{ \bar \partial u, \bar \partial u\}  \chi_N
			\psi_i }
		\\
		& = & \int_X \bar \partial \{ \bar \partial u, \bar \partial u\} \wedge \partial (\chi_N
		\psi_i)
		\\
		& = & \int_X (\bar \partial \{ \bar \partial u, \bar \partial u\} \wedge \partial \chi_N
		)  \psi_i + \int_X (\bar \partial \{ \bar \partial u, \bar \partial u\} \wedge \partial \psi_i)  \chi_N
		\\
		& = & -\int_X  \left( \{ \bar \partial u, \bar \partial u\} \wedge \bar \partial \partial \chi_N
		\right)  \psi_i + \int_X \{ \bar \partial u, \bar \partial u\} \wedge \partial \chi_N
		\wedge \bar \partial \psi_i + \int_X \left( \bar \partial \{ \bar \partial u, \bar \partial u\} \wedge \partial \psi_i  \right) \chi_N
		.
	\end{eqnarray*}
	Furthermore, there exists a constant $C>0$ depending only on the Lipschitz constant of  $\chi_N$   such that 
	\begin{eqnarray*}
		\left|  \int_X \{ \bar \partial u, \bar \partial u\} \wedge \partial \chi_N \wedge \bar \partial \psi_i \right| & \leq & C\int_{\Omega_1} |\nabla u|^2 |\nabla \psi_i| \omega^2,
	\end{eqnarray*}
	
	\begin{eqnarray*}
		\left| \int_X \left( \bar \partial \{ \bar \partial u, \bar \partial u\} \wedge \partial \psi_i  \right) \chi_N \right| \leq  C\int_{\Omega_1} |\nabla \nabla u| |\nabla \psi_i| \omega^2.
	\end{eqnarray*} 
	Thus, the assertion follows from letting $i \rightarrow \infty$ and applying Theorem~\ref{gs}.
\end{proof}

\begin{lem} \label{integrable'}
	For the harmonic map $\tilde{u}$ of Theorem~\ref{thm:pu}, 
	\[
	\int_X \left|\partial \bar \partial u \right|^2  \omega^2 < \infty.
	\]
\end{lem}

\begin{proof}
	The  Siu-Sampson's Bochner formula  (cf.~\cite{Sam85}) is simply
	\begin{eqnarray*}
	    2\left|\partial \bar \partial u \right|^2  \omega^2 = 	\partial  \bar \partial \{\bar \partial  u, \bar \partial u \}. 
	\end{eqnarray*}
		Multiply by $\chi_N$, integrate it over $X$, and 
	apply Lemma~\ref{Bintegrable'} to conclude
	\[
	2 \int_X\left|\partial \bar \partial u \right|^2 \chi_N \omega^2  =
	\int_X    \partial  \bar \partial \{\bar \partial u, \bar \partial u \} \chi_N
	\\
	=
	\int_X  \{\bar \partial u, \bar \partial u \}  \wedge \partial  \bar \partial \chi_N.
	\]
	The limit of the right hand side above as $N \rightarrow \infty$ is bounded by  Proposition~\ref{integrability}.  This proves Lemma~\ref{integrable}.
\end{proof}

\begin{lem} \label{Bpluri'}
	For  $\chi_N:X \rightarrow [0,1]$ as in Section~\ref{sec:cutoff}, 
	\[
	- \int_X \chi_N d  \{ \bar\partial \partial u, \partial u -\overline{\partial} u\} = 
	\int_X d \chi_N \wedge  \{ \bar\partial \partial u, \partial u -\overline{\partial} u\}.
	\]
\end{lem}

\begin{proof}
	Let $\Omega_1$ be the support of $\chi_N$ which is relatively compact.	With $\psi_i$ defined as in Theorem~\ref{gs}, we have
	\begin{eqnarray*}
		- \int_X \chi_N \psi_i d  \{ \bar\partial \partial u, \partial u -\overline{\partial} u\} 
		& = &  
		\int_X \psi_i d \chi_N \wedge  \{ \bar\partial \partial u, \partial u -\overline{\partial} u\} +  \int_X \chi_N d\psi_i  \wedge  \{ \bar\partial \partial u, \partial u -\overline{\partial} u\}.
	\end{eqnarray*}
	Thus,  there exists a constant $C>0$ depending only on the Lipschitz constant of $u$  in the support of $\chi_N$ such that
	\[
	\left| \int_X \chi_N d\psi_i  \wedge  \{ \bar\partial \partial u, \partial u -\overline{\partial} u\} \right| \leq  C \int_{\Omega_1} |\nabla \nabla u| \, |\nabla \psi_i|.
	\]
  The assertion follows from letting $i \rightarrow \infty$ and applying Theorem~\ref{gs}.
\end{proof}

 We are now in position to finish the proof of Theorem~\ref{thm:pu} when $Y=\Delta(G)$ is a Euclidean building.
The  Siu-Sampson-Mochizuki Bochner formula  in this case is simply 
\[
	4\left|{\partial}\bar \partial u\right|^2 \omega^2 = d \{\bar \partial \partial u,  \bar \partial u -  \partial u\}
\]
which holds for the harmonic map $u:X \rightarrow \Delta(G)$ in the regular set $\mathcal R(u)$.  By Lemma~\ref{integrable'}, we can integrate this formula  to conclude
\begin{eqnarray*}
	4\int_X  \left|{\partial} \bar \partial u\right|^2 \omega^2 & = & \int_X d  \{ \bar\partial \partial u, \partial u -\overline{\partial} u\} \\
	& = & \lim_{N \rightarrow \infty} \int_X \chi_N d  \{ \bar\partial \partial u, \partial u -\overline{\partial} u\}
	\\
	& = & -\lim_{N \rightarrow \infty} \int_X d\chi_N \wedge  \{ \bar\partial \partial u, \partial u -\overline{\partial} u\}
	\\
	& = & 0.
\end{eqnarray*}
Here the  third equality follows from \cref{Bpluri'} and the  last equality   is due to Lemma~\ref{integrable'} and Proposition~\ref{abc}.    From this, we conclude that 
${\partial} \bar \partial u=0$ a.e.~on the regular set $\mathcal R(u)$ of $u$.  

To show that $u$ is smooth near every point $p \in \mathcal R(u)$, let   $\Omega \subset \mathcal R(u)$ be a neighborhood of $p$ such that $u$ maps $\Omega$ into an apartment $A \simeq \bR^N$ of $\Delta(G)$ and let $\phi \in C^{\infty}_c(\Omega)$.  For a sequence $\{\psi_i\}$ as in Theorem~\ref{gs}, we have
\[
\lim_{i \rightarrow \infty} \int_\Omega \phi  \ \partial \psi_i \wedge \bar \partial u  \, \omega =0
\]
and thus
\begin{eqnarray*}
	0 = \lim_{i \rightarrow \infty} \int_\Omega (\phi \psi_i) \, \partial \bar \partial u \ \omega 
	= - \lim_{i \rightarrow \infty} \int_\Omega (\phi \partial  \psi_i + \psi_i \partial \phi)\, \wedge \bar \partial u \ \omega
	= -  \int_\Omega   \partial \phi\, \wedge \bar \partial u \ \omega.
\end{eqnarray*}
In other words, $\partial \bar \partial u=0$ weakly in $\Omega$ which implies $u \in C^{\infty}(\Omega)$.  Thus, we have shown $u$ is a smooth map and $\partial \bar \partial u=0$ in $\mathcal R(u)$.  We can now apply Lemma~\ref{phequiv} to conclude that $u$ is a pluriharmonic map in the sense of Definition~\ref{def:pluriharmonic}.


\end{document}